\documentclass[reqno,11pt]{amsart}

\usepackage{amssymb}
\usepackage[all,cmtip]{xy}
\usepackage{leftidx}
\usepackage{comment}
\usepackage[textsize=tiny]{todonotes}
\usepackage{tikz-cd}
\newtheorem{Proposition}{Proposition}[section]
\newtheorem{Corollary}[Proposition]{Corollary}
\newtheorem{Definition}[Proposition]{Definition}
\newtheorem{Lemma}[Proposition]{Lemma}
\newtheorem{Theorem}[Proposition]{Theorem}
\newcounter{mt}

\newtheorem{MainTheorem}[mt]{Theorem}
\newtheorem{Remark}[Proposition]{Remark}
\newcounter{ct}

\DeclareMathOperator{\Val}{Val}
\DeclareMathOperator{\nc}{nc}
\DeclareMathOperator{\vol}{vol}
\DeclareMathOperator{\Kl}{Kl}

\DeclareMathOperator{\Gr}{Gr}

\DeclareMathOperator{\Span}{Span}
\DeclareMathOperator{\Dens}{Dens}

\DeclareMathOperator{\ori}{or}
\DeclareMathOperator{\sgn}{sgn}

\DeclareMathOperator{\Ad}{Ad}
\DeclareMathOperator{\id}{id}
\DeclareMathOperator{\ad}{ad}

\DeclareMathOperator{\Deg}{Deg}

\DeclareMathOperator{\ev}{ev}
\newcommand{\C}{\mathbb{C}}
\newcommand{\p}{\mathbb{P}}
\newcommand{\R}{\mathbb{R}}
\newcommand{\RR}{\mathbb{R}}

\newcommand{\calA}{\mathcal{A}}

\newcommand{\calC}{\mathcal{C}}
\newcommand{\calD}{\mathcal{D}}

\newcommand{\calG}{\mathcal{G}}
\newcommand{\calH}{\mathcal{H}}

\newcommand{\calT}{\mathcal{T}}

\newcommand{\calV}{\mathcal{V}}
\newcommand{\calZ}{\mathcal{Z}}
\newcommand{\TC}{\mathcal{TC}}
\newcommand{\DTC}{\mathcal{DTC}}

\newcommand{\GV}{\mathbf{GV}} %graded vector space
\newcommand{\DGV}{\mathbf{DGV}} %differential graded vector space
\newcommand{\DBGV}{\mathbf{DBGV}} %differential bigraded vector space
\newcommand{\GLA}{\mathbf{GLA}} %graded Lie algebra
\newcommand{\DGLA}{\mathbf{DGLA}} %differential graded Lie algebra
\newcommand{\DBGLA}{\mathbf{DBGLA}} %differential bigraded Lie algebra
\newcommand{\DGGA}{\mathbf{DGGA}} %differential graded Gerstenhaber algebra
\newcommand{\GUAA}{\mathbf{GUAA}} %graded unitary associative algebra

\newcommand{\MC}{\mathrm{MC}}

\newcommand{\g}{\mathfrak g}
\newcommand{\h}{\mathfrak h}

  \newcommand{\largewedge}{\mbox{\Large $\wedge$}}

\makeatletter
\newcommand{\exterior}[1]{\mathop{\mathpalette\exterior@{#1}}}
\newcommand{\exterior@}[2]{%
	\raisebox{\depth}{%
		\fontsize{\sf@size}{1}%
		\m@th
		$\ifx#1\displaystyle\textstyle\else#1\fi\largewedge$}%
	^{\mspace{-2mu}#2}%
	\kern-\scriptspace
}
\makeatother

\usepackage{amsmath}
\DeclareFontFamily{U}{mathx}{}
\DeclareFontShape{U}{mathx}{m}{n}{<-> mathx10}{}
\DeclareSymbolFont{mathx}{U}{mathx}{m}{n}
\DeclareMathAccent{\widecheck}{0}{mathx}{"71}

\allowdisplaybreaks

\title{The volume of tubes in Lie groups}
\author[Andreas Bernig]{Andreas Bernig}
\author[Dmitry Faifman]{Dmitry Faifman}
\author[Jan Kotrbat{\'y}]{Jan Kotrbat{\'y}}

\address{Institut f\"ur Mathematik, Goethe-Universit\"at Frankfurt, Robert-Mayer-Str. 10, 60054 Frankfurt, Germany}
\email{bernig@math.uni-frankfurt.de}

\address{D\'epartement de Math\'ematiques et de Statistique, Universit\'e de
	Montr\'eal, CP 6128 succ Centre-Ville, Montr\'eal, QC H3C 3J7, Canada}
\email{dmitry.faifman@umontreal.ca}

\address{Charles University, Faculty of Mathematics and Physics, Sokolovsk\'a 49/83, 186 00 Prague, Czechia}
\email{kotrbaty@karlin.mff.cuni.cz}

\thanks{A.B. was supported by DFG grant BE 2484/10-1.}
\thanks{D.F. was supported by an NSERC Discovery grant and ISF grant No. 1750/20}
\thanks{J.K. was supported by CU grants PRIMUS/24/SCI/009 and UNCE/24/SCI/022.}

\begin{document}

	\begin{abstract} The problem of computing the volume of tubes in riemannian manifolds goes back to Weyl and Hotelling. Here we find explicit Taylor series for the volume of a tube in a Lie group equipped with a bi-invariant metric. The coefficients are smooth valuations, given by the convolution powers of the surface area valuation. 
		 We show that the tube coefficients can be naturally described as the unique valuations given by universal formulas through the formalism of differential graded Lie and Gerstenhaber algebras; in fact, they are generated by the gauge action on the Maurer--Cartan cone in the free differential graded Lie algebra on one generator. Moreover, we introduce a new convolution product on the corresponding free Gerstenhaber algebra which is compatible with the convolution of valuations and differential forms.

	To complete the picture, we show that a Lie group---not necessarily connected---admits a smooth bi-invariant valuation, beyond the Euler characteristic and the Haar measure, if and only if it admits a bi-invariant riemannian metric.
	\end{abstract}

	\maketitle 
	
	\tableofcontents

	\section{Introduction and main results}

In 1939,  extending previous work of Hotelling \cite{hotelling}, Weyl  \cite{weyl_tubes} computed the volume of the tube ($\epsilon$-neighbourhood) about a compact $n$-dimensional submanifold $Z\subset\R^N$: For $\epsilon\ll1$, 
\begin{equation}\label{eq:weyl_tube} \vol (Z_\epsilon)= \sum_{k=0}^n \omega_{N-k}\mu_k(Z)\epsilon^{N-k},\end{equation}
where $\omega_j$ is the volume of the unit ball  in $\R^j$. Remarkably, the coefficients $\mu_k(Z)$ only depend on the metric induced on $Z$, and are aptly named the \emph{intrinsic volumes} of $Z$. They can be expressed as integrals over $Z$ of the Lipschitz--Killing curvatures, which are certain  polynomials in the entries of the Riemann curvature tensor.

Computing tube formulas in more general settings proved to be challenging, with few cases settled to date. In real space forms, tube formulas were worked out in works of Allendoerfer \cite{allendoerfer} and Santal\'o \cite{santalo}. Complex space forms were treated by Gray in \cite{gray_submanifolds}. Finally, explicit tube formulas for all rank one symmetric spaces were computed by Gray--Vanhecke \cite{gray_vanhecke}. For a comprehensive overview of these results, see Gray's book \cite{gray_tubes}. More recently, Alesker's theory of valuations was used to establish general kinematic formulas on complex space forms by the first named author, Fu, and Solanes  \cite{bernig_fu_solanes}, on exceptional spheres by Solanes--Wannerer \cite{solanes_wannerer}, and  to further study tube formulas by Solanes--Trillo \cite{solanes_trillo}.

Valuations on smooth manifolds were introduced by Alesker in the early 2000s, as a natural extension to smooth manifolds of the notion of valuation in convex geometry \cite{alesker_val_man2,alesker_survey07, alesker_val_man4, alesker_val_man3}. A smooth valuation on an $n$-dimensional manifold $M$ is a finitely-additive measure on a family of test bodies, e.g. the compact submanifolds with corners $\mathcal P(M)$, given by

\begin{displaymath}
	Z\mapsto \int_Z \mu +\int _{\nc(Z)}\omega,
\end{displaymath} 
where $\mu$ is a smooth measure on $M$, $\omega$ an $(n-1)$-form on the sphere bundle of $M$, and $\nc(Z)$ is the normal cycle of $Z$.

The space of smooth valuations on $M$ is denoted $\mathcal V^\infty(M)$. The intrinsic volumes are examples of smooth valuations on $\RR^n$. More generally, for any riemannian manifold $M$, the \emph{tube coefficients}
\[\upsilon_k(Z)=\frac{1}{k!}\left.\frac{d^k}{d\epsilon^k}\right|_{0}\vol(Z_\epsilon),\quad Z\in\mathcal P(M)\]
are elements of $\calV^\infty(M)$,  as was shown in a recent paper by Solanes and Trillo \cite{solanes_trillo}.

Alesker has introduced a natural product on $\calV^\infty(M)$, which can be informally described as follows: If $A, B\in\mathcal P(M)$, and $\chi$ denotes the Euler characteristic, then $\chi(A\cap \bullet)\cdot \chi(B\cap \bullet)=\chi(A\cap B\cap \bullet)$. While $\chi(A\cap \bullet)$ is not a smooth valuation, all smooth valuations can be generated by taking integrals of such valuations over measured families of subsets $A$ \cite{faifman_hofstaetter}.

It was observed by Alesker that Weyl's result implies the existence on a riemannian manifold $(M,g)$ of a distinguished collection of smooth valuations $\mu_j$, $0\leq j\leq n=\dim M$, still called intrinsic volumes, which commute with restrictions under isometric embeddings. They can be defined as follows:
For any compact submanifold with corners $Z\subset M$, and for any isometric embedding of $M$ into euclidean space $\R^N$, \eqref{eq:weyl_tube} holds for $\epsilon\ll1$.

The so-called Lipschitz--Killing space $\mathcal {LK}(M)\subset\mathcal V^\infty(M)$ spanned by the intrinsic volumes is a subalgebra of $\mathcal V^\infty(M)$ under the Alesker product, in fact $\mu_1^j=c_j \mu_j$ with $c_j$ a universal coefficient.
Furthermore, Fu--Wannerer \cite{fu_wannerer} have shown that $\mathcal {LK}(M)$ consists of exactly those valuations that are given by universal formulas in terms of the riemannian structure, more precisely using the Cartan apparatus of solder/connection/curvature forms.
Similar distinguished subspaces of valuations were also found in K\"ahler manifolds \cite{bernig_fu_solanes_wannerer}, pseudo-riemannian manifolds \cite{bernig_faifman_solanes_pseudo} and contact manifolds \cite{faifman_heisenberg}.

In this paper we prove a Lie group counterpart of these results. Let $G$ be a Lie group. The subspace $\mathcal V^\infty(G)^{G\times G}\subset \mathcal V^\infty(G)$ of bi-invariant valuations admits a natural convolution product, introduced by Alesker and the first named author in the compact case and extended to the general case by the authors in \cite{bernig_faifman_kotrbaty_part1}.   
Assume $G$ is equipped with a bi-invariant metric. We will show that the tube coefficients in $G$ are given by universal formulas in terms of the Lie structure and the metric, and span a convolution subalgebra of $\mathcal V^\infty(G)^{G\times G}$.

While not every Lie group admits a bi-invariant metric, this is in fact the case for any Lie group admitting interesting smooth bi-invariant valuations, as our first theorem shows.
 
\begin{MainTheorem}\label{mthm:existence}
	A  Lie group $G$ of dimension at least 2 admits a smooth bi-invariant valuation outside of $\Span\{\chi, \vol_G\}$ if and only if $G$ admits a bi-invariant riemannian metric.
\end{MainTheorem}
In \cite[Theorem B]{bernig_faifman_kotrbaty_part1} we proved this under the additional assumption that $G$ is connected. The general case established here relies on some deep results from group theory.

Next we look for an abstract space that will act as a space of universal formulas, generating valuations on arbitrary Lie groups equipped with a bi-invariant metric. Consider the free differential graded Lie algebra $\calA$ on one generator $\xi$, namely, $\mathcal A=\langle \xi, \gamma\rangle$ with $\deg \xi =0$,  and $\gamma=d\xi$. The corresponding graded exterior algebra $\calG=\largewedge \mathcal A$ is the free differential graded Gerstenhaber algebra on one generator $\xi$. It comes equipped with two gradings, the degree $\deg$ inherited from $\calA$, and the wedge degree $\deg_w$, which we combine into $\Deg x=\deg x-\deg_w x$. It also has two differential operators: $d$ which is inherited from $\calA$, and  the Koszul boundary operator $\partial$, which we combine into $Dx=dx+(-1)^{\Deg x+1}\partial x$. Finally, there are two products on $\calG$: the Schouten--Nijenhuis extension of the Lie bracket on $\calA$, and the wedge product. We will later show that these two products can also be naturally combined into one convolution product rendering $(\calG, \ast, \Deg, D)$ a differential graded algebra.

For any Lie algebra $\mathfrak g$ we have a realization map $\Re_\g:\calG\to \Omega_\g:=\Omega(\g, \largewedge\g)$, mapping $\xi$ to the identity map $\id \in C^\infty(\g, \g)$, and $\gamma$ to the tautological $1$-form $d(\id)\in\Omega^1(\g, \g)$. Later $\g$ will be the Lie algebra of a Lie group equipped with a bi-invariant metric, and we will be considering a closely related spherical realization map $\Re_\g^S$ mapping $\calG$ to real-valued forms on the sphere bundle of $G$. 

 We define the space of abstract tube coefficients \[\calT\calC= \{\omega\in \calG: \Deg \omega=-1, D\omega\in\calH\},\]
where $\calH$ is spanned by all wedge products not having $\xi$ as a factor. The degree restriction simply translates under the realization map to the requirement that a tube coefficient is an $(n-1)$-form on the sphere bundle, while the requirement $D\omega\in\calH$ effectively means that its de Rham differential is a vertical form (see definitions in Section \ref{sec:tube_coefficients}); up to a measure, any smooth valuation can be given by such an $(n-1)$-form. 
We define also the closely associated space $\calD\calT\calC=\{\tau\in \calH: \Deg \tau=0, D\tau=0\}$. While a-priori it is only clear that $D(\calT\calC)\subset \calD\calT\calC$, we will see that in fact $\calD\calT\calC$ coincides with $D(\calT\calC)$ up to constants.

In \cite{bernig_faifman_kotrbaty_part1} we defined a convolution product on $\Omega_{\ori}(SG)^{G\times G}\otimes \Dens(\g^*)$, where $\Dens(\g^*)$ is the one-dimensional space of dual densities on the Lie algebra of $G$, which corresponds to a convolution product on $\Omega(\mathbb P_+(\g^*), \largewedge\g)^{G\times G}$. One can similarly define a convolution on $\Omega_\g=\Omega(\g, \largewedge\g)$, see Proposition \ref{def_convolution_of_forms}. We prove that the Gerstenhaber algebra $\calG$ admits a convolution product which corresponds to those various convolutions under the respective realization maps.

\begin{MainTheorem} \label{mainthm_conv_dgga}
	
	There exists a unique convolution product $*$ on $\calG$ such that for every Lie algebra $\g$, the realization map $\Re:(\calG,\ast)\to (\Omega_\g,\ast)$ is a morphism of algebras. It is associative, satisfies $\Deg(\tau\ast \zeta)=\Deg(\tau)+\Deg(\zeta)$ and the Leibniz rule
	\begin{displaymath}
		D(\tau * \zeta)=D\tau * \zeta+(-1)^{\Deg \tau} \tau * D\zeta.
	\end{displaymath}
Moreover, $\calD\calT\calC \subset (\calG,*)$ is a convolution subalgebra.
\end{MainTheorem}

We then find an explicit basis for $\calT\calC$, exploiting the formalism of Maurer--Cartan elements and the gauge action in the differential graded Lie algebra $\calA$.
Define 
\begin{align*}
	\zeta_\epsilon  & :=e^{\epsilon \xi} * 0 =  - \sum_{i=0}^\infty \frac{\epsilon^{i+1}}{(i+1)!} (\ad \xi)^i \gamma\in \calA[[\epsilon]],\\
	 \rho_\epsilon & = \sum_{m=0}^\infty \varrho_m \epsilon^m :=\widehat\exp(-\zeta_\epsilon)\in \calG[[\epsilon]],
\end{align*}
where the last exponent is in the sense of formal power series using the wedge product.

\begin{MainTheorem} \label{mainthm_expl_descr_va}
The space $\mathcal{TC}$ of abstract tube coefficients is given by $\xi \wedge \mathcal{DTC}$. $\mathcal {DTC}$ is spanned by $\varrho_m$, $m\geq 0$. Furthermore,  $\varrho_m=\frac{1}{m!}\varrho_1^{\ast m}$. 
\end{MainTheorem}

Theorems \ref{mainthm_conv_dgga} and \ref{mainthm_expl_descr_va} together tell us that on every Lie group $G$ equipped with a bi-invariant metric, there is a distinguished subalgebra $\calT\calC(G)$ of the convolution algebra of valuations on $G$ that are constructed using an appropriate realization map from the algebra $\calT\calC$ of abstract tube coefficients.  Thus, for each such $G$ we have a natural homomorphism of algebras $\C[t] \to \left(\mathcal V^\infty(G)^{G \times G} \otimes \Dens(\g^*),*\right)$. This exhibits some formal resemblance with the valuation-theoretic version of Weyl's principle \cite{fu_wannerer}, which says that for each riemannian manifold, there is a natural homomorphism of algebras $\C[t] \to \left(\mathcal V^\infty(M),\cdot\right)$, mapping $t^k$ to (a certain normalization of) the intrinsic volume $\mu_k$. In this case, the formal model of differential forms on the sphere bundle is based on riemannian data (connection and curvature forms), while in our case it is based on Lie algebraic data, combined with the choice of a bi-invariant metric.

	We show that, when $G$ is semisimple, the convolution algebra $\mathcal{TC}(G)$ is infinite-dimensional, unless its Lie algebra is $\mathfrak{so}(3,\R)$. In particular, $\mathcal{TC}(G)$ is in general distinct from the Lipschitz--Killing subspace.   An important exception is the case $G=\R^n$, where the two spaces are not only isomorphic as vector spaces, but, through the Alesker--Fourier transform  \cite{alesker_fourier,faifman_wannerer_fourier}, as algebras.

Finally we elucidate the geometric meaning of $\calT\calC$. Consider a Lie group $G$ equipped with a bi-invariant metric. First we have a map $\Re_\g^S$ that to an element $\omega \in \calT\calC$ associates an $(n-1)$-form on the sphere bundle of $G$. It is given by the composition of the map $\Re_\g$ above with the restriction to the sphere bundle and with the Hodge star operator. Finally, integration over the normal cycle of $\Re_\g^S$ defines a valuation $\Re_\g^V(\omega) \in \mathcal V^\infty(G) \otimes \Dens(\g^*)$.

Denote by $\upsilon_m\in\mathcal V^\infty(G) \otimes \Dens(\g^*)$, $m\geq 1$, the valuation given by
 
	\[\upsilon_m= \Re_\g^V(\omega_m)=\int _{\nc(\bullet)} \Re_\g ^S\omega_m,\]
where $\omega_m:=\frac{1}{m}\xi\wedge \varrho_{m-1}$. Denote by $\upsilon_0:=\vol \otimes \vol^*$ the unit element in $\mathcal V^\infty(G) \otimes \Dens(\g^*)$. 

\begin{MainTheorem} \label{mthm:tube}
Let $G$ be a Lie group equipped with a bi-invariant metric. Then for all $Z\in\mathcal P(G)$ and $\epsilon\ll1$, 
\[  \vol(Z_\epsilon)=\sum _{k=0}^\infty\upsilon_k(Z)\epsilon^{k}.\] 
\end{MainTheorem}
Thus the image  $\mathcal{TC}(G)\subset\calV^\infty(G) \otimes \Dens(\g^*)$ of $\calT\calC$ under the realization map is spanned by the tube volume coefficients. 

 We will give two different proofs of this theorem. The first is more abstract and uses the free differential Gerstenhaber algebra described above together with recent results on tubes in riemannian manifolds due to Solanes and Trillo. The second approach is more geometric and direct, but gives less insight on the universal nature of the tube coefficients.

While we assume that our bi-invariant metric is riemannian, it is straightforward to accommodate the case of a bi-invariant Finsler metric by adjusting the definition of  the realization map $\Re^S_\g$.

\subsection{Plan of the paper}
In Section \ref{sec:dgla} we introduce the various algebraic categories that we will need, and describe explicitly the construction of free objects. In Section \ref{sec_free_dgga_one_generator} we study the free differential graded Gerstenhaber algebra on one generator $\calG$, which is key in the construction of the universal tube coefficients algebra $\calT\calC$, and find an explicit basis of $\calD\calT\calC$. In Section \ref{sec_realization} we prove an injectivity result, which will allow us to prove identities on $\calG$ by verifying them in some special cases. Sections \ref{sec_convolution_double_forms_g} and \ref{sec:abstract_convolution} construct a convolution product on $\calG$ and establish its key properties, proving Theorem \ref{mainthm_conv_dgga} and completing the proof of Theorem \ref{mainthm_expl_descr_va}. We then turn to valuations on Lie groups, proving Theorem \ref{mthm:tube} in Section \ref{sec:tube_coefficients} and Theorem \ref{mthm:existence} in Section \ref{sec:existence}.

\subsection*{Acknowledgements.} We are indebted to Thomas Wannerer for some illuminating suggestions. A large part of this work was carried out during  D.F.'s term at Tel Aviv University,  A.B.'s research stay at IHES, and our visit at the R\'enyi institute and we wish to thank those institutions for their support and productive atmosphere.

\section{Differential graded Lie algebras and Gerstenhaber algebras} 
\label{sec:dgla}

{In this section we build an abstract framework which will later be used to construct a canonical space of invariant valuations on a Lie group. More precisely, we construct explicitly the free Gerstenhaber algebra generated by a graded vector space. Its existence also follows from general category-theoretic considerations, however, the explicit construction will be useful. In particular, it allows us to construct the Koszul boundary operator $\partial$ on this algebra which will eventually be used to define an abstract analog of the space of closed differential forms on the sphere bundle of a Lie group. We will show that this space corresponds, via the exponential map, to the Maurer--Cartan elements in a differential graded Lie algebra. Finally, to construct the Maurer--Cartan elements, we use the gauge action, a procedure well known in general deformation theory.

Our main reference is \cite{manetti}. We will consider the following categories: 
\begin{itemize}
\item $\GV$: graded real vector spaces.
\item $\DGV$: differential graded vector spaces.
\item $\DBGV$: differential bigraded vector spaces.
\item $\GLA$: graded Lie algebras.
\item $\DGLA$: differential graded Lie algebras.
\item $\DBGLA$: differential bigraded Lie algebras.
\item $\DGGA$: differential graded Gerstenhaber algebras.
\item $\GUAA$: graded unitary associative algebras.
\end{itemize}
In all cases the morphisms are the obvious ones, i.e., those preserving all structures like degrees, differentials, brackets etc.

\subsection{Graded vector spaces and differential graded vector spaces}
 
 \begin{Definition}
 	\begin{enumerate}
 		\item A \emph{graded vector space} is a vector space with a direct sum decomposition $V=\bigoplus_{n \in \mathbb Z} V_n$. If $v\in V_n$ is homogeneous, we write $\bar v=\deg v=n$ for its degree. 
 		\item A \emph{differential graded vector space} is a graded vector space $V=\bigoplus_{n \in \mathbb Z} V_n$ with a linear map $d:V \to V$ such that $d(V_n) \subset V_{n+1}$ and $d \circ d=0$. 
 	\end{enumerate}
 \end{Definition}

\begin{Proposition}\label{prop_free_dgv}
	Let $V$ be a graded vector space. Then there exists a (unique up to a $\DGV$-isomorphism) differential graded vector space  $W$, a $\GV$-morphism $\iota:V \to W$ such that the following universal property holds: If $U$ is a differential graded vector space and $\phi:V \to U$ a $\GV$-morphism, then $\phi$ can be uniquely extended to a $\DGV$-morphism $\tilde \phi:W \to U$.    
\end{Proposition}

\proof 
We set $W:=V \oplus V$. The differential is defined by $d(v_1,v_2):=(0,v_1)$. The grading is defined by $\deg(v,0)=\deg v, \deg(0,v)=\deg v+1$. The extension of $\phi$ is $\tilde \phi(v,w)=\phi(v)+d\phi(w)$.
\endproof

The categories of bigraded vector spaces and differential bigraded vector spaces are defined similarly and such that $d(V_n^m) \subset V_{n+1}^m$. We will usually write $\deg=n, \deg_w=m$. Note that every graded vector space can be considered as a bigraded vector space with $\deg_w \equiv 1$. 

\subsection{Graded Lie algebras}

\begin{Definition} \label{def_graded_Lie_algebra}
	A \emph{graded Lie algebra} is a graded vector space $\calA$ with a bilinear map $[\bullet,\bullet]:\calA \times \calA \to \calA$ satisfying the following properties: 
	\begin{enumerate}
		\item $[\calA_i,\calA_j] \subset \calA_{i+j}$ (homogeneity),
		\item $[X,Y]+(-1)^{\bar X \bar Y} [Y,X]=0$ (graded skew-symmetry),
		\item $(-1)^{\bar X \bar Z} [[X,Y],Z]+(-1)^{\bar Y \bar X}[[Y,Z],X]+(-1)^{\bar Z \bar Y}[[Z,X],Y]=0$ (graded Jacobi identity).
	\end{enumerate}
\end{Definition}

\begin{Proposition} \label{prop_univ_prop_GLA}
	Let $V$ be a graded vector space. Then there exists a (unique up to $\GLA$ isomorphism) graded Lie algebra $L(V)$ and a $\GV$-morphism $\iota:V \to L(V)$ such that the following universal property holds:
	For every graded Lie algebra $C$ and every $\GV$-morphism $\phi: V \to C$ there exists a unique $\GLA$-morphism $\tilde \phi:L(V) \to C$ such that the following diagram commutes 
	\begin{displaymath}
		\xymatrix{L(V) \ar[dr]^{\tilde \phi} & \\V \ar[u]^\iota \ar[r]^\phi & C}
	\end{displaymath}
	$L(V)$ is called the \emph{free graded Lie algebra} generated by $V$. 
\end{Proposition}

The construction is as follows: Let $T(V)=\oplus_{n=0}^\infty V^{\otimes n}$ be the tensor algebra on $V$, which is the universal associative algebra generated by $V$. The grading of $V$ induces a natural grading on $T(V)$. The graded Lie bracket is given on homogeneous elements $X, Y \in T(V)$ by
\begin{displaymath}
	[X, Y]:=X\otimes Y-(-1)^{\deg X\deg Y}Y\otimes X.
\end{displaymath}
	Then $L(V)$ is defined as the smallest graded Lie subalgebra of $T(V)$ containing $V$. Equivalently, $L(V)=\oplus_{n=1}^{\infty} L_n(V)$, where $L_1(V)=V$ and $L_n(V)=[V, L_{n-1}(V)]$.

\begin{Proposition} \label{prop_dynkin_projector}
	The \emph{Dynkin projector} $\rho: T(V) \to T(V)$ defined by  
\begin{displaymath}
	\rho(V^{\otimes 0})=0, \quad \rho(v_1 \otimes \ldots \otimes v_n):=\frac{1}{n} [v_1,[v_2,\ldots,[v_{n-1},v_n]]], \quad v_1,\ldots,v_n \in V
\end{displaymath}
is a projection onto $L(V)$. The map $\tilde \phi$ from Proposition \ref{prop_univ_prop_GLA} is given by 
\begin{displaymath}
	\tilde \phi (\rho(v_1 \otimes \cdots \otimes v_n))=\frac{1}{n} [\phi(v_1),[\phi(v_2),\ldots,[\phi(v_{n-1}),\phi(v_n)]]].
\end{displaymath}
\end{Proposition}

\proof
Analogous to \cite[Theorem 2.4.4 and Corollary 2.4.5]{manetti}.
\endproof

\begin{Definition} \label{def_bigradded_Lie_algebra}
A \emph{bigraded Lie algebra} is a bigraded vector space $(\calA,\deg,\deg_w)$ with a bilinear map $[\bullet,\bullet]: \calA \times \calA \to \calA$ satisfying the following properties: 
\begin{enumerate}
	\item $\deg [X,Y]=\deg X+\deg Y, \deg_w [X,Y]=\deg_w X+\deg_w Y-1$ (bihomogeneity); 
	\item $[X,Y]+(-1)^{\deg X \deg Y+(\deg_wX-1)(\deg_wY-1)} [Y,X]=0$ (bigraded skew-symmetry); 
	\item $[X, [Y,Z]]=[[X,Y], Z]+ (-1)^{\deg X \deg Y+(\deg_w X-1)(\deg_wY-1)}[Y, [X, Z]]$ (bigraded Jacobi identity).
\end{enumerate}
\end{Definition}

\subsection{Differential graded Lie algebras}

\begin{Definition} \label{def_differential_bigradded_Lie_algebra}

		A \emph{differential graded Lie algebra (DGLA)} is a graded Lie algebra that is also a differential graded vector space such that the graded Leibniz rule  
		\begin{equation} \label{eq_graded_Leibniz_rule}
			d[X,Y]=[dX,Y]+(-1)^{\deg X} [X,dY]
		\end{equation} 
		is satisfied.  A differential bigraded Lie algebra (DBGLA) is a bigraded Lie algebra that is also a bigraded differential vector space such that \eqref{eq_graded_Leibniz_rule} holds.
\end{Definition}

Thus a DGLA can be considered as a DBGLA with $\deg_w=1$.

\begin{Proposition}\label{prop:free_dgla}
	Let $(V,d_V)$ be a differential graded vector space. Then there exists a (unique up to $\DGLA$-isomorphism) DGLA $(\calA,d_{\calA})$ and a $\DGV$-morphism $\iota:(V,d_V) \to (\calA,d_{\calA})$ such that the following universal property is satisfied:
	For every DGLA $(C,d_C)$ and every $\DGV$-morphism $\phi: (V,d_V) \to (C,d_C)$ there exists a unique $\DGLA$-morphism $\tilde \phi:(\calA,d_{\calA}) \to (C,d_C)$ such that the following diagram commutes 
	\begin{displaymath}
		\xymatrix{(\calA,d_{\calA}) \ar[dr]^{ \tilde \phi} & \\(V, d_V) \ar[u]^\iota \ar[r]^\phi & (C,d_C).}
	\end{displaymath}
	$(\calA,d_{\calA})$ is called the \emph{free differential graded Lie algebra} generated by $(V,d_V)$. 
\end{Proposition}

\proof
We define a map $d:T(V) \to T(V)$ in the natural way:
\begin{align}
	d(v_1 \otimes v_2 \otimes \cdots v_n) & :=dv_1 \otimes v_2 \otimes \cdots \otimes v_n+(-1)^{\bar v_1} v_1 \otimes dv_2 \otimes \cdots \otimes v_n \nonumber\\
	& \quad + \ldots+(-1)^{\bar v_1+\cdots+\bar v_{n-1}} v_1 \otimes \ldots \otimes v_{n-1} \otimes dv_n, \label{eq_d_on_tensor_power}
\end{align} 
turning $T(V)$ into a DGLA.
 
Note that $d$ commutes with the Dynkin projector $\rho$. Indeed, it suffices to check the identity $d\rho x=\rho dx$ for $x=v_1\otimes\dots\otimes v_n$, which is straightforward by induction on $n$. Hence the restriction of $d$ to the free graded Lie algebra $L(V) \subset T(V)$  generated by $V$ turns $L(V)$ into a DGLA $(\calA,d_{\calA})$, where $\calA=L(V)$. 

We claim that $(\calA,d_{\calA})$ satisfies the universal property. If $\phi:(V,d_V) \to (C,d_C)$ is a $\DGV$-morphism, we can forget the differential and look at this as a graded vector space morphism. By the universal property of the free graded Lie algebra, we obtain a unique $\GLA$-morphism $\tilde \phi: \calA \to C$. 

We claim that it is also a $\DGLA$-morphism, i.e. that $\tilde \phi d_{\calA}x=d_C\tilde \phi x$ for all $x \in \calA$. We may assume $x=n\rho(v_1\otimes\dots\otimes v_n)$. We then use induction on $n$, noting that the case of $n=1$ holds by assumption. Denoting $w=(n-1)\rho(v_2\otimes\dots\otimes v_n)$, we have $x=[v_1, w]$ and so by the induction assumption
\begin{align*}
	\tilde \phi d_{\calA} x&= \tilde\phi [d_V v_1, w]+(-1)^{\overline v_1}\tilde\phi [v_1, d_{\calA}w]\\&=[\tilde\phi  d_V v_1, \tilde\phi  w]+(-1)^{\overline v_1}[\tilde\phi  v_1, \tilde\phi  d_{\calA}w]\\&= [d_C \tilde\phi v_1, \tilde\phi w]+(-1)^{\overline v_1} [\tilde\phi v_1, d_C\tilde\phi w]\\&=d_C[\tilde \phi v_1, \tilde \phi w]
\\&= d_C\tilde\phi[v_1, w]\\&=d_C\tilde\phi x.
\end{align*}
 This finishes the proof.
\endproof

\subsection{Graded exterior powers}
\label{subsec_graded_ext_power}

We collect some basic facts about graded exterior powers and refer to \cite[Section 10.1]{manetti} for more details. 

Let $V$ be a graded vector space, where we write $\deg v$ for the degree of a homogeneous element. 

\begin{Proposition} \label{prop_univ_prop_graded_ext_power}
	Let $V$ be a graded vector space. Then there exists a graded unitary associative algebra $(\largewedge V,\wedge,\deg_w)$ that is unique up to a unique $\GUAA$-morphism and a $\GV$-morphism $\iota:V \to \largewedge V$ such that the following universal property is satisfied:
	For every graded unitary associative algebra  $G$ and every $\GV$-morphism $\phi:V \to G$ such that 
	\begin{displaymath}
		\phi(v_1)\phi(v_2)=(-1)^{\deg v_1 \deg v_2+1} \phi(v_2)\phi(v_1),
	\end{displaymath} 
	there exists a unique $\GUAA$-morphism $\tilde \phi:\largewedge V \to G$ such that the following diagram commutes 
	\begin{displaymath}
		\xymatrix{\largewedge V \ar[dr]^{ \tilde \phi} & \\V \ar[u]^\iota \ar[r]^\phi & G}
	\end{displaymath}
\end{Proposition}

\proof 
The proof is the usual one: Take $T(V)=\bigoplus_k V^{\otimes k}$ the tensor algebra of $V$ and factor out by the two-sided ideal $I$ in $T(V)$ generated by all elements $v_1 \otimes v_2+(-1)^{\deg v_1 \deg v_2}v_2 \otimes v_1$, where $v_1,v_2 \in V$ are homogeneous elements. Clearly $I$ is homogeneous with respect to the natural grading on $T(V)$, and hence the quotient inherits a grading denoted $\deg_w$.
\endproof
We will call $\deg_w$ the \emph{wedge degree}. When convenient we put $\widehat X=\deg_wX$.

\begin{Proposition} \label{prop_grading_on_exterior_power}
	The grading $\deg$ of $V$ extends to a grading on $\largewedge V$ by imposing additivity under the wedge product. For
	$X, Y\in \largewedge V$ bi-homogeneous it then holds that
	\begin{equation} \label{eq_supersymmetry}
		X \wedge Y=(-1)^{\deg X\deg Y+\deg_w X\deg_wY}Y\wedge X.
	\end{equation}  
\end{Proposition}

\proof 
We define a grading $\deg$ on $T(V)$ such that for all homogeneous elements $v_1,\ldots,v_l \in V$,
\begin{displaymath}
	\deg(v_1 \otimes \ldots \otimes v_l)  =\sum_i \deg v_i.
\end{displaymath}
The ideal $I$ is homogeneous, and hence the quotient $\largewedge V$ inherits a grading $\deg$. The displayed equation then follows by induction on the wedge degree of $X$ and $Y$.
\endproof

We will need another description of $\largewedge V$ as a subset of $T(V)$ instead of quotient. For homogeneous elements $v_1,\ldots,v_n$ and a permutation $\pi \in \mathcal S_n$, the Koszul sign $\varepsilon(\pi)$ is the signature of the restriction of $\pi$ to the indices such that $v_i$ has odd degree. The antisymmetric Koszul sign is defined by $\chi(\pi):=\varepsilon(\pi) \sgn(\pi)$. 

We define the linear operators $p_k: V^{\otimes k} \to V^{\otimes k}$ by $p_0=\id$, 
\begin{align}\label{eq_graded_antisymmetrization}
	p_k(v_1 \otimes \ldots \otimes v_k)= \frac{1}{k!} \sum_{\pi \in \mathcal S_k} \chi(\pi) v_{\pi_1} \otimes \ldots \otimes v_{\pi_k},
\end{align}
and set $p:=\bigoplus_k p_k:T(V) \to T(V)$.
Then one can check that $p$ is a projection with kernel $I$, from which it follows that $\largewedge V \cong \mathrm{Image}(p) \subset T(V)$.

\subsection{Differential graded Gerstenhaber algebra}
	
\begin{Definition}
	A \emph{differential graded Gerstenhaber algebra (DGGA)} is a quadruple $(\calG,d,[\bullet,\bullet],\wedge)$ such that $(\calG,d,[\bullet,\bullet])$ is a differential bigraded  Lie algebra, $(\calG,\wedge)$ is a bigraded unitary associative algebra and the two structures are compatible in the sense that for $X,Y,Z \in \calG$
	\begin{align}
		d(X \wedge Y)& =dX \wedge Y+(-1)^{\deg X} X \wedge dY \label{eq_compatibility_d_wedge}\\
		[X\wedge Y, Z]&=X\wedge[Y, Z]+(-1)^{\deg X\deg Y+\deg_w X\deg_w Y}Y\wedge [X, Z]. \label{eq_compatibility_bracket_wedge}
	\end{align} 
\end{Definition}

We remark that from the antisymmetry of the Lie bracket and \eqref{eq_compatibility_bracket_wedge}, it follows that 
\begin{equation} \label{eq_compatibility_bracket_wedge2}
	[X, Y\wedge Z] =[X, Y]\wedge Z+(-1)^{\deg Y\deg Z+\deg_wY\deg_wZ}[X, Z]\wedge Y.
\end{equation}

\begin{Proposition} \label{prop_free_dgga}
	If $\mathcal A$ is a DGLA, there is a unique structure of a DGGA on $\largewedge \mathcal A$ such that the inclusion $\mathcal A \to \largewedge \mathcal A$ is a $\DBGLA$-morphism.  $\largewedge \mathcal A$ is called the \emph{free differential graded Gerstenhaber algebra} generated by $\mathcal A$. 
\end{Proposition}

\proof 
The grading of $\calA$ induces a grading $\deg$ on $\largewedge \calA$. The second grading $\deg_w$ comes from the natural grading $\largewedge \calA=\bigoplus_k \largewedge^k \calA$.

We have to construct the differential $d$ and the Lie bracket $[\bullet,\bullet]$ on $\largewedge \calA$. Uniqueness of both is clear from \eqref{eq_compatibility_d_wedge}, \eqref{eq_compatibility_bracket_wedge} and \eqref{eq_compatibility_bracket_wedge2}. 

To construct $d$, note that the map $d:T(\calA) \to T(\calA)$ defined in \eqref{eq_d_on_tensor_power} satisfies $dI \subset I$ and hence induces a map $d:\largewedge \calA \to \largewedge \calA$. Equation \eqref{eq_compatibility_d_wedge} is then obvious.

The graded Schouten--Nijenhuis extension of the Lie bracket (see  \cite[Equation (II.14)]{graded_schouten_nijenhuis}) is given by

\begin{align*}
	&[v_1 \wedge \ldots \wedge v_k,w_1 \wedge \ldots \wedge w_l] := \\&\sum_{i,j} \epsilon_{i,j}  v_1 \wedge \ldots \wedge \widecheck v_i \wedge \ldots \wedge v_k 
	 \wedge  [v_i,w_j] \wedge w_1 \wedge \ldots \wedge \widecheck w_j \wedge \ldots \wedge w_l,
\end{align*}
where $\widecheck v_i$ means that $v_i$ is omitted, and 
\begin{displaymath}
 		\epsilon_{i,j}:=(-1)^{i+j+\deg v_i \sum_{a=i+1}^{k} \deg v_a+\deg w_j \sum_{b=1}^{j-1} \deg w_b}.
\end{displaymath}

The bihomogeneity of the bracket is obvious and the bigraded skew-symmetry is an easy computation. We have to prove the bigraded Jacobi identity and the bigraded Leibniz rule (see Definitions \ref{def_bigradded_Lie_algebra} and \ref{def_differential_bigradded_Lie_algebra}). Our proof of these equations is adapted from \cite[Lemma 9.1.4]{manetti}.	

For $X \in \largewedge \calA$, let us write $\ad_X:=[X,\bullet]$. One can easily verify that $\ad_X$ is a graded derivation in the sense that 
\begin{equation} \label{eq_derivation_adX}
	\ad_X(Y \wedge Z)=\ad_X(Y) \wedge Z +(-1)^{\deg X \deg Y+\deg_wY(\deg_wX-1)} Y \wedge \ad_X Z.
\end{equation} 

Using the antisymmetry of the bracket, this implies eq. \eqref{eq_compatibility_bracket_wedge}. 

With the Leibniz rule \eqref{eq_compatibility_d_wedge}, we see that the operator $$A_X:=d \circ \ad_X+(-1)^{\deg X+1}\ad_X \circ d-\ad_{dX}$$ satisfies 
\begin{equation} \label{eq_leibniz_rule_AX}
	A_X(Y \wedge Z)=A_XY \wedge Z \pm Y \wedge A_XZ, 
\end{equation}
where the sign depends on the bidegrees of $X,Y$. We want to show that $A_X=0$, which is precisely the bigraded Leibniz rule. Suppose first that $X$ and $Y$ are both of wedge degree $1$, i.e. elements of $\calA$. Then the equation $A_XY=0$ boils down to the graded Leibniz rule \eqref{eq_graded_Leibniz_rule} satisfied in the DGLA $\calA$. By \eqref{eq_leibniz_rule_AX} we conclude that $A_XY=0$ for $X \in \calA$ and $Y \in \largewedge \calA$. Using the skew-symmetry of the bracket, we see that $A_XY=\pm A_YX$, where the sign depends on the bidegrees of $X,Y$. We thus have $A_YX=0$ for all $Y \in \largewedge \calA, X \in \calA$. By \eqref{eq_leibniz_rule_AX}, we even have $A_YX=0$ for all $X,Y \in \largewedge \calA$, which we wanted to prove.  

Let us next prove the bigraded Jacobi identity (see Definition \ref{def_bigradded_Lie_algebra}). With
\begin{displaymath}
	T(X,Y):=\ad_X \circ \ad_Y + (-1)^{\overline X \overline Y+(\widehat X-1)(\widehat Y-1)+1}\ad_Y \circ \ad_X-\ad_{\ad_XY}
\end{displaymath}
we want to prove that $T(X,Y)=0$ for all $X,Y \in \largewedge \calA$. Repeated application of \eqref{eq_derivation_adX} shows that $T(X,Y)$ satisfies the equation 
\begin{equation} \label{eq_derivation_T}
	T(X,Y)(Z \wedge W)=T(X,Y)Z \wedge W \pm  Z \wedge T(X,Y) W.
\end{equation}
where the sign depends on the bidegrees of $X,Y,Z$.
If $X,Y,Z \in \calA$, then $T(X,Y)Z=0$ by the graded Jacobi identity and the graded skew-symmetry (Definition \ref{def_graded_Lie_algebra}) in the graded Lie algebra $\calA$. By \eqref{eq_derivation_T} it follows that $T(X,Y)Z=0$ for all $X,Y \in \calA, Z \in \largewedge \calA$. An easy computation shows that 
\begin{displaymath}
	T(X,Y)Z=\pm T(Y,Z)X=\pm T(Z,X)Y,
\end{displaymath}
where the sign depends on the bidegrees of $X,Y,Z$. It follows from what we already have shown that $T(Y,Z)X=0$ for all $X,Y \in \calA, Z \in \largewedge \calA$. By \eqref{eq_derivation_T} this even holds for all $X \in \largewedge \calA$. Hence $T(Z,X)Y=0$ for all $Z,X \in \largewedge \calA$ and $Y \in \calA$, and finally for all $Y \in \largewedge \calA$ by \eqref{eq_derivation_T}.
\endproof
For  $\calG:=\largewedge \calA$, we write $\calG^j=\largewedge^j \calA$ and $\calG_k=\{x\in\calG: \deg x=k\}$.

\begin{Proposition} \label{prop_univ_prop_dgga}
	Let $G$ be a differential graded Gerstenhaber algebra and $\calA$ a differential graded Lie algebra, considered as a bigraded vector space with $\deg_w=1$. Let $\phi:\calA \to G$ be a $\DBGLA$-morphism. Then there exists a unique $\DGGA$-morphism $\tilde \phi:\largewedge \calA \to G$ extending $\phi$. 
\end{Proposition}

\proof 
By the universal property of the graded exterior power (Proposition \ref{prop_univ_prop_graded_ext_power}), there is a unique $\GUAA$-morphism $\tilde \phi$ extending $\phi$. The Lie bracket on $\largewedge \calA$ and the Lie bracket on $G$ both satisfy eqs. \eqref{eq_supersymmetry} and \eqref{eq_compatibility_bracket_wedge}, hence $\tilde \phi$ is compatible with the Lie brackets. Similarly, the differentials on $\largewedge \calA$ and $G$ both satisfy eq. \eqref{eq_compatibility_d_wedge}, hence $\tilde \phi$ is compatible with $d$. We conclude that $\tilde \phi$ is a $\DGGA$-morphism.
\endproof

Let us finally summarize the previous universal properties in one theorem. 

\begin{Theorem} \label{thm_universal_property_dgga}
	Let $V$ be a graded vector space. Let $W$ be the free differential graded vector space generated by $V$ (Proposition \ref{prop_free_dgv}), $\calA$ the free differential graded Lie algebra generated by $W$ (Proposition \ref{prop:free_dgla}) and $\mathcal G=\largewedge \calA$ be the free differential graded Gerstenhaber algebra generated by $\calA$ (Proposition \ref{prop_free_dgga}). Then $\mathcal G$ has the following universal property: If $\phi:V \to G$ is a $\GV$-morphism and $G$ a differential graded Gerstenhaber algebra, then there exists a unique $\DGGA$-morphism $\tilde \phi: \mathcal G \to G$ extending $\phi$.
	\begin{displaymath}
		\xymatrix{\mathcal G=\largewedge \calA \ar[dr]^{\tilde \phi} & \\V \ar[u] \ar[r]^\phi & G}
	\end{displaymath}
\end{Theorem}

\proof 
By Proposition \ref{prop_free_dgv}, we can extend $\phi$ uniquely to a $\DBGV$-morphism $\phi_1:W \to G$. By Proposition \ref{prop:free_dgla}, we can extend $\phi_1$ uniquely to a $\DBGLA$-morphism $\phi_2:\calA \to G$. 
Finally by Proposition \ref{prop_univ_prop_dgga} we can extend $\phi_2$ uniquely to a $\DGGA$-morphism $\phi_3:\largewedge \calA \to G$. We thus set $\tilde \phi:=\phi_3$. 
\endproof

\subsection{The Koszul boundary operator}
\begin{Definition}
	Given a DGLA $\calA$ and the associated DGGA $\calG=\largewedge \calA$, the \emph{Koszul boundary operator} $\partial: \calG^\bullet\to \calG^{\bullet-1}$  is defined inductively on the wedge degree as follows: For $v\in \calA$ and $X\in \calG$ we set $\partial (v)=0$ and
	\begin{equation}\label{eq:koszul_exterior}\partial (v\wedge X)=[v, X]-v\wedge\partial X.\end{equation}
\end{Definition}

Note that $\partial$ decreases $\deg_w$ by $1$ while preserving $\deg$. Moreover,
\begin{displaymath}
	\partial( \largewedge^m \calA_1) \subset \largewedge^{m-2} \calA_1\otimes \calA_2.
\end{displaymath}

\begin{Proposition}\label{prop:partial_wedge}
	It holds for $X, Y\in \calG$ that 
	\[\partial (X\wedge Y)= \partial X\wedge Y+(-1)^{\deg_w X-1}[X, Y]+(-1)^{\deg_w X}X\wedge \partial Y.\]
\end{Proposition}

\proof
We prove by induction on $\deg_w X+\deg_wY$. The statement is clear for $\deg_w X=\deg_wY=1$. We may assume by linearity that $X, Y$ are simple wedge products. As both sides of the identity are easily seen to be supersymmetric in $X, Y$, it suffices to consider the case $X=u\wedge X'$ with $u\in\calA$. 

The verification is now straightforward. We have by the induction hypothesis
\begin{align*}
	\partial(X\wedge Y)&=\partial (u\wedge X'\wedge Y)
	\\&=[u, X'\wedge Y]-u\wedge\partial (X'\wedge Y)
	\\&=[u,X']\wedge Y+(-1)^{\overline {X'}\overline Y+\widehat {X'}\widehat Y}[u, Y]\wedge X'
	\\&-u\wedge \partial X'\wedge Y+(-1)^{\widehat {X'}}u\wedge [X', Y]+(-1)^{\widehat {X'}+1}u\wedge X'\wedge \partial Y.
\end{align*}
This should be shown to equal
\begin{align*}
	&\partial (u\wedge X')\wedge Y+(-1)^{\widehat X-1}[u\wedge X',Y]+(-1)^{\widehat X} u\wedge X'\wedge\partial Y\\
	&=[u,X']\wedge Y-u\wedge \partial X'\wedge Y
	\\&+(-1)^{\widehat {X'}}u\wedge [X',Y]+(-1)^{\overline {X'}\overline u}X'\wedge[u,Y]
	\\&+(-1)^{\widehat X}u\wedge X'\wedge\partial Y,
\end{align*}
which is easily seen to hold. 
\endproof

\begin{Proposition}\label{prop:partial_lie_derivative}
	The Koszul boundary operator is a derivation for the Lie bracket on $\calG$:
	\begin{displaymath}
		\partial [X, Y]=[\partial X, Y]+(-1)^{\deg_w X+1}[X, \partial Y], \quad X, Y\in \calG.
	\end{displaymath}
\end{Proposition}

\proof
We do induction on $\deg_w X+\deg_wY$. The assertion is clear for $\deg_w X=\deg_wY=1$. We may assume by linearity that $X, Y$ are simple wedge products. As the identity is easily seen to be supersymmetric in $X, Y$, we may assume $\deg_w X\leq \deg_wY$ and $Y=v\wedge Y'$ with $v\in\calA$.

The verification is now straightforward. We have by Proposition \ref{prop:partial_wedge} and the induction hypothesis 

\begin{align*}
	\partial [X,Y]&=\partial([X, v]\wedge Y'+(-1) ^{\overline v\overline {Y'}+\widehat {Y'}}[X, Y']\wedge v) \\
	&=[\partial X, v]\wedge Y' +(-1)^{\widehat X+1}[[X, v], Y']+(-1)^{\widehat X}[X, v]\wedge \partial Y'\\&+(-1)^{\overline v\overline X+\widehat X+1}[v, [X, Y']]+(-1)^{\overline v\overline X+\widehat X}
	v\wedge \partial[X, Y'] \\&=[\partial X, v]\wedge Y' +(-1)^{\widehat X+1}[[X, v], Y']+(-1)^{\widehat X}[X, v]\wedge \partial Y'\\&+(-1)^{\overline v\overline X+\widehat X+1}[v, [X, Y']]+(-1)^{\overline v\overline X+\widehat X}
	v\wedge [\partial X, Y']+(-1)^{\overline v\overline X+1}v\wedge [X, \partial Y'].
\end{align*}
We ought to show this equals
\begin{align*}
	&[\partial X, v\wedge Y']+(-1)^{\widehat X+1}[X, \partial (v\wedge Y')]\\
	&=[\partial X, v\wedge Y']+(-1)^{\widehat X+1}[X, [v,Y']-v\wedge\partial Y']
	\\&=[\partial X, v]\wedge Y'+(-1)^{\overline v\overline{Y'}+\widehat {Y'}}[\partial X, Y']\wedge v
	\\&+(-1)^{\widehat X+1}[X, [v, Y']]+(-1)^{\widehat X}[X, v]\wedge\partial Y'+(-1)^{\widehat X+\overline v\overline {Y'}+\widehat {Y'}+1}[X, \partial Y']\wedge v,
\end{align*}
which easily follows from the supersymmetry relations and the Jacobi identity.
This completes the induction and the proof.
\endproof

\begin{Proposition}
	It holds that $\partial^2=0$ on $\calG$.
\end{Proposition}
\proof
If $\partial^2 X=\partial ^2 Y=0$, it follows by Propositions \ref{prop:partial_lie_derivative} and \ref{prop:partial_wedge} that also $\partial^2(X\wedge Y)=0$. The claim follows by induction on wedge degree.
\endproof

\subsection{Maurer--Cartan elements and the gauge action}

\begin{Definition}
	Let $\calA$ be a differential graded Lie algebra. An element $a \in \calA_1$ such that 
	\begin{equation}\label{eq:maurer_cartan}
		da+\frac12 [a,a]=0
	\end{equation}
	is called a \emph{Maurer--Cartan element}, and $\mathrm{MC}(\calA)$ denotes the cone of such elements. 
\end{Definition}

If $\calA$ is nilpotent, there is an action, called {\it gauge action}, of $\exp(\calA_0)$ on $\calA_1$ that leaves $\mathrm{MC}(\calA)$ invariant \cite[Lemma 6.3.4]{manetti}. Since we will be dealing with a differential graded Lie algebra that is not nilpotent, we work with a slight variation of the gauge action as follows.

Let $\calA[[\epsilon]]$ be the space of formal power series with coefficients in $\calA$. This is again a differential graded Lie algebra, where $\epsilon$ is of degree $0$ and where 
\begin{displaymath}
	d\sum_{i=0}^\infty a_i \epsilon^i:=\sum_{i=0}^\infty da_i \epsilon^i, \quad\left[\sum_{i=0}^\infty a_i\epsilon^i,\sum_{j=0}^\infty b_j \epsilon^j\right]:=\sum_{i,j=0}^\infty [a_i,b_j]\epsilon^{i+j}.
\end{displaymath}
A sequence of $\sum_i a_i^{(k)}\epsilon^i$ of elements in $\calA[[\epsilon]]$ converges to $\sum_i a_i\epsilon^i$ if for each $i$, $a_i^{(k)}$ converges to $a_i$ in the discrete topology. We write $\calA[[\epsilon]]_{>0} \subset \calA[[\epsilon]]$ for the elements without absolute term.

\begin{Definition}
	The \emph{gauge action} of $\exp(\calA_0[[\epsilon]]_{>0})$ on $\calA_1[[\epsilon]]$ is defined by the converging series
	\begin{displaymath}
		e^a * x :=x+ \sum_{n=0}^\infty \frac{\ad(a)^n}{(n+1)!}([a,x]-da), \quad a \in \calA_0[[\epsilon]]_{>0}, x \in \calA_1[[\epsilon]].
	\end{displaymath}
\end{Definition} 

Observe that this is well defined, as for each fixed power of $\epsilon$ there are only finitely many summands.

\begin{Proposition} \label{prop_gauge_action}
	The set $\mathrm{MC}(\calA[[\epsilon]])$ of Maurer--Cartan elements is invariant under the gauge action. In particular, since
	 $0\in\MC(\calA[[\epsilon]])$, we have $e^a * 0 \in \mathrm{MC}(\calA[[\epsilon]])$ for every $a \in \calA_0[[\epsilon]]_{>0}$.
\end{Proposition}

\proof
Following \cite[Section 6.3]{manetti}, we construct a new DGLA $\calA[[\epsilon]]':=\calA[[\epsilon]] \oplus \R \cdot d$, where $d$ is a formal symbol of degree $1$. The Lie bracket is defined by 
\begin{displaymath}
	[\tau+v d,\zeta+w d]':=[\tau,\zeta]+v d\zeta-(-1)^{\deg \tau} w d\tau, \quad \tau,\zeta \in \calA[[\epsilon]], v,w \in \R.
\end{displaymath}
The differential is defined by 
\begin{displaymath}
	d'(\tau+v d)=[d,\tau+v d]=d\tau.
\end{displaymath}
One easily checks that $(\calA[[\epsilon]]',[\bullet,\bullet]',d')$ is a DGLA. The gauge action can be rewritten in $\calA[[\epsilon]]'$ as 
\begin{displaymath}
	e^a * x=e^{\ad a}(x+d)-d.
\end{displaymath}

Let $\tau',\zeta'\in \calA[[\epsilon]]'$. By the graded Jacobi identity, we have for each $a \in\calA_0[[\epsilon]]_{>0}$  
\begin{displaymath}
	\ad(a) [\tau',\zeta']'=[\ad(a)\tau',\zeta']'+[\tau',\ad(a)\zeta']'.
\end{displaymath}
It follows inductively that 
\begin{displaymath}
	\ad(a)^n [\tau',\zeta']'=\sum_{k+l=n} \binom{n}{k} [\ad(a)^k \tau',\ad(a)^l \zeta']' 
\end{displaymath}
and hence 
\begin{displaymath}
	e^{\ad a}[\tau',\zeta']'=[e^{\ad a}\tau',e^{\ad a}\zeta']'.
\end{displaymath}

Let $\tau \in \calA_1[[\epsilon]]$. Then $\tau$ is a Maurer--Cartan element if and only if $[\tau+d,\tau+d]'=0$. In this case, 
\begin{displaymath}
	[e^a* \tau+d,e^a* \tau+d]'=[e^{\ad(a)}(\tau+d),e^{\ad(a)}(\tau+d)]'= e^{\ad(a)}[\tau+d,\tau+d]'=0,
\end{displaymath}
which means that $e^a* \tau$ is a Maurer--Cartan element.
\endproof

Let $\calA$ be a DGLA and $\calG$ the DGGA generated by $\calA$. All operations on $\calG$ extend to the space of formal power series $\calG[[\epsilon]]$ by (multi)-linearity in $\epsilon$, making the latter into a DGGA. Convergence in $\calG[[\epsilon]]$ is defined as for $\calA[[\epsilon]]$. 

Let $\calA_1[[\epsilon]]_{>0} \subset \calA[[\epsilon]]$ denote the space of elements of degree $1$ with no absolute term. Note that elements in this space pairwise commute. We define the exponential map by the converging series
\begin{displaymath}
	\widehat{\exp}:\calA_1[[\epsilon]]_{>0} \to \calG[[\epsilon]], \zeta \mapsto \sum_{n=0}^\infty \frac{\zeta^{\wedge n}}{n!}.
\end{displaymath}

\begin{Proposition} \label{prop_mc_and_vals}
	Let $\zeta \in \calA_1[[\epsilon]]_{>0}$. Then $-\zeta \in \mathrm{MC}(\calA[[\epsilon]])$ if and only if $\widehat\exp(\zeta) \in \ker(d-\partial)$.
\end{Proposition}

\begin{proof}
	We have $d\widehat\exp(\zeta)= d\zeta \wedge \widehat\exp(\zeta)$ and 
	\begin{displaymath}
		\partial \widehat\exp(\zeta)=\sum_{j=0}^\infty\frac{1}{j!}\partial (\zeta^j)=\sum_{j=0}^\infty\frac{1}{j!}{j\choose 2}[\zeta, \zeta]\wedge\zeta^{j-2}=\frac12 [\zeta, \zeta]\wedge \widehat\exp(\zeta).
	\end{displaymath} 
	Thus  $\widehat\exp(\zeta) \in \ker(d-\partial)$ if and only if $d \zeta=\frac 12[ \zeta, \zeta]$, that is $-\zeta$ is Maurer--Cartan.
\end{proof}

\section{The free differential graded Gerstenhaber algebra on one generator}
\label{sec_free_dgga_one_generator}

In the previous section we considered the free differential graded Gerstenhaber algebra generated by a graded vector space. We will now apply this general construction to a vector space generated by one element. In the resulting Gerstenhaber algebra we identify a certain subspace which we denote $\DTC$. Its importance will become clear in Section \ref{sec:tube_coefficients}, where its elements will yield closed and vertical $n$-forms on the sphere bundle of an $n$-dimensional Lie group endowed with a bi-invariant riemannian metric and hence, via the Rumin differential, smooth valuations on the group.

The main result of this section is an explicit description of the space $\DTC$. Namely, using the gauge-action construction of Maurer--Cartan elements, we identify a sequence of elements and using some combinatorial arguments we show that it in fact spans the whole space.

\subsection{The tube forms}

Let $V_0$ denote the graded vector space over $\R$ spanned by one element $\xi$ of degree $0$. Let $V$ be the free differential graded vector space generated by $V_0$, i.e. $V=\R \cdot \xi \oplus \R \cdot \gamma$, where $\gamma=d\xi$, see Proposition \ref{prop_free_dgv}. We set $V_1=\R \cdot \gamma$.

The free DGLA on one generator, denoted $\mathcal A$, is the free DGLA generated by $V_0$, which exists by Proposition \ref{prop:free_dgla}. Denote by $\calA_k$ the homogeneous elements in $\calA$ of degree $k$. Note that $\xi \in \calA_0, \gamma \in \calA_1$.

Since $0$ is a Maurer--Cartan element in $\calA[[\epsilon]]$, the gauge action from Proposition \ref{prop_gauge_action} allows us to produce more Maurer--Cartan elements. Since $\epsilon \xi \in \calA_0[[\epsilon]]_{>0}$, we have 
\begin{align}
	\zeta_\epsilon & :=e^{\epsilon \xi} * 0 =  - \sum_{i=0}^\infty \frac{\epsilon^{i+1}}{(i+1)!} (\ad \xi)^i \gamma  \nonumber \\
	& = -\epsilon \gamma-\frac{\epsilon^2}{2} [\xi \gamma ]-\frac{\epsilon^3}{6} [\xi^2\gamma]-\cdots \in \mathrm{MC}(\mathcal A[[\epsilon]]).
	\label{eq_def_zeta}
\end{align}

Here and in the following, we use the right-stacked iterated Lie bracket notation, e.g. $[\xi^2\gamma]=[\xi\xi\gamma]$ stands for  $[\xi,[\xi,\gamma]]$.  

We let $\calG:=\largewedge \calA$ be the free differential graded Gerstenhaber algebra generated by $\R \cdot \xi$, see Section \ref{sec:dgla}. Recall that $\calG^i=\largewedge^i \calA$.

We will use another grading on $\mathcal G$, called symbol degree, which is given by $\deg_s\xi=\deg_s\gamma=1$ and $\deg_s [X,Y]=\deg_sX+\deg_sY$ and $\deg_s X \wedge Y=\deg_sX +\deg_sY$. To see that it is well defined, recall from Proposition \ref{prop_univ_prop_GLA} that the DGLA $\calA$ is given by $L(V)=\bigoplus_{n=1}^\infty L_n(V) \subset T(V)$, where $L_1(V)=V$ and $L_n(V)=[V,L_{n-1}(V)]$. The symbol degree on $L_n(V)$ is just $n$. To see that it extends to the graded exterior power is the same argument as in Proposition \ref{prop_grading_on_exterior_power}.

\begin{Definition} \label{def_varrho_m}
	For $m \geq 0$, we define the \emph{tube forms} $\varrho_m \in \bigoplus_{i \leq m} \calG^i_i$ by 
	\begin{align}
		\varrho_\epsilon=\sum_{m=0}^\infty \varrho_m \epsilon^m:=\widehat\exp(-\zeta_\epsilon)\in \calG[[\epsilon]],
	\end{align}
	where $\zeta_\epsilon \in \mathrm{MC}(\mathcal A[[\epsilon]])$ is defined in \eqref{eq_def_zeta}. 
\end{Definition}

 Observe that $\deg_s\varrho_m=m$. By Proposition \ref{prop_mc_and_vals}, we have 
\begin{equation} \label{eq:tau_closed}
	d\varrho_m=\partial \varrho_m.
\end{equation}
By \eqref{eq_def_zeta} we find that 
\begin{equation} \label{eq_main_term_rhot}
	\varrho_\epsilon=1+\sum_{m=1}^\infty  \frac{\gamma^m}{m!}\epsilon^m+\upsilon,
\end{equation}
where $\upsilon$ consists of all summands containing Lie brackets. 

Denoting $\gamma^k_m:=\frac{1}{k!(m-k)!}\gamma^k$, the few first and the last summands of $\varrho_m$ can be written as 
\begin{align*}
	\varrho_m&=\gamma^{m}_m\\
	&+[\xi\gamma]\wedge \gamma^{m-2}_m\\
	& +[\xi^2\gamma]\wedge \gamma^{m-3}_m+([\xi^3\gamma]+3[\xi \gamma ]^2) \wedge \gamma^{m-4}_m+\\
	&+(10[\xi^2\gamma] \wedge [\xi \gamma] + [\xi^4\gamma]) \wedge \gamma^{m-5}_m \\
	& +([\xi^5 \gamma]+15[\xi \gamma]^3+15 [\xi^3\gamma][\xi\gamma]+10[\xi^2\gamma]^2) \wedge \gamma^{m-6}_m\\
	& +\dots\\
	&+\frac{1}{m!}[\xi^{m-1}\gamma].
\end{align*}

Observe that all the wedge factors in the summands comprising $\varrho_m$ have degree $1$, and thus they commute. In particular,
\begin{align*}
	\varrho_0&=1\\
	\varrho_1&= \gamma\\
	\varrho_2&=\frac 12(\gamma^2+[\xi\gamma])\\
	\varrho_3&=\frac 1{6}(\gamma^3+3[\xi\gamma ]\gamma+[\xi^2\gamma])\\
	\varrho_4&=\frac1{24}(\gamma^4+6[\xi\gamma]\gamma^2+4[\xi^2\gamma]\gamma +3[\xi\gamma]^2+[\xi^3\gamma]).
\end{align*}

\begin{Definition}\label{def:omega}
	Define  for $m\geq 1$
	\begin{displaymath}
		\omega_m:=\frac 1m\xi\wedge\varrho_{m-1}.
	\end{displaymath}
\end{Definition}

\begin{Proposition}\label{prop:primitive}
	For each $m\geq1$  
	\begin{displaymath}
		\varrho_m=d\omega_m+\partial \omega_m.
	\end{displaymath}	 
\end{Proposition}
In particular, this yields a recursive formula for $\varrho_m$:
\begin{displaymath}
	\varrho_m=\frac1m(d+\partial)(\xi\wedge\varrho_{m-1}).
\end{displaymath}
Later, in Proposition \ref{prop:hardlefschetz_convolution}, we will  see another recursive formula using the convolution operation.

\proof
Using Prop. \ref{prop:partial_wedge} we have 
\begin{align*}
	d\omega_m & =\frac1m \gamma\wedge \varrho_{m-1}+\frac1m \xi\wedge d\varrho_{m-1}\\
	\partial \omega_m & =-\frac1m \xi\wedge \partial \varrho_{m-1}+\frac 1m [\xi, \varrho_{m-1}].
\end{align*}

Since $\partial \varrho_{m-1}=d\varrho_{m-1}$ by eq. \eqref{eq:tau_closed}, it remains to verify that
\begin{displaymath}
	\gamma\wedge\varrho_{m-1}+[\xi,\varrho_{m-1}]=m\varrho_m.
\end{displaymath}
Multiplying by $\epsilon^{m-1}$ and summing over $m\geq 1$, we should check
\begin{equation}\label{eq:omega}
\gamma \wedge \widehat\exp(-\zeta_\epsilon)+[\xi,\widehat\exp(-\zeta_\epsilon)]= \frac{d}{d\epsilon} \widehat\exp(-\zeta_\epsilon).
\end{equation}
We have \begin{equation}\label{eq:zeta_dt}
	\frac{d}{d\epsilon} \zeta_\epsilon=- \sum_{i=0}^\infty \frac{[\xi^i \gamma]}{i!}\epsilon^{i}=-\gamma+[\xi,\zeta_\epsilon],
\end{equation}
and so the right hand side of the desired equality \eqref{eq:omega} is \begin{align*}
	 \frac{d}{d\epsilon} \widehat\exp(-\zeta_\epsilon) &= -\frac{d}{d\epsilon}\zeta_\epsilon \wedge \widehat\exp(-\zeta_\epsilon) \\
	& =\gamma \wedge \widehat\exp(-\zeta_\epsilon) - [\xi,\zeta_\epsilon] \wedge \widehat\exp(-\zeta_\epsilon).\end{align*}

Now 
\begin{align*}
	[\xi,\widehat\exp(-\zeta_\epsilon)] & =\sum_{i=0}^\infty(-1)^i \frac{1}{i!}[\xi,\zeta_\epsilon^i]=\sum_{i=1}^\infty(-1)^i\frac1{(i-1)!}[\xi, \zeta_\epsilon]\wedge \zeta_\epsilon^{i-1}\\&=- [\xi, \zeta_\epsilon]\wedge\widehat\exp(-\zeta_\epsilon),
\end{align*}
concluding the proof.
\endproof

The first values of $\omega_m$ are
\begin{align*}
	\omega_1 &=\xi\\
	\omega_2 &=\frac 12 \xi \wedge \gamma\\
	\omega_3 &=\frac 16 \xi\wedge (\gamma^2+[\xi\gamma])\\
	\omega_4 &=\frac 1{24}\xi\wedge(\gamma^3+3[\xi\gamma ]\gamma+[\xi^2\gamma])\\
	\omega_5 & =\frac1{120}\xi\wedge(\gamma^4+6[\xi\gamma]\gamma^2+4[\xi^2\gamma]\gamma +3[\xi\gamma]^2+[\xi^3\gamma]).
\end{align*}

\subsection{Abstract valuations}

\begin{Definition} \label{def_operator_D}
	We write $\Deg X=\deg X-\deg_wX$, $\calG(k)=\{\tau\in \calG: \Deg\tau=k\}$. Define the operator $D:\calG(\bullet)\to \calG(\bullet+1)$ by
	\begin{displaymath}
		DX=dX+(-1)^{\Deg X+1}\partial X,
	\end{displaymath}
	and denote $\calZ:=\ker(D)$.
\end{Definition}

Since $d$ and $\partial$ commute, we find that
\begin{displaymath}
	Dd+dD=0, \quad D\partial =\partial D, \quad D^2=0.
\end{displaymath}
In particular, $\calZ$ is closed under both $d$ and $\partial$, and $\mathrm{Image}(D)\subset \calZ$.

\begin{Definition}
	The space of \emph{horizontal} elements $\calH\subset \calG$ is the span of all wedge products of elements from $\calA$ that are in $\oplus_{j=1}^\infty\calA_j$. In other words, none of the wedge factors is proportional to $\xi$. 
\end{Definition}

We now define a subspace of $\calG$ closely related to the space of abstract tube coefficients that will be defined in Section \ref{sec:abstract_convolution}.

\begin{Definition}
	\begin{displaymath}
		\calD\calT\calC  :=\calZ\cap\calH\cap\calG(0). 
	\end{displaymath}
\end{Definition}

It is easily checked that the tube forms $\varrho_m$ from Definition \ref{def_varrho_m} belong to $\calD\calT\calC $. The main theorem of this section is that they already generate $\calD\calT\calC $. 
  
\begin{Theorem} \label{thm_abstract_valuations}
	The space $\calD\calT\calC $ is spanned by the tube forms $\varrho_m$, $m\geq 0$. 
\end{Theorem}

We will need some technical statements before we can prove the theorem.

\subsection{Higher Jacobi identities}

As a consequence of the graded Jacobi identity satisfied in $\calA$, there are higher Jacobi identities among iterated Lie brackets. A systematic study of such identities is contained in \cite{alekseev_ivanov}. We only need a special case that we prove directly here. 

Recall that we constructed $\calA$ as a graded Lie subalgebra of the tensor algebra $T(V)$, where $V=\R \cdot \xi \oplus \R \cdot \gamma$, see the proof of Proposition \ref{prop_univ_prop_GLA}. 
\begin{Lemma}\label{lem:universal_commutators}
	It holds in $T(V)$  that 

\begin{displaymath}
		[\xi^b \gamma \xi^a \gamma]=\sum_{0\leq i+j\leq a+b} c_{i,j}(a,b)\xi^i\gamma\xi^j \gamma\xi^{a+b-i-j},
	\end{displaymath}	
	where
	\begin{displaymath}
		c_{i,j}(a,b)= (-1)^{i+a+b}{a\choose j}\left((-1)^j{b\choose i}+{b\choose i+j-a}\right).
	\end{displaymath}

In terms of generating functions, we have
\begin{equation} \label{eq_genfunction_c}
		\sum_{i,j}c_{i,j}(a,b)x^iy^j=(x-1)^b \big((y-1)^a+(x-y)^a\big).
\end{equation}
\end{Lemma}

\proof
By induction on $a$, one proves 
\begin{displaymath}
		[\xi^a \gamma]=\sum_{i=0}^a (-1)^{i+a} \binom{a}{i} \xi^i \gamma \xi^{a-i},
\end{displaymath}
from which the case $b=0$ of the lemma follows. The general case follows by induction over $b$.
\endproof

Let $B_n$ denote the Bernoulli numbers, that is
\[\frac{x}{e^x-1}=\sum_{n=0}^\infty \frac{B_n}{n!}x^n.\]
Thus $B_0=1$, $B_1=-1/2$, $B_2=1/6$, and $B_{2k+1}=0$ for $k\geq 1$.

We set 
\begin{displaymath}
	\beta_b(m)=-\frac{2(2^{b+1}-1)}{b+1}{m\choose b}B_{b+1}.
\end{displaymath}

A generating function for $\beta_b(m)$ is 
\begin{equation} \label{eq_gen_funct_beta}
	\sum_{b \leq m} \beta_b(m) x^b \frac{y^{m-b}}{m!}=\frac{2e^y}{e^x+1}.
\end{equation}

\begin{Proposition}\label{prop:higher_jacobi}
	\begin{enumerate}
		\item For all odd $m\geq 1$ we have the relation 
		\begin{equation}\label{eq:bernoulli_relation}
			\sum_{b=0}^{m}\beta_b(m) [\xi^b\gamma\xi^{m-b}\gamma]=0,
		\end{equation}
		or equivalently
		\begin{equation} \label{eq:delta_bernoulli2}
			[\gamma \xi^m\gamma]=\sum_{j=1}^{\frac{m+1}{2}}{m\choose 2j-1}\frac{4^j-1}{j}B_{2j} [\xi^{2j-1}\gamma\xi^{m-(2j-1)}\gamma].
		\end{equation}
	\item The iterated Lie brackets $[\xi^{m-2k}\gamma  \xi^{2k}\gamma]$,  $0\leq k\leq \frac m 2$, form a basis of
	$\mathrm{Span} \{[\xi^{m-j}\gamma \xi^{j}\gamma]\}_{j\geq 0}$.
	\end{enumerate}

\end{Proposition}
\proof

Using Lemma \ref{lem:universal_commutators} and \eqref{eq_genfunction_c}, \eqref{eq:bernoulli_relation} is equivalent to 
\begin{displaymath}
		\sum_{b=0}^m \beta_b(m) (x-1)^b \big((y-1)^{m-b}+(x-y)^{m-b}\big)=0
\end{displaymath}	
	for all odd $m$, which is equivalent to saying that the formal power series
\begin{displaymath}
		\sum_{b \leq m} \frac{1}{m!} \beta_b(m) (x-1)^b \big((y-1)^{m-b}+(x-y)^{m-b}\big) t^m \in \R[x,y][[t]]
\end{displaymath}
	is even. By \eqref{eq_gen_funct_beta}, this power series is given by 
\begin{displaymath}
		2 \frac{e^{(y-1)t}+e^{(x-y)t}}{e^{(x-1)t}+1},
\end{displaymath}
	and it is easy to check that it is indeed even. This proves \eqref{eq:bernoulli_relation}. The non-zero terms in \eqref{eq:bernoulli_relation} are for $b=0$ and $b=2j-1$ odd, and writing it out gives \eqref{eq:delta_bernoulli2}.

Let us now prove (ii). From \eqref{eq:delta_bernoulli2} it follows that the displayed elements form a generating system. It remains to show that they are linearly independent. We prove this by induction over $m$, the cases $m=0,1$ being trivial. Suppose that $m \geq 2$ and that we have a linear relation 
\begin{displaymath}
		\sum_{k=0}^{\left\lfloor \frac{m}{2} \right\rfloor} c_k [\xi^{m-2k}\gamma \xi^{2k}\gamma]=0.
\end{displaymath}

This is equivalent to 
\begin{align}
\label{eq:all_c_k}
		\sum_{k=0}^{\left\lfloor \frac{m}{2} \right\rfloor} c_k(x-1)^{m-2k}((y-1)^{2k}+(x-y)^{2k})=0.
\end{align}

Taking two partial derivatives with respect to $y$ gives 
\begin{displaymath}
		\sum_{k=0}^{\left\lfloor \frac{m-2}{2} \right\rfloor} c_{k+1}(2k+2)(2k+1)(x-1)^{(m-2)-2k}((y-1)^{2k}+(x-y)^{2k})=0.
\end{displaymath}

By induction, $c_{k+1}=0$ for $k=0,\dots,\lfloor \tfrac{m-2}{2} \rfloor$, and so it follows from \eqref{eq:all_c_k} that also $c_0=0$.  
\endproof

\begin{Lemma} \label{lemma_sum_unequal0}
	For each $m \geq0$, we have 
\begin{displaymath}
		\sum_{k=0}^m [ \xi^{m-k} \gamma \xi^k \gamma] \neq 0.
\end{displaymath}
\end{Lemma}

\proof
Otherwise we would obtain as above that $\sum_{b=0}^m (x-1)^b((y-1)^{m-b}+(x-y)^{m-b})=0$, but the coefficient of $x^m y^0$ is $m+2 > 0$. 
\endproof

\subsection{Cohomology}

\begin{Lemma}\label{lem:unique_closed}
	Let $\omega \in \largewedge^m\mathcal A_1$ such that $d\omega=0$. Then there exist $c \in \C$ with $\omega=c \gamma^m$.
\end{Lemma}

\proof
Denote $\Lambda_j(k):=\{\lambda_1\geq\dots\geq\lambda_j\geq 0,  \lambda_1+\dots+\lambda_j=k\}$.
Since $d$ preserves the symbol degree, we may assume $\deg_s\omega=n\geq m$. 
If $n=m$ then, since $\xi\wedge\xi=0$, it must hold that $\omega=c\gamma^m+ c_1\xi \wedge\gamma^{m-1}$, but then $d\omega=0$ implies $c_1=0$ and we are done.

Next assume $n>m$. We may write
\begin{displaymath}
	\omega= \sum_{\lambda\in\Lambda_{m}(n-m)} c_\lambda [ \xi^{\lambda_1} \gamma  ]\wedge\dots\wedge [ \xi^{\lambda_m} \gamma  ].
\end{displaymath}
Then
\begin{align*}
	d\omega&=\sum_{j=1}^m(-1)^{j+1}\sum_{\lambda\in\Lambda_{m}(n-m): \lambda_j>0}c_\lambda\bigwedge_{t=1}^{j-1}[ \xi^{\lambda_t}\gamma] \wedge \sum_{k=0}^{\lambda_j-1}[\xi^{\lambda_j-1-k}\gamma\xi^k\gamma]\wedge\bigwedge_{t=j+1}^m[\xi^{\lambda_t}\gamma] 
	\\&=\sum_{j=1}^m \quad\sum_{\lambda\in\Lambda_{m}(n-m): \lambda_j>0} c_\lambda\sum_{k=0}^{\lambda_j-1}[\xi^{\lambda_j-1-k}\gamma\xi^k\gamma]\wedge \bigwedge_{1\leq t\leq m ,t\neq j}[\xi^{\lambda_t}\gamma] .
	\end{align*}

For $\mu=i_j(\lambda):=(\lambda_1,\dots, \lambda_{j-1}, \lambda_{j+1}, \dots, \lambda_m)$ we have $|\mu|\leq n-m$ and $\lambda_j=n-m-|\mu|$. For any fixed $\mu$ with $|\mu|<n-m$ we find, by Lemma \ref{lemma_sum_unequal0}, that
\begin{equation} \label{eq_sum_coeffs}
\sum_{j=1}^m \sum_{\substack{\lambda \in \Lambda_m(n-m)\\ i_j(\lambda)=\mu}} c_\lambda=0.	
\end{equation}

Consider any $\widehat \lambda \in \Lambda_m(n-m)$. Since $n>m$, $|\widehat \lambda|=n-m>0$ and so $\widehat \lambda\neq (0,\dots, 0)$. Let $r$ be such that $\widehat\lambda_1=\dots=\widehat\lambda_r>\widehat\lambda_{r+1}$. Put $\mu:=(\widehat\lambda_2,\dots,\widehat \lambda_m)$. 

Assume that $i_j(\lambda)=\mu$ for some $j$ and $\lambda \in \Lambda_m(n-m)$. We claim that $1\leq j\leq r$.

Indeed, assume by contradiction that $j>r$. Then $\lambda_j=n-m-|\mu|=\widehat\lambda_1  \geq \mu_1=\lambda_1$, and so $\lambda_1=\dots=\lambda_j=\hat \lambda_1$. It follows that $\sum_{i=1}^j \lambda_i>\sum_{i=1}^j \hat \lambda_i$. For $i>j$ we have $i-1\geq j>r$ and so $\lambda_i=\mu_{i-1}=\widehat \lambda_i$. We thus get $|\lambda|>|\hat \lambda|$, while both are equal $n-m$ by assumption. 

Thus $j\leq r$. Since $i_j(\widehat \lambda)=\mu=i_j(\lambda)$ and  $\widehat \lambda_j=n-m-|\mu|=\lambda_j$, it holds that $\lambda=\widehat \lambda$. As $|\mu|<|\hat \lambda|=n-m$, eq. \eqref{eq_sum_coeffs} gives us $0=\sum_{j=1}^m \sum_{\lambda: i_j(\lambda)=\mu} c_\lambda=rc_{\widehat \lambda}$ and hence $c_{\widehat \lambda}=0$. Consequently, $\omega=0$.
\endproof

\proof[Proof of Theorem \ref{thm_abstract_valuations}]
Any $\rho \in  \calD\calT\calC$ must in fact lie in $\largewedge \mathcal A_1$. Indeed, since $\rho$ is horizontal, the homogeneous components $\rho_j \in \largewedge ^j \mathcal A$ must be linear combinations of wedges of elements of degree $1$ or greater. But the degree of $\rho_j$ equals the wedge degree $j$, and so $\rho_j \in \largewedge^j \mathcal A_1$.

Since $\rho \in \calD\calT\calC $, we have $d\rho_j=\partial \rho_{j+1}$ for all $j$. In particular, if $k$ is the degree of $\rho$, then  $d\rho_k=0$. By Lemma \ref{lem:unique_closed}, $\rho_k=c \gamma^k$ for some constant $c$. Then, by \eqref{eq_main_term_rhot}, $\rho-c k! \varrho_k \in \calD\calT\calC $ has degree strictly less than $k$, and so the statement follows by induction on the degree. 
\endproof

\section{Realization map}
\label{sec_realization}

Given a finite-dimensional Lie algebra $\g$, the space of differential forms $\Omega_\g:=\Omega(\g,\largewedge \g)$ is a DGGA. The universal property of the free DGGA $\calG$ gives rise to a map $\Re_\g:\calG \to \Omega_\g$ which we call the realization map. The main result of this section is the following injectivity property: The intersection of the kernels of the different realization maps $\Re_\g$, as $\g$ runs over all finite-dimensional Lie algebras, is trivial. In Sections \ref{sec_convolution_double_forms_g} and \ref{sec:abstract_convolution}, we will use this result to establish some statements about the structure of $\calG$ that appear hard to prove by direct arguments. 

The proof proceeds in two steps. In the first step, we use a realization map to some space $\Omega_U$ of differential forms on a vector space $U$ with values in the exterior algebra of the free Lie algebra it generates, and show its injectivity. To do so, we use the explicit constructions of the free DGLA and the graded wedge product in terms of tensor algebras and Dynkin projectors from Section \ref{sec:dgla}. The second step is a standard truncation argument to reduce from the infinite-dimensional free Lie algebra to a finite-dimensional Lie algebra.  

\subsection{Double forms on $\g$}
All differential forms will be considered with complex coefficients. Thus we will write $\Omega(X, V)$ for the differential forms on $X$ with values in $\C\otimes V$.

\begin{Definition}\label{def:double_forms}
	Let $\mathfrak g$ be a finite-dimensional Lie algebra. Define
	\begin{displaymath}
	  \Omega_\g:=\Omega(\g, \largewedge\g)=\Gamma(\mathfrak g, \C\otimes\largewedge\g^*\otimes \largewedge\g).
	\end{displaymath} 
\end{Definition}

We write $\Omega_\g^{\mathrm{eqv}}$ for the subspace consisting of all forms invariant under the group of automorphisms of $\g$. 

Since $\g$ acts on itself by the adjoint action, an element $x \in \g$ induces a vector field $x^\sharp$ on $\g$. Explicitly we have 
\begin{displaymath}
	x^\sharp|_\xi =[x,\xi], \quad \xi \in \mathfrak g. 
\end{displaymath}

We define  the space of $\ad$-invariant elements  
\begin{displaymath}
\Omega_\g^{\ad}=\{\tau\in\Omega_\g: \mathfrak L_{x^\sharp} \tau = \ad_x \tau, \quad x \in \mathfrak g\}.
\end{displaymath}

If $\g$ is the Lie algebra of a Lie group $G$, we also define the subspace $\Omega_\g^G$ of $\Ad G$-equivariant elements, i.e. elements $\tau$ for which  
\begin{equation} \label{eq_invariance_concrete_lie_group}
	\Ad_{g^{-1}}^* \otimes \Ad_g \tau=\tau
\end{equation}
for all $g \in G$, where $\Ad_{g^{-1}}^*$ acts on the form part of $\tau$ and $\Ad_g$ acts on $\largewedge \g$. Evidently $\Omega_\g^{\mathrm{eqv}}\subset\Omega_\g^G\subset \Omega_\g^{\ad}$. If $G$ is connected and simply connected and $\g$ is the Lie algebra of $G$, then the last inclusion is an equality.

\begin{Lemma}\label{lem:gerstenhaber_structure}
	 $\Omega_\g$ is in a natural way a differential graded Gerstenhaber algebra and $\Omega_\g^{\mathrm{eqv}},\Omega_\g^{G},\Omega_\g^{\ad}$ are differential graded Gerstenhaber subalgebras.
\end{Lemma}

\proof 
There is the obvious bigrading $\Omega_\g=\bigoplus_{d,k} \Omega^d(\g, \largewedge^k\g)$. For a homogeneous element $\rho \in  \Omega^d(\g, \largewedge^k\g)$ we set $\deg \rho=d$ and $\deg_w \rho=k$.

The product is the usual double wedge product, namely $(\tau_1 \otimes X_1) \wedge (\tau_2 \otimes X_2)=(\tau_1 \wedge \tau_2) \otimes (X_1 \wedge X_2)$, where $\tau_i$ are forms on $\g$ and $X_i \in\largewedge \g$. The differential is defined by $d(\tau \otimes X)=(d\tau) \otimes X$. The Lie bracket is defined by $[\tau_1 \otimes X_1,\tau_2 \otimes X_2]=(\tau_1 \wedge \tau_2) \otimes [X_1,X_2]$, where the bracket on the right hand side means the Schouten--Nijenhuis bracket in $\largewedge \g$. The verification of the axioms of a differential graded Gerstenhaber algebra is then straightforward. It is clear that the wedge product and the Lie bracket commute with automorphisms of the Lie algebra, rendering the subspaces in the statement subalgebras. 
\endproof 

\begin{Definition} 
The realization map $\Re_\g: \mathcal G \to \Omega_\g$ is the unique extension as a $\DGGA$-morphism of the $\GV$-morphism $\R \cdot \xi \to \Omega_\g, \xi \mapsto \id\in C^\infty(\g, \g)=\Omega^0(\g,\largewedge^1 \g) \subset \Omega_\g$. Its existence follows from Theorem \ref{thm_universal_property_dgga}. 
\end{Definition}
Since $\Omega_\g^{\mathrm{eqv}}$ is a $\DGGA$-subalgebra that contains the image of $\xi$, we see that the image of $\Re_\g$ is contained in $\Omega_\g^{\mathrm{eqv}}$.

\subsection{Injectivity of the realization map}

 Let $\g:=\g_f(U)$ be the free (ungraded) Lie algebra generated by a vector space $U$. 
 
This is an infinite-dimensional Lie algebra whenever $\dim U>1$. Since it is the direct sum of the finite-dimensional spaces $L_n(U)$, see Section \ref{sec:dgla}, we can use the colimit topology on $\g$. It induces a topology on each space $\largewedge^k \g$. A function $f:U \to \largewedge^k \g$ is called weakly smooth if for every $\eta$ in the topological dual space to $\largewedge^k \g$, the scalar-valued function $\langle \eta,f\rangle:U \to \R$ is smooth. Similarly we can define weakly smooth differential forms on $U$ with values in $\largewedge^k \g$.

 We let 
\begin{displaymath}
	\Omega_U:=\Omega(U,\largewedge \g)=\bigoplus_{d,k} \Omega^d(U,\largewedge^k \g)
\end{displaymath}
be the space of weakly smooth $\largewedge^{\mathrm{fin}} g$-valued forms on $U$.  

The wedge product, the Lie bracket and the differential are defined in the natural way, turning $\Omega_U$ into a differential graded Gerstenhaber algebra. 

By Theorem \ref{thm_universal_property_dgga} we can extend the map $\xi \mapsto \mathrm{id} \in C^\infty(U,U) \subset \Omega^0(U,\largewedge \g)$ to a $\DGGA$-morphism $\Re_U:\calG \to \Omega_U$.

Recall that $\calG^k=\largewedge^k \calA$, and $\calG_d \subset \calG$ is the space of degree $d$ elements. Denote $\calG^k_d=\calG_d \cap \calG^k$.

When no confusion can arise, we will write simply $\tau(X;Y_1,\ldots,Y_d)$ instead of $\Re_U \tau|_X(Y_1,\ldots,Y_d)$.

For instance,
\begin{displaymath}
 ([\xi,\gamma] \wedge \gamma \wedge \xi)(X;Y_1,Y_2)=[X,Y_1] \wedge Y_2 \wedge X-[X,Y_2] \wedge Y_1 \wedge X \in \largewedge^3 \g
\end{displaymath}   
and 
\begin{displaymath}
	([\gamma,\gamma] \wedge \gamma)(X;Y_1,Y_2,Y_3)= \sum_{\pi \in \mathcal S_3} \sgn(\pi) [Y_{\pi_1},Y_{\pi_2}] \wedge Y_{\pi_3} \in \largewedge^2 \g.
\end{displaymath}   

We now proceed to establish the claimed injectivity of the realization map, starting with a special case.

\begin{Lemma} \label{lemma_evaluation_k1}
The map $\Re_U: \calA_d \to \Omega^d(U,\g)$ is injective whenever $\dim U \geq d+1$.
\end{Lemma}

\proof 
Let $W:=\R\cdot \xi \oplus \R\cdot \gamma$ and let $T(W)$ be the tensor algebra generated by $W$. An element of $T(W)$ is a noncommutative polynomial in $\xi$ and $\gamma$. The degree of a monomial is the number of  $\gamma$ factors. Let $T(W)_d$ denote the subspace of degree $d$. We will also make use of the symbol degree on $T(W)$, which assigns degree $1$ to both $\xi$ and $\gamma$. The Lie bracket on $T(W)$ is defined by $[a,b]=a \otimes b - (-1)^{\deg a \deg b} b \otimes a$. 
Recall that $\calA$ is the free differential graded Lie algebra generated by $\xi$. It is the image of the Dynkin projector $\rho_W:T(W) \to T(W)$ that is inductively defined by $\rho_W(a):=a, \rho_W(a_1 \otimes \ldots \otimes a_m) :=\frac{m-1}{m} [a_1,\rho_W(a_2 \otimes \ldots \otimes a_m)]$  for $a,a_1,\ldots,a_m \in W$. 

Let $T(U)$ be the free tensor algebra generated by $U$, and $T(U)$ the tensor algebra of $U$ with the usual Lie bracket $[a,b]=a \otimes b-b \otimes a$. The image of the Dynkin projector $\rho_U:T(U) \to T(U)$ is the free Lie algebra $\g=\g_f(U)$. 

Let $\tau=\xi^{i_1} \gamma^{j_1} \xi^{i_2} \gamma^{j_2} \cdots \xi^{i_l}\gamma^{j_l} \in T(W)$ be a monomial of degree $d$. Given elements $\mathcal B=(X,Y_1,\ldots,Y_d)$ in $U$, we set 
\begin{align*}
	\ev_{\mathcal B}\tau=&\tau(X;Y_1,\ldots,Y_d)  :=	\sum_{\pi \in \mathcal S_d} \sgn(\pi) X^{i_1} Y_{\pi_1} Y_{\pi_2} \ldots  Y_{\pi_{j_1}}\\
	& \quad  X^{i_2} Y_{\pi_{j_1+1}} \ldots Y_{\pi_{j_1+j_2}} \ldots X^{i_l} Y_{\pi_{j_1+\ldots+j_{l-1}+1}} \ldots Y_{\pi_{j_1+\ldots+j_l}} \\&\in T(U),
\end{align*}
and extend this by linearity to $T(W)_d$. Note that for $\tau \in \calA_d \subset T(W)$, $\ev_{\mathcal B}\tau$ coincides with the previously defined $\tau(X;Y_1,\ldots,Y_d)=\Re_U \tau|_X(Y_1,\ldots,Y_d)$.

Obviously, $\tau(X;Y_1,\ldots,Y_k)$ is a polynomial in $X$, linear in each $Y_i$ and antisymmetric in the $Y_i$'s. The image of the map from $T(W)_d$ to $T(U)$ given by $\tau \mapsto \tau(X;Y_1,\ldots,Y_d)$ consists of all noncommutative polynomials in $X,Y_1,\ldots,Y_d$ that are multilinear and antisymmetric in the $Y_i$'s. When $X,Y_1,\ldots,Y_d$ are linearily independent, the map is obviously injective.

We claim that the following diagram commutes
\begin{equation*}
	\begin{tikzcd}
		\mathcal{A}_d \arrow[r, "\ev_{\mathcal{B}}"] & \g \\
		T(W)_d \arrow[r, "\ev_{\mathcal{B}}"'] \arrow[u, two heads, "\rho_W"] & T(U) \arrow[u, two heads, "\rho_U"']
	\end{tikzcd}
\end{equation*}

It suffices to prove this for noncommutative monomials $\tau$ in $\xi$ and $\gamma$ by induction over the symbol degree $m$. If $m=1$, then $\rho_U$ and $\rho_W$ are the identity, so there is nothing to prove. Otherwise, we can write $\tau=\xi \otimes \tau'$ or $\tau=\gamma \otimes \tau'$ with $\tau'$ of symbol degree $m-1$.   

In the first case we have 
\begin{align*}
	\rho_U (\tau(X;Y_1,\ldots,Y_d)) & =\rho_U(X \otimes \tau'(X;Y_1,\ldots,Y_d))\\
	& = \frac{m-1}{m} [X,\rho_U(\tau'(X;Y_1,\ldots,Y_d))]
\end{align*}
and 
\begin{align*}
	(\rho_W(\tau))&(X;Y_1,\ldots,Y_d) = \frac{m-1}{m}[\xi, \rho_W(\tau')] (X;Y_1,\ldots,Y_d) \\
	& = \frac{m-1}{m} (\xi \otimes \rho_W(\tau')-\rho_W(\tau') \otimes \xi) (X;Y_1,\ldots,Y_d)\\
	& = \frac{m-1}{m} \left(X \otimes \rho_W(\tau')(X;Y_1,\ldots,Y_d)-\rho_W(\tau')(X;Y_1,\ldots,Y_d) \otimes X\right)\\
	& = \frac{m-1}{m} [X,\rho_W(\tau')(X;Y_1,\ldots,Y_d)].
\end{align*}
By induction, these terms are equal.

In the second case, we have 
\begin{align*}
  \rho_U&(\tau(X;Y_1,\ldots,Y_d)) = \frac{1}{(d-1)!} \rho_U\left(\sum_\pi \sgn(\pi) Y_{\pi_1} \otimes \tau'(X;Y_{\pi_2},\ldots,Y_{\pi_d})\right)\\
  & = \frac{1}{(d-1)!} \frac{m-1}{m} \sum_\pi \sgn(\pi) [Y_{\pi_1},  \rho_U(\tau'(X;Y_{\pi_2},\ldots,Y_{\pi_d}))
\end{align*}
and 
\begin{align*}
	(\rho_W(\tau))&(X;Y_1,\ldots,Y_d)  = \frac{m-1}{m}[\gamma, \rho_W(\tau')] (X;Y_1,\ldots,Y_d) \\
	& = \frac{m-1}{m} (\gamma \otimes \rho_W(\tau')-(-1)^{d-1}\rho_W(\tau') \otimes \gamma) (X;Y_1,\ldots,Y_d)\\
	& = \frac{1}{(d-1)!} \frac{m-1}{m} \sum_\pi \sgn(\pi) \Large(Y_{\pi_1} \otimes \rho_W(\tau')(X;Y_{\pi_2},\ldots,Y_{\pi_d})\\
	& \quad -(-1)^{d-1} \rho_W(\tau')(X;Y_{\pi_1},\ldots,Y_{\pi_{d-1}}) \otimes Y_{\pi_d}\Large)\\
		& = \frac{1}{(d-1)!} \frac{m-1}{m} \sum_\pi \sgn(\pi) \Large(Y_{\pi_1} \otimes \rho_W(\tau')(X;Y_{\pi_2},\ldots,Y_{\pi_d})\\
	& \quad - \rho_W(\tau')(X;Y_{\pi_2},\ldots,Y_{\pi_{d}}) \otimes Y_{\pi_1}\Large)\\
	 & = \frac{1}{(d-1)!} \frac{m-1}{m} \sum_\pi \sgn(\pi) [Y_{\pi_1},  \rho_W(\tau')(X;Y_{\pi_2},\ldots,Y_{\pi_d}),
\end{align*} 
and again we have equality by the induction hypothesis. Thus the diagram commutes.

If the elements in $\mathcal B$ are linearly independent, then $\ev_{\mathcal B}:\calA_d \to \g$ is the restriction of the injective map $\ev_{\mathcal B}:T(W)_d \to T(U)$, and so it is injective itself.
\endproof

\begin{Proposition} \label{prop_injectivity_realization_dgga}
The map $\Re_U: \calG^k_d \to \Omega_U$ is injective whenever if $\dim U \geq d+1$.
\end{Proposition}

\proof 
We let $T(\calA)$ be the tensor algebra of $\calA$, $T(\calA)_d^k$ the subspace of elements of degree $d$ containing $k$ tensor factors. We let  $T(\g)$ be the tensor algebra of $\g$. For $\mathcal B=(X, Y_1, \dots, Y_d)$ we define a map $\ev_{\mathcal B}: T(\calA)^k_d \to T(\g)$ on multihomogeneous elements $\tau_1 \otimes \ldots \otimes \tau_k$ with $\deg \tau_i=d_i, \sum_i d_i=d$ by 
\begin{multline*}
	\ev_{\mathcal B}(\tau_1 \otimes \ldots \otimes \tau_k)= (\tau_1 \otimes \ldots \otimes \tau_d)(X;Y_1,\ldots,Y_d):= \frac{1}{d_1! \cdots d_k!} \sum_{\pi \in \mathcal S_d} \sgn(\pi)\\
	 \tau_1(X;Y_{\pi_1},\ldots,Y_{\pi_{d_1}}) \otimes \ldots \otimes \tau_k(X;Y_{\pi_{d_1+\ldots+d_{k-1}+1}},\ldots Y_{\pi_d}).
\end{multline*} 
Note that this equals the map from Lemma \ref{lemma_evaluation_k1} if $k=1$. Moreover, if $\tau \in \calG^k_d \subset T(\calA)^k_d$, then $\ev_{\mathcal B}\tau=\tau(X;Y_1,\ldots,Y_d)$.

We claim that this map is injective when the elements of $\mathcal B$ are linearly independent. Since $\ev_{\mathcal B}(\tau_i)$ is antisymmetric in the $Y$'s, we can remove the numerical factor and instead of summing over all permutations, we only sum over shuffles, i.e. those permutations that satisfy $\pi_1<\ldots<\pi_{d_1}, \pi_{d_1+1} < \ldots < \pi_{d_1+d_2}$ and so on. 

Note that the multidegrees are respected: the degree of the $i$-th factor of $\ev_{\mathcal B}(\tau_1 \otimes \ldots \otimes \tau_k)$, which is an element of $\g$, equals $d_i$. Here we use a grading of $\g$ such that $\deg X=0, \deg Y_j=1$ for all $j$. We may thus restrict our attention to multihomogeneous elements. 

Now for every multihomogeneous $\tau_1 \otimes \ldots \otimes \tau_k$ with the given degrees, the summands in the formula all live in different subspaces of $T(\g)$, hence to check injectivity of the map it is enough to check injectivity for a particular shuffle $\pi$, and we take $\pi$ to be the identity. But then our map is just given by 
\begin{displaymath}
	\tau_1 \otimes \ldots \otimes \tau_k \mapsto \tau_1(X;Y_1,\ldots,Y_{d_1}) \otimes \ldots \otimes  \tau_k(X;Y_{d_1+\ldots+d_{k-1}+1},\ldots,Y_d),
\end{displaymath}
and it is injective when the elements of $\mathcal B$ are linearly independent by Lemma \ref{lemma_evaluation_k1}. 

Recall the antisymmetrization $\pi_\calG:\calA^{\otimes k} \to \calA^{\otimes k}$ from Section \ref{subsec_graded_ext_power}:
\begin{displaymath}
	\pi_\calG(\tau_1 \otimes \ldots \otimes \tau_k):=\frac{1}{k!}\sum_{\pi \in \mathcal S_k} \chi(\pi) \tau_{\pi_1} \otimes \ldots \otimes \tau_{\pi_k}, \quad \tau_i \in \calA,
\end{displaymath}  
which is a projection with image $\largewedge^k \calA=\calG^k$. Similarly, let $\pi_\g:\g^{\otimes k} \to \g^{\otimes k}$ the usual antisymmetrization, i.e. 
\begin{displaymath}
	\pi_\g(X_1 \otimes \ldots \otimes X_k):=\frac{1}{k!} \sum_{\pi \in \mathcal S_k} \sgn(\pi) X_{\pi_1} \otimes \ldots \otimes X_{\pi_k}, \quad X_i \in \g.
\end{displaymath}   
It is a projection onto $\largewedge^k \g$.

We claim that the following diagram commutes.

\begin{equation*}
	\begin{tikzcd}
		\calG_d^k \arrow[r, "\ev_{\mathcal{B}}"] & \largewedge^k \g \\
		T(\calA)_d^k \arrow[r, "\ev_{\mathcal{B}}"'] \arrow[u, two heads, "\pi_{\calG}"] & T(\g)^k \arrow[u, two heads, "\pi_\g"']
	\end{tikzcd}
\end{equation*}

Let us show the claim in the case $d_1=3,d_2=1, k=2$. The ideas in the general case are the same. Note first that
\begin{align*}
	\pi_\g \circ \ev_{\mathcal B}(\tau_1 \otimes \tau_2) & = 
	\frac{1}{12} \sum_{\pi \in \mathcal S_4} \sgn(\pi) \large(\tau_1(X;Y_{\pi_1},Y_{\pi_2},Y_{\pi_3}) \otimes \tau_2(X;Y_{\pi_4})
	\\
	& \quad -\tau_2(X;Y_{\pi_4}) \otimes \tau_1(X;Y_{\pi_1},Y_{\pi_2},Y_{\pi_3})\large).
\end{align*}
The first term is $\frac12 \ev_{\mathcal B}(\tau_1 \otimes \tau_2)$. For the second term, we let $\sigma$ be the permutation $\sigma_1=4,\sigma_2=1,\sigma_3=2,\sigma_4=3$, i.e. the permutation that interchanges the block $(1)$ and the block $(234)$, and $\pi'=\pi \circ \sigma$. Then the second term can be written as 
\begin{displaymath}
	- \frac{1}{12} \sum_{\pi' \in \mathcal S_4} \sgn(\pi') \sgn(\sigma) \tau_2(X;Y_{\pi'_1}) \otimes  \tau_1(X;Y_{\pi'_2},Y_{\pi'_3},Y_{\pi'_4}).
\end{displaymath}
Since $\sgn(\sigma)=-1$, this term equals $\frac12 \ev_{\mathcal B}(\tau_2 \otimes \tau_1)$. Since $\pi_\calG(\tau_1 \otimes \tau_2)=\frac12(\tau_1 \otimes \tau_2+\tau_2 \otimes \tau_1)$, the claim in this case follows. In general, we use the fact that interchanging two blocks is an odd permutation if and only if both lengths are odd, and the length of a block corresponds to the degree of the corresponding $\tau_i$.  

From the claim and the fact that $\pi_\calG$ and $\pi_\g$ are projections it follows that if the elements of $\mathcal B$ are linearly independent, then the map $\ev_{\mathcal B} : \calG^k_d \to \largewedge^k \g$ is the restriction of the injective map $\ev_{\mathcal B}: T(\calA)_d^k  \to T(\g)^k$, and thus in itself an injective map. 
\endproof

\begin{Lemma} \label{lemma_extension_phi}
	Let $\pi:\g\to\h$ be a morphism of Lie algebras. Let $\phi \in \h^*$ and write also $\phi:\largewedge^k \h \to \largewedge^{k-1}\h$ for the induced map given by 
	\begin{displaymath}
		\phi(X_1 \wedge \ldots \wedge X_k)=\sum_{i=1}^k (-1)^{i+1} \phi(X_i) X_1 \wedge \ldots \wedge \hat X_i \wedge \ldots \wedge X_k.
	\end{displaymath}
	Then there is a map $\tilde \phi:\largewedge^k \g \to \largewedge^{k-1} \g$ such that the following diagram commutes:
	\begin{displaymath}
		\xymatrix{\largewedge^k \g \ar[r]^{\tilde \phi} \ar[d]^{\pi} & \largewedge^{k-1} \g \ar[d]^{\pi} \\ \largewedge^k \h \ar[r]^{\phi} & \largewedge^{k-1} \h}
	\end{displaymath}
\end{Lemma}

\proof
It suffices to put for $p_1,\ldots,p_k \in \g$
\begin{displaymath}
	\tilde \phi(p_1 \wedge \ldots \wedge p_k):=\sum_{i=1}^k (-1)^{i+1} \phi({\pi} p_i) p_1 \wedge \ldots \wedge \hat p_i \wedge \ldots \wedge p_k.
\end{displaymath}
\endproof

The free Lie algebra generated by a vector space $U$, denoted $\g_f(U)$, was constructed in Section \ref{sec:dgla} as the subalgebra of $T(U)$ given by $L(U)=\bigoplus_{n=1}^\infty L_n(U)$ where $L_1(U)=U, L_n(U)=[U,L_{n-1}(U)]$. In the rest of this section, $n$ will be referred to as the Lie degree. The Lie degree extends to $\largewedge^k \g$ in the natural way. 

The lower central series of a Lie algebra $\g$ is defined in the usual way by $\g_0:=\g, \g_n=[\g,\g_{n-1}]$. Then $\g_0 \supset \g_1 \supset \g_2 \supset \ldots$ and each $\g_n$ is a Lie ideal. If $\g_n=0$ for some $n$, the Lie algebra is called nilpotent of order $n$. 

If $\g=\g_f(U)$ then the quotient $\g/\g_n$ is called the free nilpotent Lie algebra of order $n$ generated by $U$. Note that this is a finite-dimensional Lie algebra.

Let $\h$ be the free nilpotent Lie algebra of some order generated by $U$. Then there are obvious quotient maps $\pi_U:\largewedge^k \g \to \largewedge^k \h$. 

\begin{Proposition} \label{prop_from_free_to_nilpotent}
	Let $\g$ be the free algebra generated by a vector space $U$. Let $Q \in \largewedge^k \g$ be of Lie degree $\leq N$. Let $\h$ be a free nilpotent Lie algebra of order at least $N+1$ generated by $U$. If ${\pi_U} Q \in \largewedge^k \h$ vanishes, then $Q=0$.
\end{Proposition}

\proof
The proof is by induction on $k$. The induction base $k=1$ is equivalent to saying that in $\h$, the only relations of degree $\leq N+1$ are those that also hold in $\g$, which is clear.  

Let us assume that $k>1$. We can write 
\begin{equation} \label{eq_q_pi}
	Q=\sum_{i=1}^n p_i \wedge R_i
\end{equation}
with $R_i \in \largewedge^{k-1}\g$ and linearly independent $p_i \in \g$.  By the induction base, the images ${\pi_U} p_i \in \h$ are still linearly independent. We thus find $\phi_j \in \h^*$ with $\phi_j({\pi_U} p_i)=\delta_{ij}, i=1,\ldots,n$.

We prove by induction that for every $m$ we can write 
\begin{equation} \label{eq_Q_with_several_p}
	Q=\sum_{i_1,\ldots,i_m=1}^n p_{i_1} \wedge \ldots \wedge p_{i_m} \wedge R_{i_1,\ldots,i_m},
\end{equation}
where $R_{i_1,\ldots,i_m} \in \largewedge^{k-m} \g$ is antisymmetric in its entries. For $m=1$ this is \eqref{eq_q_pi}. Suppose that we have shown \eqref{eq_Q_with_several_p} for some $m \geq 1$.  
We apply $\phi_j \circ {\pi_U}$ to this equation. Using the antisymmetry of the $R$'s and Lemma \ref{lemma_extension_phi}, this gives 
\begin{multline*}
	m \sum_{i_2,\ldots,i_m=1}^n {\pi_U} p_{i_2} \wedge \ldots \wedge {\pi_U} p_{i_m} \wedge {\pi_U} R_{j,i_2,\ldots,i_m} \\
	=  (-1)^{m+1}\sum_{i_1,\ldots,i_m=1}^n {\pi_U} p_{i_1} \wedge \ldots \wedge {\pi_U} p_{i_m} \wedge {\pi_U} \tilde \phi_j(R_{i_1,\ldots,i_m}),
\end{multline*}
which implies by induction hypothesis that
\begin{displaymath}
	m \sum_{i_2,\ldots,i_m=1}^n p_{i_2} \wedge \ldots \wedge p_{i_m} \wedge R_{j,i_2,\ldots,i_m}=  (-1)^{m+1}\sum_{i_1,\ldots,i_m=1}^n p_{i_1} \wedge \ldots \wedge p_{i_m} \wedge \tilde \phi_j(R_{i_1,\ldots,i_m}).
\end{displaymath} 
Plugging this into \eqref{eq_Q_with_several_p} we find
\begin{displaymath}
	Q=\frac{ (-1)^{m+1}}{m} \sum_{i_1,\ldots,i_{m+1}=1}^n p_{i_1} \wedge \ldots \wedge p_{i_{m+1}} \wedge \tilde \phi_{i_1}(R_{i_2,\ldots,i_{m+1}}).
\end{displaymath}
Writing $R_{i_1,\ldots,i_{m+1}}$ for the antisymmetrization of $\frac{ (-1)^{m+1}}{m} \tilde \phi_{i_1}(R_{i_2,\ldots,i_{m+1}})$ finishes the induction over $m$.

For $m=k$, we find that 
\begin{equation} \label{eq_Q_with_only_p}
	Q=\sum_{i_1,\ldots,i_k=1}^n \alpha_{i_1,\ldots,i_k} p_{i_1} \wedge \ldots \wedge p_{i_k}
\end{equation}
with $\alpha_{i_1,\ldots,i_m}$ a family of scalars that is antisymmetric in its indices. Applying $\phi_j \circ {\pi_U}$ gives 
\begin{displaymath}
	0=k \sum_{i_2,\ldots,i_k=1}^n \alpha_{j,i_2,\ldots,i_k}  {\pi_U} p_{i_2} \wedge \ldots \wedge {\pi_U} p_{i_k}.
\end{displaymath}
By our induction hypothesis we conclude that 
\begin{displaymath}
	0=\sum_{i_2,\ldots,i_k=1}^n \alpha_{j,i_2,\ldots,i_k}  p_{i_2} \wedge \ldots \wedge p_{i_k}
\end{displaymath}
for every $j$. Plugging this into \eqref{eq_Q_with_only_p} shows that $Q=0$.
\endproof

\begin{Theorem} \label{thm_injectivity_realization}
	Let $\tau \in \calG$. If $\Re_\h \tau=0$ for all real finite dimensional Lie algebras $\h$, then $\tau=0$.
\end{Theorem}

\proof
We may assume that $\tau$ is homogeneous for both gradings, with $\deg\tau=d$ and $\deg_w\tau=k$. Let $N$ be the symbol degree of $\tau$.
Let  $\g=\g_f(U)$ be the free Lie algebra generated by a vector space $U$ of dimension $d+1$. Let $\h$ be the corresponding free nilpotent Lie algebra of order $ N+1$. For a basis $(X,Y_1,\ldots,Y_d)$ of $U$ we have 

\begin{displaymath}
	\Re_\h \tau|_X(Y_1,\ldots,Y_d)=\pi_U(\Re_U \tau|_X(Y_1,\ldots,Y_d)).
\end{displaymath}
Since the Lie degree of $\Re_U \tau|_X(Y_1,\ldots,Y_d)$ is not larger than $N$, the claim follows from Propositions \ref{prop_injectivity_realization_dgga} and \ref{prop_from_free_to_nilpotent}.
\endproof

 \section{Convolution of double forms on a Lie algebra}
 \label{sec_convolution_double_forms_g}
 
  We mimic the construction from \cite{bernig_faifman_kotrbaty_part1} to construct a convolution product on $\Omega_\g$. It can be thought of as a deformation of the usual wedge product. The main technical result of the section is that this product satisfies a Leibniz rule with respect to the operator $D$.

\subsection{Defining convolution.} 

For a real finite dimensional Lie algebra $\g$, we consider $\Omega_\g=\Omega(\g, \largewedge\g)$ as in Definition \ref{def:double_forms}, and the subspace $\Omega_\g^{\ad}$ of $\ad$-invariant elements.

The following construction is taken, with some adaptions, from \cite[Section 4]{bernig_faifman_kotrbaty_part1}. 
 Let $e_1,\dots,e_n$ a basis of $\mathfrak g$ and $e_1^*,\ldots,e_n^*\in\g^*$ the dual basis. Recall that for an element $x \in \mathfrak g$, $x^\sharp$ denotes the induced vector field on $\mathfrak g$. Explicitly, $x^\sharp|_\xi=[x,\xi]$.
  
 \begin{Definition} \label{def_convolution_product_omega_g}
 	For $\tau,\zeta \in \Omega_\g^{\ad}$ and $r\geq0$ we set
 	\begin{displaymath}
 		\hat S_r(\tau \otimes \zeta)= \sum_{|K|=r} \iota_{e_K^\sharp} \tau \wedge \iota_{e_K^*} \zeta,
 	\end{displaymath}
 	where the sum runs over all sets $K=\{k_1,\dots,k_r\}\subset\{1,\dots,n\}$ and we define $\iota_{e_K^\sharp}=\iota_{e^\sharp_{k_1}}\circ\cdots\circ\iota_{e^\sharp_{i_r}}$ and similarly for $\iota_{e^*_K}$.
 \end{Definition}
 
 The definition of $\hat S_r$ depends neither on the ordering of the elements in $K$ nor on the choice of a basis. Note that $\deg \hat S_r(\tau \wedge \zeta)=\deg \tau+\deg \zeta-r,\deg_w \hat S_r(\tau \wedge \zeta)=\deg_w \tau+\deg_w \zeta-r$. Clearly $\hat S_r(\tau \otimes \zeta)=0$ if $r>\deg \tau$ or $r>\deg_w\zeta$.
 
 \begin{Proposition} \label{def_convolution_of_forms}
 	The \emph{convolution product} $\tau * \zeta \in \Omega_\g^{\ad}$ defined on homogeneous elements $\tau,\zeta \in \Omega_\g^{\ad}$ with $k_1=\deg \tau, l_1=\deg_w \tau, k_2=\deg \zeta, l_2=\deg_w\zeta$ by 
 	\begin{displaymath}
 		\tau * \zeta=\sum_r  \epsilon^r_{k_1,l_1,k_2,l_2} \hat S_r(\tau \otimes \zeta),
 	\end{displaymath}
 	where 
 	\begin{align*}
 		\epsilon^r_{k_1,l_1,k_2,l_2}=(-1)^{(k_2+l_2)(l_1+r)+r k_1},
 	\end{align*}
 	is well defined.
 \end{Proposition}

 \begin{proof}
 	
 We have to check that the definition makes sense. Letting $G$ be the simply connected Lie group with Lie algebra $\g$, we have $\Omega_\g^G=\Omega_\g^{\ad}$ so it remains to check that $\tau*\zeta \in \Omega_\g^{G}$ whenever $\tau,\zeta \in \Omega_\g^G$.

 	By \eqref{eq_invariance_concrete_lie_group} we have to show
 	\begin{displaymath}
 		\Ad_{g^{-1}}^* \otimes \Ad_g \hat S_r(\tau \otimes \zeta)=\hat S_r(\tilde\tau\otimes\tilde\zeta), \quad r\geq 0.
 	\end{displaymath}
 	To this end, observe first that for any index subset $K$ we have
 	\begin{align*}
 		\Ad_{g^{-1}}^* \circ\iota_{e_K^\sharp}=\iota_{(\Ad_g e_K)^\sharp}\circ(\Ad_{g^{-1}}^*)^*,
 	\end{align*}
 		\begin{align*}
 		\Ad_g\circ\iota_{e_K^*}=\iota_{(\Ad_g e_K)^*}\circ\Ad_g.
 	\end{align*}
 	Then the claim follows easily from the invariance of $\tau$ and $\zeta$ and the fact that the definition of $\hat S_r$ is independent of the basis of $\g$. 
\end{proof}
 
 \begin{Proposition}
 	\label{prop:associativity}
 	The convolution product on $\Omega_\g^{\ad}$ is associative.
 \end{Proposition}
 
 \proof
 The proof is verbatim the same as in \cite[Proposition 4.4]{bernig_faifman_kotrbaty_part1}.
 \endproof	
 
 \subsection{Leibniz rule}
  
Now we will show that the convolution in $\Omega_\g^{\ad}$ satisfies a version of the Leibniz rule. As in Section \ref{sec_free_dgga_one_generator} we set $\Deg X=\deg X-\deg_wX$ and define the operator $D:\Omega_\g^{\ad}\to \Omega_\g^{\ad}$ by
 \begin{displaymath}
 	D\tau=d\tau+(-1)^{\Deg \tau+1}\partial \tau.
 \end{displaymath}
 
 \begin{Proposition} 
 	\label{prop:Leibniz}
 	For any $\tau, \zeta \in \Omega_\g^{\ad}$ one has 
 	\begin{displaymath}
 		D(\tau * \zeta)= D\tau * \zeta + (-1)^{\Deg \tau} \tau * D\zeta.
 	\end{displaymath}
 \end{Proposition}
 
First we will prove a technical lemma.
\begin{Lemma}
 	For $\tau,\zeta \in \Omega_\g^{\ad}$ we have  
 	\begin{align} \label{eq_relation_d_partial}
 		\begin{split}
 			d \hat S_r(\tau \otimes \zeta) & =(-1)^r  \hat S_r(d\tau \otimes \zeta)+(-1)^{\deg \tau+r}  \hat S_r(\tau \otimes d\zeta) \\
 			& \quad +(-1)^{\deg_w \tau+1} \hat S_{r-1}(\partial \tau \otimes \zeta)+(-1)^r  \hat S_{r-1}(\tau \otimes \partial \zeta)\\
 			&\quad+(-1)^{\deg_w \tau}  \partial \hat S_{r-1}(\tau \otimes \zeta).
 		\end{split}
 	\end{align}
 \end{Lemma}
 
 \proof
 We set $k_1:=\deg \tau$, $l_1:=\deg_w\tau$, $k_2:=\deg \zeta$, and  $l_2:=\deg_w \zeta$. Define maps $[\bullet], \{\bullet\}: \Omega_\g^{\ad}  \otimes \Omega_\g^{\ad} \to \Omega_\g^{\ad}$	
 by 
 \begin{align*}
 	[\tau \otimes \zeta] & :=\sum_{a,b=1}^n [e_a,e_b] \wedge\iota_{e_a^*} \tau \wedge\iota_{e_b^*} \zeta,\\
 	\{\tau \otimes \zeta\} & :=\sum_{a,b=1}^n\iota_{[e_a,e_b]^\sharp} \tau \wedge\iota_{e_a^*}\iota_{e_b^*} \zeta.	
 \end{align*}
  
We claim that  
 \begin{align}
 	d \hat S_r(\tau \otimes \zeta) & =(-1)^r \hat S_r(d\tau \otimes \zeta)+(-1)^{k_1+r} \hat S_r(\tau \otimes d\zeta) \notag \\
 	& \quad - [S_{r-1}(\tau \otimes \zeta)] -\frac12  \{S_{r-2}(\tau \otimes \zeta)\}, \label{eq_dS}\\
 	\partial \hat S_r(\tau \otimes \zeta) & =  \hat S_r(\partial \tau \otimes \zeta)+ (-1)^{l_1+r}  \hat S_r(\tau \otimes \partial \zeta) \notag \\
 	& \quad  +(-1)^{l_1+1}  [{S_r}(\tau \otimes \zeta)]+ \frac{(-1)^{l_1+1}}{2} \{S_{r-1}(\tau \otimes \zeta)\}. \label{eq_partialS}
 \end{align}
 
 To prove the claim, we rewrite $\hat S_r$ as  
 \begin{displaymath}
 	\hat S_r(\tau \otimes \zeta)=\frac{1}{r!} \sum_{|I|=r}\iota_{e_I^\sharp} \tau \wedge\iota_{e_I^*} \zeta,
 \end{displaymath}
 where $I$ runs over all $r$-tuples. 
 
 To prove the first equation, we compute 
 \begin{displaymath}
 	d \hat S_r(\tau \otimes \zeta)= \frac{1}{r!} \sum_{|I|=r} d\iota_{e_I^\sharp} \tau \wedge\iota_{e_I^*} \zeta +(-1)^{k_1+r}\iota_{e_I^\sharp} \tau \wedge\iota_{e_I^*}d\zeta.
 \end{displaymath}
 The second summand is $(-1)^{k_2+r} \hat S_r(\tau \otimes d\zeta)$. 
 
 Consider now the first summand. Using the formula $[\mathcal L_X,\iota_Y]=\iota_{[X,Y]}$ and induction on $r$, we find that 
 \begin{align*}
 	d\iota_{e_{i_1}^\sharp} \cdots\iota_{e_{i_r}^\sharp} & = (-1)^r\iota_{e_{i_1}^\sharp} \cdots\iota_{e_{i_r}^\sharp} d\\
 	& \quad +\sum_{j=1}^r (-1)^{j+1}\iota_{e_{i_1}^\sharp}  \cdots \widehat{\iota_{e_{i_j}^\sharp}} \cdots\iota_{e_{i_r}^\sharp} \mathcal{L}_{e_{i_j}^\sharp} \\
 	& \quad + \sum_{1 \leq j_1<j_2 \leq r} (-1)^{r+j_1+j_2}\iota_{e_{i_1}^\sharp}  \cdots \widehat{\iota_{e_{i_{j_1}}^\sharp}} \cdots \widehat{\iota_{e_{i_{j_2}}^\sharp}} \cdots\iota_{e_{i_r}^\sharp}\iota_{[e_{i_{j_1}}^\sharp,e_{i_{j_2}}^\sharp]}.
 \end{align*}
 
 Replacing this into our formula for $d \hat S_r(\tau \otimes\zeta)$, we get three more terms, the first of which is $(-1)^r \hat S_r(d\tau \otimes \zeta)$. 
 
 For the second term, we use that $\tau$ is ${\ad}$-invariant, i.e.  $\mathcal L_{e_j^\sharp} \tau=\ad_{e_j} \tau$. 
 
 We write $I'=(i_2,\ldots,i_r)$ and $I''=(i_3,\ldots,i_r)$. For $j=1$ we obtain the term 
 \begin{displaymath}
 	\frac{1}{r!} \sum_{i_1=1}^n \sum_{I'} \ad_{e_{i_1}}\iota_{e_{I'}^\sharp} \tau \wedge\iota_{e_{i_1}^*}\iota_{e_{I'}^*} \zeta = - \frac{1}{r!} \sum_{I'} [\iota_{e_{I'}^\sharp} \tau \otimes\iota_{e_{I'}^*} \zeta]=- \frac{1}{r} [S_{r-1}(\tau \otimes \zeta)].
 \end{displaymath}
 For other values of $j$, we obtain the same sum, after some change of indices. 
 
 The last term only appears if $r \geq 2$. We first take $j_1=1,j_2=2$. Then we obtain a term 
 \begin{displaymath}
 	- \frac{1}{r!} \sum_{i_1,i_2=1}^n \sum_{I''} \iota_{[e_{i_1}^\sharp,e_{i_2}^\sharp]}\iota_{e_{I''}^\sharp} \tau \wedge\iota_{e_{i_1}^*}\iota_{e_{i_2}^*}\iota_{e_{I''}^*} \zeta=-\frac{1}{r(r-1)} \{S_{r-2}(\tau \otimes \zeta)\}.
 \end{displaymath}
 For other values of $j_1<j_2$, we obtain the same term, after some change of indices. Since there are $\frac{r(r-1)}{2}$ such pairs, we obtain \eqref{eq_dS}.
 
 For the boundary operator we compute 
 \begin{align*}
 	r! \partial \hat S_r(\tau \otimes \zeta) & = \frac12 \sum_{a,b} \sum_{|I|=r} [e_a,e_b] \wedge\iota_{e_b^*}\iota_{e_a^*} (\iota_{e_I^\sharp} \tau \wedge\iota_{e_I^*}\zeta) \\
 	& = \sum_{|I|=r} \iota_{e_I^\sharp} \left( \frac12 \sum_{a,b}[e_a,e_b] \wedge\iota_{e_b^*}\iota_{e_a^*} \tau\right) \wedge\iota_{e_I^*}\zeta \\
 	& \quad + \frac12 \sum_{a,b,|I|=r} [e_a,e_b] \wedge (-1)^{l_1+1}\iota_{e_a^*}\iota_{e_I^\sharp} \tau \wedge\iota_{e_b^*}i_{e_I^*} \zeta \\
 	& \quad +  \frac12 \sum_{a,b,|I|=r} [e_a,e_b] (-1)^{l_1}\iota_{e_b^*}\iota_{e_I^\sharp} \tau \wedge\iota_{e_a^*}\iota_{e_I^*}\zeta\\
 	& \quad + (-1)^{l_1} \frac12 \sum_{a,b,|I|=r}\iota_{e_I^\sharp}\tau \wedge [e_a,e_b] \wedge\iota_{e_I^*}\iota_{e_b^*}\iota_{e_a^*}\zeta\\
 	& = r! \hat S_r(\partial \tau \otimes \zeta)+(-1)^{l_1+1} r! [S_r(\tau \otimes \zeta)] \\
 	& \quad + (-1)^{l_1+r} \frac12 \sum_{a,b,|I|=r}\iota_{e_I^\sharp} \tau \wedge\iota_{e_I^*}([e_a,e_b] \wedge \iota_{e_b^*}\iota_{e_a^*}\zeta)\\
 	& \quad + \frac{(-1)^{l_1+r+1} r}{2} \sum_{a,b,i_r,|I'|=r-1}\iota_{e_{I'}^\sharp}\iota_{e_{i_r}^\sharp} \tau \wedge\iota_{e_{i_r}^*} [e_a,e_b] \cdot\iota_{e_{I'}^*}\iota_{e_b^*}\iota_{e_a^*}\zeta\\  
 	& = r! \hat S_r(\partial \tilde\tau \otimes \zeta)+(-1)^{l_1+1} r! [S_r(\tau \otimes \zeta)]+ (-1)^{l_1+r} r! \hat S_r(\tau \otimes \partial \zeta)\\
 	& \quad + \frac{(-1)^{l_1+r+1} r}{2} \sum_{a,b,|I'|=r-1}\iota_{e_{I'}^\sharp}\iota_{[e_a,e_b]^\sharp}\tau \wedge\iota_{e_{I'}^*}\iota_{e_b^*}\iota_{e_a^*}\zeta\\  
 	& = r! \hat S_r(\partial \tau \otimes \zeta)+(-1)^{l_1+1} r! [S_r(\tau \otimes \zeta)] + (-1)^{l_1+r} r! \hat S_r(\tau \otimes \partial \zeta)\\
 	& \quad +\frac{(-1)^{l_1+1}r}{2} (r-1)! \{S_{r-1}(\tau \otimes \zeta)\}.  
 \end{align*}
 This proves \eqref{eq_partialS}. Combining \eqref{eq_dS} and \eqref{eq_partialS} we obtain \eqref{eq_relation_d_partial}.
 \endproof
 
 \begin{proof}[Proof of Proposition \ref{prop:Leibniz}]
 	From \eqref{eq_relation_d_partial} we infer that 
 	\begin{align*}
 		D(\tau * \zeta) & = d(\tau * \zeta)+(-1)^{k_1+l_1+k_2+l_2+1} \partial(\tau * \zeta)\\
 		&  =	\sum_r \epsilon^r_{k_1,l_1,k_2,l_2} (-1)^r \hat S_r(d\tau \otimes \zeta) \\
 		&\quad+\sum_{r} (-1)^{k_1+r} \epsilon^r_{k_1,l_1,k_2,l_2} \hat S_r(\tau \otimes d\zeta) \\
 		& \quad + \sum_{r} (-1)^{l_1} \epsilon^r_{k_1,l_1,k_2,l_2} \partial \hat S_{r-1}(\tau \otimes \zeta)\\
 		& \quad +\sum_{r} (-1)^{l_1+1} \epsilon^r_{k_1,l_1,k_2,l_2} \hat S_{r-1}(\partial \tau \otimes \zeta)\\
 		& \quad +\sum_{r}(-1)^r \epsilon^r_{k_1,l_1,k_2,l_2} \hat S_{r-1}(\tau \otimes \partial \zeta)\\
 		& \quad +\sum_r (-1)^{k_1+l_1+k_2+l_2+1}\epsilon^r_{k_1,l_1,k_2,l_2} \partial \hat S_r(\tau \otimes \zeta)\\
 		& = \sum_r ((-1)^{l_1}\epsilon^r_{k_1,l_1,k_2,l_2} +(-1)^{k_1+l_1+k_2+l_2+1} \epsilon^{r+1}_{k_1,l_1,k_2,l_2}) \partial \hat S_r(\tau \otimes \zeta)\\
 		& \quad +\sum_r \epsilon^r_{k_1+1,l_1,k_2,l_2} \hat S_r(d\tau \otimes \zeta) \\ 
 		& \quad + (-1)^{k_1+l_1+1} \sum_r \epsilon^r_{k_1,l_1-1,k_2,l_2} \hat S_r(\partial \tau \otimes \zeta) \\
 		& \quad + (-1)^{k_1+l_1} \sum_r \epsilon^r_{k_1,l_1,k_2+1,l_2} \hat S_r(\tau \otimes d\zeta) \\
 		& \quad + (-1)^{k_1+l_1+k_2+l_2+1} \sum_r \epsilon^r_{k_1,l_1,k_2,l_2-1} \hat S_r(\tau \otimes \partial \zeta) \\
 		& = \left(d\tau +(-1)^{\Deg \tau+1}\partial \tau\right) * \zeta+(-1)^{\Deg \tau} \tau * \left(d\zeta+(-1)^{\Deg \zeta+1}\partial \zeta \right)\\
 		& = D\tau * \zeta+(-1)^{\Deg \tau} * D\zeta.
 	\end{align*}
 \end{proof}

 Let us finish this section with a remark. If we put a factor $t^r$ in front of $\hat S_r$ in Definition \ref{def_convolution_product_omega_g}, we obtain a $1$-parameter family of products $*_t, t \in \R$ satisfying a Leibniz rule with respect to the operators $D_t=d+ t \partial$. The case $t=0$ corresponds, up to a sign, to the usual wedge product. We can thus think of $(\Omega_\g^{\mathrm{ad}},*_t,D_t)$ as a deformation of the differential graded Gerstenhaber algebra $(\Omega_\g^{\mathrm{ad}},\wedge,d)$. 

\section{Abstract convolution}\label{sec:abstract_convolution}

We construct a convolution product on the free differential graded Gerstenhaber algebra generated by one element. The approach from Section \ref{sec_convolution_double_forms_g} does not work directly, since we no longer have a finite basis. Instead we will rewrite the operations $\hat S_r$ in a form that does not use any basis. The idea is that we choose $r$ factors from the second form and plug them into some $\gamma$-terms of the first form. To make this work requires some technical care. It is not obvious, and appears hard to prove directly, that the convolution product obtained in this way satisfies associativity and the Leibniz rule. Instead, we show that the operators $\hat S_r$ and hence the convolution are compatible with the realization map and use the injectivity result from Section \ref{sec_realization} and the associativity and Leibniz rule on the space $(\Omega_\g,*)$ from Section \ref{sec_convolution_double_forms_g} to deduce the corresponding statements for $(\calG,*)$. 

We then show that the space $\DTC$ is closed under convolution and, using the explicit description of this space from Section \ref{sec_free_dgga_one_generator}, compute its algebra structure. Finally we define the space $\TC\subset \calG$ of abstract tube coefficients, and introduce a slight modification of the convolution product on $\calG$, which we call valuation convolution and denote by $\vee$, turning $\TC$ into an algebra. 

\subsection{Statement of the theorem}	
	  
Recall the DGV $V=\R \cdot \xi \oplus \R \cdot \gamma$, generating the DGLA $\calA$ and DGGA $\calG=\largewedge \calA$.

The differential graded Gerstenhaber algebra $\calG$ is a universal model for $\largewedge\g$-valued differential forms on $\g$, in the sense that for every Lie algebra $\g$ we have the realization map $\Re_\g:\calG \to \Omega_\g$. This map is a morphism of DGGAs. However, on $\Omega_\g$  there is also the convolution product, defined in Section \ref{sec_convolution_double_forms_g}. Our next theorem shows that the convolution product can be defined on the universal model as well. 
 
\begin{Theorem} \label{thm_existence_convolution_DGGA}
	There exists a unique convolution product $*$ on $\calG$ such that for every Lie algebra $\g$, the realization map $\Re:(\calG,\ast)\to (\Omega_\g,\ast)$ is a morphism of algebras. It is associative, satisfies $\Deg(\tau\ast \zeta)=\Deg(\tau)+\Deg(\zeta)$ and the Leibniz rule:
	\begin{displaymath}
		D(\tau * \zeta)=D\tau * \zeta+(-1)^{\Deg \tau} \tau * D\zeta.
	\end{displaymath}
\end{Theorem}

The uniqueness of the convolution product follows immediately from the injectivity of the realization maps, see Theorem \ref{thm_injectivity_realization}. The construction is more involved and is the topic of this section.

\subsection{Contraction}

Let $T(V)$ be the tensor algebra of $V$, and $T(V)_k$ be the component of degree $k$ (i.e. containing precisely $k$ copies of $V_1=\Span\{\gamma\}$). Recalling that $\calA\subset T(V)$, we have $\calA_k=T(V)_k\cap \calA$. The Dynkin projector 
\begin{align*}
	\rho: T(V) & \to \calA\\
	a_1 \otimes \cdots \otimes a_n & \mapsto \frac{1}{n} [a_1,[a_2,\ldots,[a_{n-1},a_n]]]
\end{align*}
induces a map 
\begin{displaymath}
	\largewedge^k\rho: \largewedge^k T(V)\to\largewedge^k \calA,
\end{displaymath}
and thus also a map
\begin{displaymath}
	\overline\rho:=\bigoplus_{k=0}^\infty \largewedge^k\rho: \largewedge T(V) \to \largewedge\calA=\calG.
\end{displaymath}

Define the linear map 

\begin{displaymath}
	r:\calA^{\otimes k} \otimes T(T(V))_k \to T(T(V))
\end{displaymath}
by 
\begin{displaymath}
	r:\left(a_1 \otimes \cdots \otimes a_k\right) \otimes  P(\gamma[k], \xi[n-k]) \mapsto P(a_1,\dots, a_k, \xi[n-k]).
\end{displaymath}
Here $P(\gamma[k], \xi[n-k])$ is an arbitrary element in $T(T(V))$ which has degree $k$, that is $\gamma$ appears $k$ times in each monomial summand of $P$, and we replace the $i$-th $\gamma$ in each monomial with $a_i$, for all $i=1,\dots, k$. 

Let $p: T(T(V))\to\largewedge T(V)\subset T(T(V))$ be the antisymmetrization, see \eqref{eq_graded_antisymmetrization}. We then define the map $\iota$ as the following composition:

\begin{displaymath}
	\iota: \largewedge^k \calA \otimes \calG_k \hookrightarrow \calA^{\otimes k} \otimes T(T(V))_k \stackrel{r}{\to}  T(T(V)) \stackrel{p}{\to} \largewedge T(V) \stackrel{\bar \rho}{\to} \calG.
\end{displaymath}

We extend this to a map  
\begin{displaymath}
	\iota:\largewedge^k \calA \otimes \calG_l \to \calG
\end{displaymath}
for all $l$ by putting for $b \in \calG_l$
\begin{displaymath}
	\iota_{a_1 \wedge \ldots \wedge a_k}b:=
	\begin{cases} 
		0 & l<k \\ \iota_{a_1 \wedge \ldots a_k \wedge \gamma \wedge \ldots \wedge \gamma}b & l \geq k.
	\end{cases}
\end{displaymath}
Extending by linearity, we obtain the \emph{contraction map} $\iota:\largewedge^k\calA\otimes \calG\to \calG$, for all $k$.

\subsection{The maps $S_r$}
For $r\geq 0$, define the disintegration function $f_r:\calG\to \largewedge^r \mathcal A \otimes \calG$ as the unique linear map $f_r:\largewedge^\bullet\mathcal A\to \largewedge^r\mathcal A\otimes \largewedge^{\bullet-r}\mathcal A$ satisfying
\begin{align*}&f_r(x_1\wedge\dots\wedge x_n)=\frac{1}{r!(n-r)!}\cdot\\
	&\cdot\sum_{\sigma\in  \mathcal S_n} (-1)^{\sgn\sigma+\chi(\sigma)} x_{\sigma 1}\wedge\dots \wedge x_{\sigma r}\otimes x_{\sigma (r+1)}\wedge\dots\wedge x_{\sigma n},
\end{align*}
where \[\chi(\sigma)=\sum_{\scriptsize\begin{array}{cc}i<j\\\sigma i>\sigma j\end{array}}\deg x_i\deg x_j.\] Note that $f_r$ is well-defined by eq. \eqref{eq_supersymmetry}. We will need the factor switching map 
\[\beta_r: \mathcal G \otimes \largewedge^r\mathcal A \to \largewedge^r\mathcal A\otimes\mathcal G\]
which on pairs of $\deg$-homogeneous elements is given by
\[\beta_r(\tau\otimes\zeta)=(-1)^{(\deg\tau-r)\deg \zeta} \zeta\otimes\tau.\]
The reason behind the sign in $\beta_r$ will become evident in the course of the proof of Proposition \ref{prop:realization_morphism}. Essentially, both factors will correspond to differential forms, with $\tau$ corresponding to a form of degree $\deg\tau-r$ following the substitution of $r$ vectors into the realization of $\tau$, while $\zeta$ corresponds to a form of degree $\deg\zeta $.

For $r \geq 0$ we define the operators $S_r:\calG \otimes  \calG\to \calG \otimes  \calG$ by 
\begin{align*}\label{eq:shift_abstract}
	&\xymatrix@=2cm{\calG\otimes \calG\ar[r]^{\mathrm{Id}\otimes f_r} & \calG\otimes \largewedge^r\mathcal A\otimes \calG \ar[r]^{\mathrm{Id}\otimes\largewedge^r\ad_{-\xi}\otimes\mathrm{Id}}&  \calG\otimes \largewedge^r\mathcal A\otimes \calG}
	\\&	\xymatrix@=2cm{\ar[r]^{\beta_r\otimes \mathrm{Id}}& \largewedge^r\mathcal A\otimes \calG\otimes \calG \ar[r]^{\iota\otimes\mathrm{Id}}& \calG\otimes\calG}.
\end{align*}

We observe that $S_r$ decreases by $r$ the wedge degree of the right factor while leaving the wedge degree of the left factor invariant; and decreases by $r$ the form degree of the tensor product. 

We write $\widehat S_r:\calG \otimes \calG \to \calG$ for the composition of $S_r$ with the wedge product.

\subsection{Convolution}

\begin{Definition}
The convolution of two homogeneous elements $\tau,\zeta \in \mathcal G$ with $\deg \tau=k_1,\deg_w\tau=l_1,\deg \zeta=k_2,\deg_w\zeta=l_2$ is defined by
\begin{displaymath}
 \tau \ast \zeta:=\sum_{r=0}^\infty \epsilon^r_{k_1,l_1,k_2,l_2} \hat S_r(\tau\otimes\zeta),
\end{displaymath}
where 
\begin{displaymath}
	\epsilon^r_{k_1,l_1,k_2,l_2}=(-1)^{(k_2+l_2)(l_1+r)+r k_1}.
\end{displaymath}
\end{Definition}

We now verify that our two definitions of convolution are compatible through the realization map.
\begin{Proposition}\label{prop:realization_morphism}
	Let $\tau,\zeta \in \mathcal G$ and $\mathfrak g$ a finite dimensional Lie algebra. Then 
	\begin{displaymath}
		\Re_\g(\tau * \zeta)=\Re_\g \tau * \Re_\g \zeta.
	\end{displaymath}
\end{Proposition}
For the proof we will need the following general construction. Assume $\alpha$ is a $k$-form on $\g$, and $\beta_1,\dots, \beta_k$ are $\g$-valued forms on $\g$, of degrees $b_1,\dots, b_k$, respectively. Setting  $b:=b_1+\dots+b_k$, we define the composition form 
\[ \omega=\alpha(\beta_1,\dots, \beta_k)\] to be the $b$-form on $\g$ given by

\begin{align*}&\omega(v_1, \dots, v_{b}):=
	\\&  \frac{1}{b!}\sum_{\sigma\in \Sigma_{b}} (-1)^{\sgn \sigma}\alpha(\beta_1(v_{\sigma1}, \dots, v_{\sigma b_1}),\dots, \beta_k(v_{\sigma (b-b_k+1)}, \dots, v_{\sigma b}) ). \end{align*}

\proof 
It suffices to prove that for $\tau,\zeta \in \mathcal G$ we have 
\begin{equation}\label{eq:claimed_Sr}
	\Re_\g\hat S_r(\tau \otimes \zeta)=\hat S_r(\Re_\g \tau \otimes \Re_\g \zeta).
\end{equation}
We may further assume $\zeta=\zeta_1\wedge\dots\wedge \zeta_m$, $\zeta_j\in\calA$. Put $\widetilde\tau=\Re_\g \tau$, $\widetilde\zeta=\Re_\g\zeta$. Consider the summand of $ \widehat S_r(\widetilde\tau\otimes\widetilde\zeta)$ given by 
\begin{displaymath}
\iota_{e_1^\sharp \wedge \dots \wedge e_r^\sharp}\widetilde \tau \wedge \iota_{e_1^*\wedge\dots\wedge e_r^*}\widetilde \zeta,
\end{displaymath}
where $e_i^\sharp=[e_i,\xi]$. 

Applying $ \iota_{e_1^*\wedge\dots\wedge e_r^*}$ to $\widetilde\zeta$, one has summands corresponding to all choices of $r$ wedge factors of $\widetilde \zeta$. Consider the summands corresponding to $\widetilde\zeta_1,\dots,\widetilde\zeta_r$. Those are

\begin{align*} S=&\frac{1}{r!} \sum_{\sigma\in \Sigma_r}(-1)^{\sgn \sigma}\iota_{e_1^\sharp \wedge \dots \wedge e_r^\sharp}\widetilde \tau\wedge e_1^*(\widetilde\zeta_{\sigma 1})\wedge\dots\wedge e_r^*(\widetilde\zeta_{\sigma r})\wedge \widetilde\zeta_{r+1}\wedge\dots\wedge \widetilde\zeta_n 
\\  &=\epsilon\frac{1}{r!} \sum_{\sigma\in \Sigma_r}(-1)^{\sgn \sigma}  e_1^*(\widetilde\zeta_{\sigma 1})\wedge\dots\wedge e_r^*(\widetilde\zeta_{\sigma r})\wedge \iota_{e_1^\sharp \wedge\dots\wedge e_r^\sharp}\widetilde \tau\wedge \widetilde\zeta_{r+1}\wedge\dots\wedge \widetilde\zeta_m,
\end{align*}

where $\epsilon=(-1)^{(\deg\widetilde \tau -r) \sum_{j=1}^r \deg\widetilde\zeta_j}$. By definition,
\begin{displaymath}
	\iota_{e_1^\sharp \wedge\dots\wedge e_r^\sharp}\widetilde \tau= \widetilde \tau( [e_1,\xi], \dots, [e_r, \xi], \bullet[\deg \widetilde\tau -r] ).
\end{displaymath}

Now $e_j^*(\widetilde\zeta_{\sigma j})$  is a $\C$-valued form, and so $e_j^*(\widetilde\zeta_{\sigma j})[e_j,\xi]= [e_j^*(\widetilde\zeta_{\sigma j})e_j, \xi]$ is a $\g$-valued form on $\g$. We may thus present $S$ as a sum of composition forms wedged by the remaining $\widetilde\zeta_j$:

\begin{align*} S=&\epsilon\frac{1}{r!} \sum_{\sigma\in \Sigma_r}(-1)^{\sgn \sigma}\widetilde \tau( [e_1^*(\widetilde\zeta_{\sigma 1})e_1,\xi], \dots, [e_r^*(\widetilde\zeta_{\sigma r})e_r, \xi], \bullet[\deg \widetilde\tau -r])\wedge \widetilde\zeta_{r+1}\wedge\dots\wedge \widetilde\zeta_m.
\end{align*}

	Now summing over all $r$-tuples of basis elements, by the anti-symmetry of $\tau$ we may  also sum over $\sigma$ and thus obtain the form  
\begin{displaymath}
S=	\epsilon\widetilde \tau([\widetilde \zeta_1,\xi],\dots, [\widetilde \zeta_r,\xi], \bullet [\deg \widetilde\tau-r] )\wedge \widetilde\zeta_{r+1}\wedge\dots\wedge \widetilde\zeta_m,
\end{displaymath}
which transparently corresponds under the realization map to the abstract convolution, namely to the summands with the disintegration function splitting off the first $k$ wedge factors of $\zeta$. 

As both sides of eq. \eqref{eq:claimed_Sr} transform by the same sign under permutations of the $\zeta_j$ factors,  this completes the proof.
\endproof

\proof[Proof of Theorem \ref{thm_existence_convolution_DGGA}] 
Associativity of the convolution follows at once from Proposition \ref{prop:associativity} and Theorem \ref{thm_injectivity_realization}. That $\Deg$ is additive is immediate from definitions. Since the realization map is a $\DGGA$-morphism, it commutes with $D$ and then the Leibniz rule follows from Proposition \ref{prop:Leibniz} and Theorem \ref{thm_injectivity_realization}.
\endproof

\begin{Corollary}\label{cor:valuations_convolution_algebra}
	The subspaces $\mathcal Z$, $\mathcal H$ and $\calD\calT\calC $ are closed under convolution.
\end{Corollary}
\begin{proof}
	For $\mathcal Z$, this follows at once from the Leibniz law for convolution. Inspecting the contraction map, it is clear that $\iota( (\largewedge ^k\mathcal A\cap \mathcal H)\otimes  \mathcal H)\subset \mathcal H$, and consequently $\mathcal H$ is a convolution subalgebra. For $\calD\calT\calC$ this now follows by the additivity of $\Deg$. 
\end{proof}

\begin{Proposition}\label{prop:hardlefschetz_convolution}
	It holds that 
	\begin{displaymath}
	  \varrho_m * \gamma =(m+1)\varrho_{m+1}.
	\end{displaymath}
\end{Proposition}

\proof

We have by definition 
\begin{align*}
	\hat S_1(\zeta_\epsilon \otimes \gamma) &=- \sum_{i=0}^\infty\frac{1}{(i+1)!} \hat S_1([\xi^i\gamma] \otimes  \gamma)\epsilon^{i+1}\\
	&=\sum_{i=0}^\infty\frac{1}{(i+1)!}[\xi^{i+1}\gamma]\epsilon^{i+1}\\
	&=-[\xi,\zeta_\epsilon].
\end{align*}

For $\tau_1,\tau_2 \in \calG$ one checks that 
\begin{displaymath}
	\hat S_1((\tau_1 \wedge \tau_2) \otimes \gamma)=(-1)^{\deg \tau_2} \hat S_1(\tau_1 \otimes \gamma) \wedge \tau_2+(-1)^{\deg \tau_1} \tau_1 \wedge \hat S_1(\tau_2 \otimes \gamma).
\end{displaymath}

By induction we deduce that 
\begin{displaymath}
	\hat S_1(\zeta_\epsilon^i \otimes \gamma)=(-1)^{i-1}i \zeta_\epsilon^{i-1} \wedge \hat S_1(\zeta_\epsilon \otimes \gamma).
\end{displaymath}

The definition of the convolution product gives, using $\deg \zeta_\epsilon^i=i$, 
\begin{align*}
	\widehat\exp(-\zeta_\epsilon) * \gamma & = \widehat\exp(-\zeta_\epsilon) \wedge \gamma+\sum_{i=0}^\infty \frac{(-1)^i}{i!} \hat S_1((-\zeta_\epsilon)^i \otimes \gamma)\\
	& =\widehat \exp(-\zeta_\epsilon) \wedge \gamma+\sum_{i=1}^\infty \frac{(-1)^{i-1}}{(i-1)!} \zeta_\epsilon^{i-1} \wedge \hat S_1(\zeta_\epsilon \otimes \gamma)\\
	& = \widehat\exp(-\zeta_\epsilon) \wedge \gamma-\widehat\exp(-\zeta_\epsilon) \wedge [\xi,\zeta_\epsilon]\\
	& = \frac{d}{d\epsilon} \widehat\exp(-\zeta_\epsilon),
\end{align*}
where in the last line we used eq. \eqref{eq:zeta_dt}. Comparing the coefficient of $\epsilon^m$ on both sides gives the claimed equality $\varrho_m * \gamma=(m+1)\varrho_{m+1}$. 
\endproof

\begin{Corollary}\label{cor:dgla_convolution_table}
	It holds that  $\varrho_m=\frac{1}{m!}\varrho_1^{\ast m}$ and $\varrho_m\ast \varrho_n={m+n \choose n} \varrho_{m+n}$. 
\end{Corollary}
\proof
Since the convolution product on $\calG$ is associative, $\varrho_1^{\ast m}$ is well defined.
From Proposition \ref{prop:hardlefschetz_convolution} we deduce that $\varrho_m\ast \varrho_1=(m+1)\varrho_{m+1}$ for all $m$. 
It follows by induction on $m$ that $\varrho_m=\frac{1}{m!}\varrho_1^{\ast m}$. The second equality is an immediate consequence. The last statement follows from Theorem \ref{thm_abstract_valuations}.
\endproof

\subsection{The abstract valuation convolution}

\begin{Definition}
	The valuation convolution on $\mathcal G$ is given by
	\[\kappa_1\vee\kappa_2= \kappa_1\ast D\kappa_2.\]
\end{Definition} 
\begin{Lemma}
	The valuation convolution is associative. 
\end{Lemma}
\begin{proof}
 We use Theorem \ref{thm_existence_convolution_DGGA}. By the Leibniz rule for convolution, 
	 \begin{displaymath}
	 	\kappa_1\vee (\kappa_2\vee\kappa_3)=\kappa_1\ast D(\kappa_2\ast D\kappa_3)=\kappa_1\ast (D\kappa_2\ast D\kappa_3),
	\end{displaymath}
	 while
	 \begin{displaymath}
	 	(\kappa_1\vee \kappa_2) \vee \kappa_3=(\kappa_1 \ast D\kappa_2) \ast D\kappa_3
	 \end{displaymath}
	 by definition. Equality now follows from associativity of convolution.
\end{proof}

\begin{Lemma}\label{lem:D_horizontal}
	For all $\tau\in\mathcal H$, $D\tau \in \mathcal H$.
\end{Lemma}
\proof
For $x,y\in\calA$ it holds that $\deg [x,y]\geq 1$ when $\deg x, \deg y\geq1$, in particular $\partial (\mathcal H)\subset \mathcal H$. It is also clear that $d(\mathcal H)\subset \mathcal H$, implying the statement.
\endproof

We identify a subspace of $\calG$ which under the realization map  $\Re^S_\g$ from Section \ref{sec:tube_coefficients} below gives rise to $(n-1)$-forms over the sphere bundle of a Lie group equipped with a bi-invariant metric, with the property that their differential is a vertical form. 
\begin{Definition}
	The space of \emph{abstract tube coefficients} is 
	$$\calT\calC=\{\kappa \in \calG(-1): D\kappa\in\calH\}.$$ 
\end{Definition}

There is a simple relation between the spaces $\TC$ and $\DTC$ as follows. 
\begin{Proposition}\label{prop:wedge_with_xi}
	It holds that $\DTC=D(\TC) \oplus \C$, while $\TC= \xi \wedge \DTC$.
\end{Proposition}
	
\proof
 
Evidently $D(\TC ) \oplus \C \subset \DTC$. 	The reverse inclusion  $ \DTC \subset D(\TC) \oplus \C $ follows from Theorem \ref{thm_abstract_valuations} and Proposition \ref{prop:primitive}. 

If $\tau\in \DTC$ then by the Leibniz rule for the convolution and Corollary \ref{cor:valuations_convolution_algebra},
	\begin{displaymath}
	  D(\xi\wedge \tau)=(-1)^{\Deg \tau}  D(\xi\ast\tau)= (-1)^{\Deg \tau} \gamma \ast\tau\in \DTC \subset \calH,
	\end{displaymath}
	and $\Deg(\xi \wedge \tau)=\Deg \tau-1=-1$, hence $\xi \wedge \DTC  \subset \TC$.
	
On the other hand, let $\kappa \in \calT\calC$. Since $\Deg\kappa=-1$ we can write it as $\kappa=\xi \wedge \tau$ for some $\tau$. Since $\xi\wedge \xi=0$, we may assume $\tau\in\mathcal H$. Then 
	\begin{displaymath}
		\xi \wedge D\tau=(-1)^{\Deg D\tau} \xi\ast D\tau =  D(\xi\wedge \tau) +(-1)^{\Deg \tau+1}\gamma\ast \tau \in\mathcal H,
	\end{displaymath} 
	while by Lemma \ref{lem:D_horizontal}, $D\tau \in \mathcal H$. It  follows that $ D\tau =0$, i.e. $\tau \in \mathcal Z$. Since $0=\Deg D(\xi\wedge \tau)= 1+\Deg(\xi\wedge \tau )=\Deg\tau$, we conclude that $\tau \in \DTC$. Thus $\TC  \subset \xi \wedge \DTC$.
\endproof

\begin{Corollary}
	 The elements $\omega_j=\frac{1}{j}\xi\wedge\varrho_{j-1}$, $j\geq 1$, form a basis of $\TC$.
\end{Corollary}
\proof
This follows at once from Theorem \ref{thm_abstract_valuations}.
\endproof

\begin{Proposition}\label{prop:curvature_closed_convolution}
 The space $\TC $ is closed under the valuation convolution.
\end{Proposition}
	
\proof
 Take $\kappa_1, \kappa_2\in \TC $. By the Leibniz rule for convolution and Corollary \ref{cor:valuations_convolution_algebra}, we have $D(\kappa_1\vee\kappa_2)=D\kappa_1\ast D\kappa_2\in\calH$, while $\Deg(\kappa_1\vee\kappa_2)=\Deg(\kappa_1\ast D\kappa_2)=-1+0=-1$. 
				
\endproof

\section{The tube coefficient subalgebra}
\label{sec:tube_coefficients}

In this section we verify that the abstract tube coefficients $\calT\calC$ correspond to the tube coefficient valuations over any Lie group equipped with a bi-invariant metric, thus proving Theorem \ref{mthm:tube}. We will provide two proofs of this theorem. The first one is indirect and uses the abstract machinery developed so far. It shows that the tube coefficients are very natural objects from the Lie algebraic point of view. The second proof is more geometric and direct.

\subsection{Smooth valuations}

We begin with a rapid overview of smooth valuations, recalling key definitions and notation and referring to \cite[Section 2]{bernig_faifman_kotrbaty_part1} for details. 

To an $n$-dimensional vector space $V$, we associate three different one-dimensional spaces: the determinant $\det V=\largewedge^n V$, the space of densities (Lebesgue measures) $\Dens(V)$, and the orientation line $\ori(V)$. They are related by 
\begin{displaymath}
	\Dens(V) \cong \det V^* \otimes \ori(V). 
\end{displaymath}

If $M$ is a smooth manifold, we let $\pi:\p_M\to M$ be its cosphere bundle  with fiber $\mathbb P_+(T^*M)$. If $M$ is riemannian, we let $SM$ be the unit sphere bundle. Obviously $SM \cong \p_M$.
We write $\Omega_{\ori}(X)$ for the complex-valued  differential forms on $X$ twisted by $\ori(M)$, where $X$ could be $\p_M$ or $M$.  A form on $\p_M$ is vertical if its restriction to any contact hyperplane vanishes.

The space of smooth valuations on a smooth $n$-dimensional manifold $M$ is denoted by $\mathcal V^\infty(M)$. There is a surjection $\Omega_{\ori}^n(M) \times \Omega_{\ori}^{n-1}(\p_M) \twoheadrightarrow \mathcal V^\infty(M), (\mu,\omega) \mapsto [[\mu,\omega]]$, 
Explicitly,  $[[\mu,\omega]]$ denotes the valuation
\begin{displaymath}
	Z\mapsto \int_Z \mu +\int _{\nc(Z)}\omega,
\end{displaymath} 
where $\nc(Z)$ is the normal cycle of $Z$.

There is a also an injection $\mathcal V^\infty(M) \hookrightarrow C^\infty(M) \times \Omega^n_{\ori}(\p_M)$, which to the valuation $\phi=[[\mu,\omega]]$ puts in correspondence the function $f(x)=\phi(\{x\})$ and the $n$-form $\tau=D\omega +\pi^*\mu$, where $D$ is the Rumin operator (not to be confused with the operator $D$ on $\calG$ that is central to the present paper). The unique valuation that is mapped to the pair $(f,\tau)$ is denoted by $\{f,\tau\}$.

\subsection{Valuations on Lie groups}
The convolution of valuations on a compact Lie group is defined as 
\begin{displaymath}
	\mathcal V^\infty(G) \times \mathcal V^\infty(G) \to \mathcal V^\infty(G), (\phi_1, \phi_2) \mapsto \phi_1*\phi_2:=m_*(\phi_1 \boxtimes \phi_2),
\end{displaymath}
where $m:G \times G \to G$ is the multiplication map and $\boxtimes$ denotes the exterior product. Since the exterior product of two smooth valuations is a generalized valuation, one has to check that the push-forward under $m$ is defined, and that the result is smooth again. This was shown in \cite{alesker_bernig_convolution}.

In \cite{bernig_faifman_kotrbaty_part1}, we have defined the convolution of bi-invariant valuations on an arbitrary Lie group. 
More precisely, there is a convolution product
\[ \ast: \mathcal V^\infty(G)^{G\times G}\otimes\Dens(\g^*) \times \mathcal V^\infty(G)^{G\times G}\otimes\Dens(\g^*) \to \mathcal V^\infty(G)^{G\times G}\otimes\Dens(\g^*),\]
which coincides with the compact case once a probability Haar measure is used to trivialize $\Dens(\g^*)$.

\subsection{Forms on Lie groups}
We henceforth let $G$ be an $n$-dimensional Lie group equipped with a bi-invariant riemannian metric.	As in \cite[Section 3]{bernig_faifman_kotrbaty_part1}, we identify the left-invariant forms on $SG$ with $\Omega(S(\g), \largewedge\g^*)$ using left translations. 

	The Hodge star on $\largewedge \g$ induces an isomorphism \begin{equation*} \ast:\Omega(S(\mathfrak g),\largewedge\g)\to \Omega(S(\mathfrak g), \largewedge\g^*)\otimes\largewedge^n\g=\Omega(SG)^{L_G}\otimes\largewedge^n\g=\Omega_{\ori}(SG)^{L_G}\otimes \Dens(\g^*). \end{equation*}
	
	We arrive at the identification
	\begin{equation}\label{eq:identification}
		\Omega_{\ori}(SG)^{L_G}\otimes \Dens(\g^*)=\Omega(S(\g), \largewedge\g).
	\end{equation}

	For $\tau\in \Omega(S(\g), \largewedge\g)$, let $\deg\tau$ denote its degree as a differential form, while its wedge degree $\deg_w\tau$ corresponds to the grading on the exterior algebra $\largewedge\g$. Define also $D\tau=d \tau +(-1)^{\deg \tau-\deg_w\tau+1}\partial\tau$.
	The following fact appears in \cite[Corollary 3.3]{bernig_faifman_kotrbaty_part1}.
		
	\begin{Proposition}\label{prop:d=D}
		For $\zeta\in 	\Omega_{\ori}(SG)^{L_G}\otimes \Dens(\g^*)$, let $\tau\in \Omega(S(\g), \largewedge\g)$ be the corresponding form under eq. \eqref{eq:identification}.  
		Then $d\zeta$ corresponds to $D\tau$.
	\end{Proposition}

 \begin{Corollary} \label{cor_leibniz_part1}
The convolution product on $\Omega_{\ori}(\p_G)^{G\times G} \otimes\Dens(\g^*)$ constructed in \cite[Definition 4.2]{bernig_faifman_kotrbaty_part1} satisfies the Leibniz rule 
\begin{displaymath}
	d(\tau * \zeta)=d\tau * \zeta+(-1)^{\deg \tau-n} \tau * d \zeta.
\end{displaymath}		
\end{Corollary}

\proof
We have a surjective map from $\Omega_\g^G$ to $\Omega_{\ori}(\p_G)^{G\times G} \otimes\Dens(\g^*)$, given by restriction to the sphere, followed by the  identifications \eqref{eq:identification} and $SG=\p_G$. Under this map, $\Deg$ corresponds to the degree on $\Omega_{\ori}(\p_G)^{G\times G} \otimes\Dens(\g^*)$ minus $n$. By construction of the two convolution products, this map is an algebra morphism. The displayed equation is then a direct consequence of the Leibniz rule Proposition \ref{prop:Leibniz}.   
\endproof

\subsection{The tube formula}

Let
\begin{displaymath}
  \Re^S_\g:\mathcal G\to \Omega(SG)^{G\times G}\otimes\largewedge^n\g=\Omega_{\ori}(SG)^{G\times G}\otimes \Dens(\g^*)
\end{displaymath}	
be the composition of $\Re_\g:\mathcal G\to \Omega(\mathfrak g, \largewedge\g)$ with the restriction $\mathrm{res}$ to $S(\mathfrak g)$, followed by identification eq. \eqref{eq:identification}. It is clear from the constructions that for all $\tau,\zeta \in \calG$
	\begin{equation} \label{eq_compatibility_rgs}
		\Re^S_\g(\tau * \zeta)=\Re^S_\g \tau * \Re^S_\g \zeta.
	\end{equation}

We define another realization map by 
\begin{displaymath}
	\Re_\g^V: \calG(-1)  \to \mathcal V^\infty(G)^{G \times G} \otimes\Dens(\mathfrak g^*), \omega \mapsto [[0, \Re^S_\g\omega]]=\int_{\nc(\bullet)} \Re_\g^S\omega.
\end{displaymath}

The different realization maps fit into the following commutative diagram. 
\begin{equation*}
	\begin{tikzcd}
	&\Omega_\g^G \arrow{r}{\mathrm{res}}& \Omega(S(\g),\largewedge\g)^G \arrow{dd}{\cong}  \\
	 \\
	&\calG  \arrow[uu,"\Re_\g"] \arrow[r,"\Re_\g^S"] & \Omega_{\ori}(SG)^{G \times G} \otimes \Dens(\g^*)\\
	\\
	\TC \arrow[hookrightarrow, r] &\calG(-1) \arrow[hookrightarrow,uu]          \arrow[r,"\Re_\g^S"]  \arrow[ddr,"\Re_\g^V"] & \Omega_{\ori}^{n-1}(SG)^{G \times G} \otimes \Dens(\g^*) \arrow[dd,"\nc"]  \arrow[hookrightarrow,uu]\\
	\\
	&& \calV^\infty(G)^{G \times G} \otimes \Dens(\g^*)
	\end{tikzcd}
\end{equation*}

\begin{Definition}
	The tube coefficient algebra $\mathcal {TC}(G)\subset \mathcal V^\infty(G)^{G \times G} \otimes\Dens(\mathfrak g^*)$ is spanned by all valuations of the form $\Re_\g^V(\omega)$, with $\omega\in \calT\calC$, together with $\vol_G\otimes \vol_G^*$. 
\end{Definition}

\begin{Proposition}\label{prop:convolution_n-1}
The convolution of valuations $\phi\in \mathcal V^\infty(G)^{G \times G}\otimes \Dens(\mathfrak g^*)$ and $\psi\in \mathcal V^\infty(G)^{L_G}$ is given by 
	\begin{equation} 
		\label{eq_def_formula_convolution}
		\phi * \psi=[[ \mu(\phi)\mu_\psi, \omega_\phi\ast \tau_\psi+\mu(\phi)\omega_\psi]], 
	\end{equation}
	where $\phi=[[\mu_\phi, \omega_\phi]]$ and $\psi=\{c_\psi,\tau_\psi\}=[[\mu_\psi, \omega_\psi]]$,  with $\omega_\phi$, $\omega_\psi$ left-invariant and $d\omega_\phi, d\omega_\psi$ vertical. 
\end{Proposition}
Here  $\mu_\phi=\mu(\phi)\vol\otimes\vol^*$ (resp. $\mu_\psi$) is a twisted Haar measure, with $\mu(\phi),\mu(\psi)\in\C$. Note that the existence of such representations for $\phi$ and $\psi$ is non-trivial, but was shown in \cite[Proposition 7.3]{bernig_faifman_kotrbaty_part1}.
\proof
Denote $\eta=[[ \mu(\phi)\mu_\psi, \omega_\phi\ast \tau_\psi+\mu(\phi)\omega_\psi]]$. The $0$-form of $\eta$ is 
\begin{displaymath}
	\pi_*(\omega_\phi\ast \tau_\psi)+\mu(\phi)c_\psi,
\end{displaymath}
while by \cite[Definition 1.1]{bernig_faifman_kotrbaty_part1}, the $0$-form of $\phi * \psi$ is 
\begin{displaymath}
	c_\phi\mu(\psi)+\pi_*(\tau_\phi\ast \omega_\psi).
\end{displaymath}
 Recall that $\tau_\phi=d\omega_\phi+\pi^*\mu_\phi$,  $\tau_\psi=d\omega_\psi+\pi^*\mu_\psi$.  By Corollary \ref{cor_leibniz_part1} 
\begin{displaymath}
  d(\omega_\phi\ast\omega_\psi)=d\omega_\phi\ast \omega_\psi-\omega_\phi\ast d\omega_\psi.
\end{displaymath}

It follows that
\begin{align*} 
	\pi_*(\omega_\phi\ast \tau_\psi)-\pi_*(\tau_\phi\ast \omega_\psi)&= -\pi_*d(\omega_\phi\ast \omega_\psi) +\pi_*(\omega_\phi\ast \pi^*\mu_\psi)-\pi_*(\pi^*\mu_\phi\ast \omega_\psi)
\\
&=-d\pi_*(\omega_\phi\ast \omega_\psi)+\mu(\psi)c_\phi-\mu(\phi)c_\psi\\
&=\mu(\psi)c_\phi-\mu(\phi)c_\psi,
\end{align*}
and so the $0$-forms of $\eta$ and $\phi\ast \psi$ are equal.

Next, again using Corollary \ref{cor_leibniz_part1},
\begin{displaymath} d(\omega_\phi\ast \tau_\psi+\mu(\phi)\omega_\psi) +\mu(\phi) \pi^*\mu_\psi=\tau_\phi\ast \tau_\psi
\end{displaymath}
and hence the $n$-forms agree as well.
\endproof

\begin{Corollary}\label{cor:valuation_convolution_realization_morphism}
For $\omega_1,\omega_2 \in \TC$ it holds that	$\Re_\g^V(\omega_1\vee\omega_2)=\Re_\g^V(\omega_1)\ast \Re_\g^V(\omega_2)$.
\end{Corollary}
\proof
Denote $\phi_j=\Re_\g^V(\omega_j)$, $\zeta_j=\Re_\g^S(\omega_j)$. It holds by assumption and by Proposition \ref{prop:d=D} that the Rumin differential of $\zeta_j$ is $d\zeta_j=\Re^S_\g D\omega_j$. It follows by Propositions \ref{prop:convolution_n-1} and eq. \eqref{eq_compatibility_rgs} that \[\phi_1\ast \phi_2=[[0, \zeta_1\ast d\zeta_2]]=[[0, \Re^S_\g(\omega_1\ast D\omega_2)]]=\Re_\g^V(\omega_1\vee\omega_2).\]
\endproof

 Denote $\upsilon_m = \Re^V_\g(\omega_m)\in\mathcal V^\infty(G)^{G \times G}\otimes \Dens(\g^*)$. In particular, $\upsilon_1$ is the surface area valuation $\mu_{n-1}$. We write $\upsilon_0=\vol_G\otimes\vol_G^*$. 
\begin{Corollary}
	$\mathcal {TC}(G)$ is a convolution subalgebra of $\mathcal V^\infty(G)\otimes\Dens(\mathfrak g^*)$ which is spanned by $\upsilon_m$, $m\geq 0$, and generated by $\upsilon_1$.
\end{Corollary}
\proof
The first assertion follows at once from Propositions  \ref{prop:curvature_closed_convolution} and Corollary \ref{cor:valuation_convolution_realization_morphism}. 
Next, Theorem \ref{thm_abstract_valuations} and Proposition \ref{prop:wedge_with_xi} show that $\mathcal {TC}(G)$ is spanned by $\upsilon_m$.
Finally, using Corollaries \ref{cor:valuation_convolution_realization_morphism} and \ref{cor:dgla_convolution_table}, Definition \ref{def:omega} and Proposition \ref{prop:primitive} we find that  
\begin{displaymath}
	\upsilon_1^{\ast m}=\Re_\g^V(\xi\ast\gamma^{\ast(m-1)})=(m-1)!\Re_\g^V(\xi\wedge \varrho_{m-1})=m!\Re_\g^V(\omega_m)=m!\upsilon_m.
\end{displaymath}

\endproof

\proof[Proof of Theorem \ref{mthm:tube}.]

	First note that the statement makes formal sense. The left hand side depends on the choice of some volume $\vol$ on $G$. The right hand side is an element of $\Dens(\g^*)$, and we can use the fixed volume $\vol$ to identify it with a scalar.  
	
By $B_\epsilon\subset G$ we denote the $\epsilon$-ball centered at the neutral element. 
Let $T_\epsilon:\mathcal V^\infty(G) \to \mathcal V^\infty(G)$ be the tubular operator from \cite{solanes_trillo}. By definition of $T_\epsilon$ we have $T_\epsilon(\phi)(Z)=\phi(B_\epsilon Z)$ for $\epsilon \geq 0$ sufficiently small. Let
\begin{displaymath}
	\partial:=\left.\frac{d}{d\epsilon}\right|_{\epsilon=0} T_\epsilon: \mathcal V^\infty(G) \to \mathcal V^\infty(G)
\end{displaymath}
be the derivation operator. 

Write $R$ for the Reeb vector field on the sphere bundle $SG$ of $G$. Then by \cite[Proposition 4.2]{solanes_trillo} and Proposition \ref{prop:convolution_n-1}, 
\begin{displaymath}
	\partial \phi=[[0,i_R\tau_\phi]]=\mu_{n-1}\ast\phi, \quad \phi \in \mathcal V^\infty(G).
\end{displaymath}

We deduce that
\begin{displaymath}
	\left.\frac{d^k}{d\epsilon^k}\right|_{\epsilon=0} \mu_{n-1}(B_\epsilon Z)=(\mu_{n-1}^{\ast k+1})(Z) =(k+1)! \int_{\nc(Z)} \Re_\g^S \omega_{k+1}.
\end{displaymath}

Note that $T_\epsilon(\mu_{n-1})$ is represented by the $(n-1)$-form $\Phi_\epsilon^*({ \Re_\g^S \xi})$, where $\Phi_\epsilon$ is the Reeb flow.  Since $\Phi_\epsilon(g, v)=(g\exp(\epsilon v), v)$, $T_\epsilon(\mu_{n-1})$ is analytic in $\epsilon \in\mathbb R$. It follows that for small $\epsilon \geq 0$,
\begin{displaymath}
	\mu_{n-1}(B_\epsilon Z)=\sum_{k=0}^\infty \frac{1}{k!}\left.\frac{d^k}{ds^k}\right|_{s=0}\mu_{n-1}(B_sZ)\epsilon^k=\sum_{k=0}^\infty (k+1) \int_{\nc(Z)} \Re_\g^S \omega_{k+1} \epsilon^k.
\end{displaymath}
 Since $\frac{d}{d\epsilon} \vol(B_\epsilon Z)=\mu_{n-1}(B_\epsilon Z)$, an integration from $0$ to $\epsilon$ finishes the proof. 
\endproof

\subsection{Geometric proof of the tube formula}

\proof[Proof of Theorem \ref{mthm:tube}]
Define the function 
\begin{displaymath}
	F_\epsilon(t):=\frac{e^{\epsilon t}-1}{t}=\sum_{i=0}^\infty \frac{t^i \epsilon^{i+1}}{(i+1)!}. 
\end{displaymath}
The definition of $\zeta_\epsilon$ can be rewritten as
\begin{displaymath}
	\zeta_\epsilon=-F_\epsilon(\ad_\xi) \gamma.
\end{displaymath}

For a Lie group $G$ and $\xi \in \g$, the differential $d \exp|_\xi:\g=T_\xi \g \to T_{\exp(\xi)}G$ is given by $dR_{\exp \xi} \circ F_1(\ad \xi)$ \cite[Theorem 1.5.3]{duistermaat_kolk}. Since $\gamma$ is mapped to the identity map under the realization map $\Re_\g$, we have 
\begin{displaymath}
	\Re_\g \zeta_\epsilon=-F_\epsilon(\ad_\xi) \in \Omega^1(\g,\g) \subset \Omega_\g.
\end{displaymath}

Consider the map 
\begin{displaymath}
	\exp_\epsilon:SG=G \times S(\g) \to G, (g,\xi) \mapsto \exp_g \epsilon dL_g(\xi)=g \exp(\epsilon \xi), g \in G, \xi \in S(\g).
\end{displaymath}
Here we have used that the riemannian exponential map is $G$-equivariant and that it coincides with the Lie group exponential map at the neutral element. 

Fix $(g,\xi)$ and set $q:=\exp(\epsilon \xi) \in G$. 
We set  
\begin{displaymath}
	\Phi:= dL_{(gq)^{-1}} \circ d \exp_\epsilon|_{g,\xi} \circ (dL_g \oplus \mathrm{id}):  \g \oplus T_\xi S(\g) \to \g.
\end{displaymath}
One computes that 
\begin{displaymath}
	\Phi(X,Y)=\Ad_{q^{-1}}(X+F_\epsilon(\ad \xi)Y).
\end{displaymath}

Take a positive orthonormal basis $X_1,\ldots,X_n$ of $\g$ and denote $\tau_1,\ldots,\tau_n \in \g^*$ the dual basis. The volume form of $G$ at the neutral element is given by $\vol_G=\largewedge_{i=1}^n \tau_i$. Since it is $\Ad(G)$-invariant, it follows that
\begin{displaymath}
	\Phi^* \vol_G= \largewedge_{i=1}^n (\pi^*\tau_i+F_\epsilon(\ad \xi)^*\tau_i) \in \largewedge^n (\g^* \oplus T_\xi^* S(\g)).
\end{displaymath}
We can identify $ \largewedge^n (\g^* \oplus T_\xi^* S(\g))=\bigoplus_{k=0}^n \largewedge^k T_\xi^* S(\g) \otimes \largewedge^{n-k} \g^*$, see \cite[Equation (13)]{bernig_faifman_kotrbaty_part1}.

For a subset $I=(i_1,\ldots,i_k) \subset \{1,\ldots,n\}$ in increasing order, we let $I^c$ be the complement, ordered increasingly, and $\sigma(I)$ the sign of the permutation $(I,I^c)$. Using the notation $\tau_I=\largewedge_{l=1}^k \tau_{i_l} \in \largewedge^k \g^*$ we can write 
\begin{displaymath}
	\Phi^*\vol_G=\sum_I \sigma(I) F_\epsilon(\ad \xi)^*\tau_I \otimes \tau_{I^c},
\end{displaymath}
where the sum runs over all possible $I$.

We now apply the inverse of the Hodge star $\ast:\largewedge^k \g \otimes \largewedge^n \g^* \to \largewedge^{n-k}\g^*$, where we identify $\largewedge^n \g^* \cong \C$ using the volume form. This gives
\begin{align*}
	\ast^{-1} \Phi^*\vol_G & = \sum_I F_\epsilon(\ad \xi)^*\tau_I \otimes X_I\\
	& = \exp \left(\sum_{i=1}^n F_\epsilon(\ad \xi)^*\tau_i \otimes X_i\right)\\
	& = \exp F_\epsilon(\ad \xi)\\
	& = \mathrm{res} \circ \Re_\g(\rho_\epsilon),
\end{align*}
and therefore
\begin{displaymath}
	\exp_\epsilon^*\vol_G=\Re_\g^S(\rho_\epsilon).
\end{displaymath}

Let $Z \in \mathcal P(G)$. Then $Z$ is of positive reach. For $0<\epsilon<\mathrm{reach}(Z)$, the boundary $\partial B_\epsilon Z$ is given by $(\exp_\epsilon)_* \nc(Z)$. By the Gauss lemma, the unit vector $\xi \in \g = T_{\exp_\epsilon(g,\xi)}G$ is orthogonal to $\partial Z_\epsilon$. Hence the volume form of $\partial Z_\epsilon$ is $\vol_{\partial Z_\epsilon}=\iota_\xi \vol_G$. 

 Since $\xi=\Phi(\xi,0)$, we have 
\begin{displaymath}
\ast^{-1} \Phi^* \iota_\xi \vol_G = \ast^{-1} \iota_{(\xi,0)} \Phi^*\vol_G=\xi \wedge \ast^{-1} \Phi^*\vol_G= \mathrm{res} \circ \Re_g(\xi \wedge \rho_\epsilon)
\end{displaymath}
and thus 
\begin{displaymath}
	\exp_\epsilon^* \vol_{\partial Z_\epsilon}=\Re_\g^S(\xi \wedge \rho_\epsilon).
\end{displaymath}

We conclude that 
\begin{displaymath}
  \frac{d}{d\epsilon} \vol(Z_\epsilon)=\vol(\partial Z_\epsilon)=\int_{\nc(Z)} \Re_\g^S(\xi \wedge \rho_\epsilon)=\Re_\g^V(\xi \wedge \rho_\epsilon)(Z)=\sum_{k=1}^\infty k \upsilon_k(Z) \epsilon^{k-1}, 
\end{displaymath}
and hence 
\begin{displaymath}
	\vol(Z_\epsilon)=\sum_{k=0}^\infty \upsilon_k(Z) \epsilon^k.
\end{displaymath}
\endproof

\subsection{Example}

Let us work out explicitly a particular case, namely $G=S^3$ with its round metric. Note that on each tangent space to the sphere $S^2 \subset T_\xi S^3$, $\ad \xi$ acts by twice the complex structure, in particular $(\ad_ \xi)^2=-4\mathrm{Id}$. The series for $\zeta_\epsilon$ thus simplifies to  
\begin{displaymath}
	\zeta_\epsilon=\frac{\sin^2  \epsilon}{2} \ad \xi \gamma -\frac{\sin (2\epsilon)}{2}  \gamma.
\end{displaymath} 

Let us compute $\mu_{n-1}(B_\epsilon Z)$, where $Z=\{p\}$ is a point. We need the component of degree $2$ and wedge degree $2$ in $\varrho_\epsilon=\widehat\exp(-\zeta_\epsilon)$. We have $\gamma \wedge \ad_ \xi \gamma=0$ and $\ad_ \xi \gamma \wedge \ad_ \xi \gamma=4 \gamma \wedge \gamma$. It follows that the bidegree $(2,2)$-part in $\rho_\epsilon$  given by 
\begin{displaymath}
	\frac12 \frac{\sin^2(2\epsilon)}{4} \gamma \wedge \gamma +\frac12 \frac{\sin^4 \epsilon}{4} \ad_ \xi \gamma \wedge \ad_ \xi \gamma=\frac{\sin^2\epsilon}{2} \gamma \wedge \gamma.
\end{displaymath}
Now $*_1(\xi \wedge \gamma \wedge \gamma)$ is twice the volume form on the sphere $S^2 \subset T_{\{p\}} S^2$ and the tube formula thus gives us the correct value 
\begin{displaymath}
	\mu_2(B_\epsilon Z)=\sin^2 \epsilon \vol(S^2). 
\end{displaymath}

\section{Existence of invariant valuations}
\label{sec:existence}

 In this final section we strengthen Theorem B from \cite{bernig_faifman_kotrbaty_part1}. There we have shown that a \emph{connected} Lie group that admits a smooth bi-invariant valuation outside the span of the Euler characteristic and the Haar measure, must also admit a bi-invariant metric. As it turns out, this holds for general Lie groups.
	
Next, we showed in \cite{bernig_faifman_kotrbaty_part1} that if there is a bi-invariant metric on $G$, then the space $\calV^\infty(G)^{G \times G}$ of smooth bi-invariant valuations is typically infinite-dimensional. For compact semisimple groups $G$ with $\g \neq \mathfrak{so}(3)$ we give here a more explicit argument, showing that the subspace $\TC(G) \subset \mathcal V^\infty(G)^{G \times G}$ is infinite-dimensional.

\subsection{Existence of smooth non-obvious bi-invariant valuations}
The following two lemmas appeared in \cite{bernig_faifman_kotrbaty_part1}, and are reproduced here for convenience. While they were formulated for Lie groups, their proofs are completely general. 
\begin{Lemma}\label{lem:zero_implies_infinite}
	Let $G$ be any group, and $V$ a real finite-dimensional representation of $G$. If $v\in V$ is non-zero and $Gv$ has $0$ as its limit point, then $Gv$ is unbounded.
\end{Lemma}

\begin{Lemma}\label{lem:unbounded}
	Let $G$ be any group, and $V$ a real finite-dimensional representation of $G$. If $V$ has no  $G$-invariant Euclidean structure, then there is an open dense set of vectors $v\in V$ with $Gv$ unbounded.
\end{Lemma}

By $\sigma(A)\subset\C$ we denote the (algebraic) spectrum of $A\in\mathrm{GL}_n(\C)$, that is the roots of the characteristic polynomial of $A$ including multiplicity. The following is standard linear algebra.
\begin{Lemma}\label{lem:elliptic} Let $A\in\mathrm{GL}_n(\R)$. The following are equivalent:
	\begin{enumerate}
		\item $A$ is orthogonal for some Euclidean structure on $\R^n$.
		\item $\{A^j\}_{j=-\infty}^\infty$ is bounded in $\mathrm{GL}_n(\R)$.
		\item $|\lambda|=1$ for all $\lambda\in\sigma(A)$, and $A$ is semisimple, that is diagonalizable over $\C$.
		
	\end{enumerate}
\end{Lemma}

\begin{Definition}
	We say that $A\in\mathrm{GL}_n(\R)$ is elliptic type if it satisfies the equivalent conditions of Lemma \ref{lem:elliptic}.
\end{Definition} 

\begin{Lemma}\label{lem:radical_elliptic}
 If $A\in\mathrm{GL}_n(\R)$ satisfies that $A^k$ is elliptic type for some $k>0$, then so is $A$.
\end{Lemma}
\proof
Denote $X=\{A^{kj}:j\in\mathbb Z\}$, a bounded set in $\mathrm{GL}_n(\R)$ by assumption. Then for any $m\in\mathbb Z$ we can write $m=kj+r$ with $0\leq r\leq k-1$, and so
$\{A^{m}:m\in\mathbb Z\} \subset X\cup XA\cup\dots\cup XA^{k-1}$ is also a bounded set.
\endproof

\begin{Lemma}\label{lem:few_zero_orbits}
	Let $A\in \mathrm{GL}_n(\R)$ and denote $G=\{A^k\}_{k=-\infty}^\infty$. Assume there exist $\lambda, \lambda'\in\sigma(A)$ with $|\lambda|\geq 1\geq |\lambda'|$. Then $G v$ does not contain $0$ in its closure for some open dense set $U\subset \R^n$ of vectors $v\in\R^n$.
\end{Lemma}
\proof
Let $\R^n=E^+\oplus E^0\oplus E^-$ be the decomposition into the unstable, center, and stable eigenspaces. They correspond to direct sums of generalized eigenspaces with eigenvalues of absolute value more, equal, or less than $1$, intersected with $\R^n$. By assumption, either $E^0$ or both of $E^+, E^-$ are non-trivial. 

It is a well known fact from linear dynamics that $A^kv \to \infty$ for $k\to \infty$ if $v^+\neq 0$, and $0$ is not in the closure of $A^kv$, $k\geq 1$, if $v^0\neq 0$. 

If $E^+, E^-$ are non-trivial, we may take $U=\{v: v_+\neq 0, v_-\neq 0\}$. If $E^0$ is non-trivial, we may take $U=\{v: v_0\neq 0\}$.
\endproof

Let $h$ be a hermitian structure on $\C^n$. The $h$-singular values of $X\in\mathrm{GL}_n(\C)$, denoted $\sigma_1^h(X)\geq\dots\geq \sigma_n^h(X)$, are the square roots of the eigenvalues of $X^hX$, where $X^h$ is the transposed with respect to $h$: $h(X^hx, y)=h( x, Xy)$ for all $x,y\in\C^n$.
\begin{Lemma}\label{lem:hermitian}
	Let $A\in\mathrm{GL}_n(\C)$ be such that $|\lambda|>1$ for all $\lambda\in\sigma(A)$.
	Then there is a hermitian structure $h$ on $\C^n$ such that the the $h$-singular values of $A^k$ satisfy $\sigma_j^h(A^k)\to\infty$ as $k\to\infty$.
\end{Lemma}
\proof
We may assume $A$ has only one Jordan block in its standard Jordan form. Indeed, assuming the statement for such matrices, and letting $E_j$ be the invariant subspaces of $A$ corresponding to the Jordan blocks, we may take the $l^2$ sum of the hermitian structures $h_j$ on $E_j$ to get the desired metric.

Now choose a Jordan basis of $A$, and let $h$ be a hermitian structure on $\C^n$ for which it is an orthonormal basis. Thus $A=J_n(\lambda)$ with $|\lambda|>1$, and henceforth we work with the standard hermitian structure on $\C^n$.

Recall that the top singular value $\sigma_1(X)$ is a matrix norm, and by Gelfand's formula, the spectral radius is $\rho(X)=\lim_{k\to \infty} \sigma_1(X^k)^{1/k}$ \cite[Example 5.6.6, Corollary 5.6.14]{horn_johnson}. Hence
\begin{displaymath}
  |\lambda|=\rho(A)= \lim \sigma_1(A^k)^{1/k}.
 \end{displaymath}
In particular, for any $\epsilon>0$, $\sigma_1(A^k)\leq |\lambda|^k(1+\epsilon)^k$ for $k$ sufficiently large. Fix $\epsilon$ such that $(1+\epsilon)^{n-1}<|\lambda|$.

Recall that $\sigma_1(A^k)\cdots\sigma_n(A^k)= |\det A|^k=|\lambda|^{nk}$. Thus

\[\sigma_n(A^k)\geq \frac{|\lambda|^{nk}}{\sigma_1(A^k)^{n-1}}\geq \frac{|\lambda|^{nk}} { |\lambda|^{(n-1)k}(1+\epsilon)^{(n-1)k} }=\frac{|\lambda|^k}{(1+\epsilon)^{(n-1)k}}\to\infty.\]

\endproof

\begin{Corollary}\label{cor:determinant}
	Let $A\in\mathrm{GL}_n(\R)$ have $|\lambda|>1$ for all $\lambda\in\sigma(A)$, and $1\leq r\leq n-1$.  Then there is a Euclidean metric $P$ on $\R^n$ and a sequence $k_j\to\infty$ such that 
	\begin{displaymath}
		\inf_{E\in\Gr_r(\R^n)} |\det{}_P (A^{k_j}:E\to A^{k_j}E)| \to \infty.
	\end{displaymath}

\end{Corollary}
\proof
Choose $h$ using Lemma \ref{lem:hermitian}. We choose $m$ large enough so that $\sigma^h_n(A^m)>1$. Let $P$ be the restriction of $h$ to $\R^n$. 
Let $A^m=BS$ be the polar decomposition of $A^m$ with $B$ $P$-orthogonal and $S$ $P$-positive definite. The eigenvalues $\sigma^h_j(S)$ of $S$ are the $P$-singular values of $A^m$, in particular $\sigma^h_n(S)=\sigma^h_n(A^m)>1$.

Note that $|\det_P(B:E\to BE)|=1$ for all $E$, while $|\det_P(S:E\to SE)|\geq \sigma^h_n(S)$. Therefore $|\det(A^{m j}:E\to A^{mj}E)|\geq \sigma^h_n(A^m)^j\to \infty$ as $j\to\infty$, and we take $k_j=mj$.

\endproof

The space of translation-invariant, smooth valuations on a vector space $V$ is denoted by $\Val^\infty(V)$. The subspace of even/odd $k$-homogeneous valuations is denoted $\Val^{\pm, \infty}_k(V)$.

The Klain section of a valuation $\phi\in\Val^{+,\infty}_k(\R^n)$ is the section $\Kl_\phi\in\Gamma(\Gr_k(\R^n), \Dens(E))$ given by $\Kl_\phi(E)=\phi|_E$. Klain's theorem \cite{klain00} asserts that $\phi\mapsto \Kl_\phi$ is injective.

A similar result for odd valuations is due to Schneider \cite{schneider96}: There is a short exact sequence 
\begin{displaymath}
	0 \to L\to \Gamma(\widetilde{\mathrm{Fl}}_{k,k+1}(\R^n), \Dens(E)) \to \Val_k^{-,\infty}(\R^n)\to 0,
\end{displaymath}
where $\widetilde{\mathrm{Fl}}_{k,k+1}(\R^n)$ is the partial flag manifold of cooriented pairs $E\subset F$, and $L$ a certain subspace of sections which, for any $F\in \Gr_{k+1}(\R^n)$, has $(k+1)$-dimensional restriction to $\widetilde\Gr_{k}(F)$.

\begin{Theorem}\label{thm:elliptic}
	Assume $A^*\phi=\phi$ for some $A\in\mathrm{GL}_n(\R)$ and non-zero $\phi\in\Val^\infty_k(\R^n)$, $1\leq k\leq n-1$. Then $A$ is elliptic type.
\end{Theorem}
\proof
Denote $V=\R^n$, and $G=\{A^m\}_{m=-\infty}^\infty$. Let $|\lambda_1|\geq\dots\geq |\lambda_n|$ be the spectrum of $A$ including (algebraic) multiplicities.

{\bf Case 1:}  $\{|\lambda|-1:\lambda\in\sigma(A)\}$ all have the same sign. Replacing $A$ by $A^{-1}$ if necessary, we may assume $|\lambda|>1$ for all $\lambda\in \sigma(A)$.

Write $\phi=\phi_++\phi_-$, the even and odd components, both of which are $G$-invariant.

Assume first $\phi_+\neq 0$. Let $\kappa$ be its Klain section. Choose a metric $P$ and a sequence $k_j$ as in Corollary \ref{cor:determinant}, and identify $\kappa$ with a smooth function $\tilde \kappa$ on $\Gr_k(\R^n)$ using $P$ to trivialize the Klain bundle over $\Gr_k(\R^n)$ of densities over the tautological bundle.  The  $G$-invariance of $\kappa$ translates to 
\begin{equation} \label{eq_invariance_tilde_kappa}
	\tilde \kappa(E)=\tilde \kappa(A^mE) |\det(A^m:E \to A^mE)|
\end{equation}
for all $E$ and $m$.

 Since $\phi_+ \neq 0$, there exists some $E$ with $\tilde \kappa(E)\neq 0$. Choose a subsequence of $k_j$, still denoted $k_j$, so that $A^{-k_j}E\to E_0\in\Gr_k(\R^n)$. It then follows from \eqref{eq_invariance_tilde_kappa} applied to $A^{-k_j}E$ that $\tilde\kappa(A^{-k_j}E) =\tilde \kappa(E)  |\det(A^{k_j}:E \to A^{k_j}E)| \to\infty$, while by the continuity of $\tilde\kappa$ we have $\tilde\kappa(A^{-k_j}E)\to \tilde\kappa(E_0)$, a contradiction. Thus Case 1 cannot hold if $\phi_+\neq 0$.

Next assume $\phi_-\neq 0$. By Schneider's theorem we have the short exact sequence \[0\to L\to \Gamma(\widetilde{\mathrm{Fl}}_{k,k+1}(\R^n), \Dens(E)) \to \Val_k^{-,\infty}(\R^n)\to 0.\]
Denote $W=\Gamma(\widetilde{\mathrm{Fl}}_{k,k+1}(\R^n), \Dens(E))$. We claim that $\phi_-\in \Val_k^{-,\infty}(\R^n)^G$ has a $G$-invariant lift $\theta\in W$. It is easily seen to follow from the fact that $\id-A$ is invertible for the induced action of $A$ on $L$: choose any lift $s$ of $\phi_-$ which then satisfies $As=s+ \xi$ with $\xi\in L$. 
We want to find $\zeta\in L$ such that \[A(s+\zeta)=s+\zeta\iff \xi+A\zeta=\zeta\iff (\id-A)\zeta=\xi.\]
The invertibility of $\id-A$ on $L$ in turn follows from the existence of $(I-A)^{-1}\in\mathrm{GL}(\R^n)$ by the assumption on $\sigma(A)$.

We can now repeat the argument for $\phi_+$. Choose a metric $P$ and a sequence $k_j$ as in Corollary \ref{cor:determinant}. Use $P$ to identify $W$ with $C^\infty(\widetilde{\mathrm{Fl}}_{k,k+1}(\R^n))$. Then $\theta$ becomes $\tilde \theta\in C^\infty(\widetilde{\mathrm{Fl}}_{k,k+1}(\R^n))$,
which satisfies an equation analogous to \eqref{eq_invariance_tilde_kappa}:
\begin{equation} \label{eq_invariance_tilde_kappa2}
	\tilde \theta(F,E)=\tilde \theta(A^mF,A^mE) |\det(A^m:E \to A^mE)|.
\end{equation}

Since $\phi_- \neq 0$, there exists $(F,E)$ with $\tilde \theta(F, E)\neq 0$. Choosing a subsequence of $k_j$ such that $A^{-k_j}(F, E)\to (F_0, E_0)$, we deduce that $\lim \tilde\theta( A^{-k_j}(F, E))=\infty$, a contradiction. We conclude that Case $1$ cannot hold under the assumptions of the theorem.

{\bf Case 2:} $|\lambda_1|\geq1 \geq |\lambda_n|$. We first assume $\det A>0$. Assume in the way of contradiction there is no $G$-invariant metric on $V$, equivalently on $V^*$. 

Since $\det A>0$, we may fix an orientation on $V$ so that the $n$-form 
$\tau=\tau_{\phi}$ of $\phi$ is a non-zero $G$-invariant element of 
\begin{align*}
	\Omega^{n-k}(\mathbb P_+(V^*), \largewedge^{k} V^*) & =\Gamma(\mathbb P_+(V^*),\largewedge^{n-k}(\xi^*\otimes V^*/\xi)^*\otimes\largewedge^{k}V^*)\\
	& = \Gamma(\mathbb P_+(V^*),\largewedge^{n-k}(\xi \otimes \xi^\perp) \otimes\largewedge^{k}V^*)\\
	& =  \Gamma(\mathbb P_+(V^*),\xi^{\otimes {n-k}}\otimes \largewedge^{n-k} \xi^\perp\otimes\largewedge^k V^*).
\end{align*}

Let $\eta\in \mathbb P_+(V^*)$ be a line such that $G v$ is unbounded and $0\notin\overline{G v}$ for any non-zero $v\in \eta$. Those assumptions hold by Lemmas \ref{lem:unbounded} and \ref{lem:few_zero_orbits} for a dense set of lines $\eta$.

Write $\tau|_\eta=c(\eta) v^{\otimes k}\otimes Q$, where $v\in\eta, v \neq 0$,  
$Q\in  \largewedge^{n-k}\eta^\perp\otimes\largewedge^k V^* $, $Q\neq 0$, and $c(\eta)\in\R$. Assume $c(\eta)\neq 0$.

Take a sequence $g_i\in G$ such that $g_iv\to \infty$. By the smoothness and $G$-invariance of $\tau$, it must hold that $g_iQ\to 0$.

By Lemma \ref{lem:zero_implies_infinite}, there is a sequence $h_i\in G$ such that $h_iQ\to\infty$ in 
$\largewedge^{n-k}V\otimes\largewedge^k V^*$. Again since $\tau$ is smooth, $h_iv\to 0$, a contradiction. It follows that $c(\eta)=0$ for a dense set of $\eta$, and so $\tau=0$. Thus $\phi=0$, in contradiction.

We conclude that there is an invariant metric for $G$, that is $A$ is of elliptic type.

Finally, if $\det A<0$, then we may apply the previous argument to $A^2$. Thus $A^2$ is elliptic type, and by Lemma \ref{lem:radical_elliptic} so is $A$.
\endproof

\begin{Remark}
We observe that in fact we prove something stronger. For contracting or expanding matrices $A$ (case 1), there are no KS-continuous invariant valuations in the sense of \cite[Definition 3.4]{bernig_faifman_opq}. 
\end{Remark}

\begin{Theorem}\label{thm:precompact}
	Let $G\subset \mathrm{GL}_n(\R)$ be any subgroup, and $\phi\in\Val_k^\infty(\R^n)^G$ non-zero with $1\leq k\leq n-1$.
	Then $G$ is precompact, that is there is a $G$-invariant metric on $\R^n$.
\end{Theorem}
In other words, the stabilizer of a nonzero $\phi\in\Val_k^\infty(\R^n)$ with $1\leq k\leq n-1$ in $\mathrm{GL}_n(\R)$ is compact.
\proof
We may replace $G$ by its closure in $\mathrm{GL}_n(\R)$, so that it is a Lie subgroup by Cartan's theorem. By Theorem \ref{thm:elliptic}, all $A\in G$ are elliptic type. It then follows by a theorem of Fresnel-van der Put \cite{fresnel_vanderput} that $G$ is compact. 
\endproof

As an immediate corollary we finally prove Theorem \ref{mthm:existence}.
\begin{proof}[Proof of Theorem \ref{mthm:existence}]
If a bi-invariant metric exists, its first intrinsic volume is a bi-invariant valuation as desired. 

In the other direction, let $\phi\not\in \Span\{\chi, \vol_G\}$ be a bi-invariant valuation. Substracting multiples of $\chi$ and $\vol_G$, we may suppose $\phi\in\mathcal W_k(G)\setminus \mathcal W_{k+1}(G)$, $1\leq k\leq n-1$. The principal symbol of $\phi$ at $e\in G$ is $\psi\in \Val_k^\infty(\mathfrak g)$ which by assumption is nonzero and $\Ad_G$-invariant. By Theorem \ref{thm:precompact}, there is an $\Ad_G$-invariant metric on $\mathfrak g$, giving rise to a bi-invariant metric on $G$.
\end{proof}

\subsection{Semisimple groups}

In \cite[Theorem B]{bernig_faifman_kotrbaty_part1} we have shown by abstract arguments that for compact connected Lie groups, the space of bi-invariant smooth valuations is infinite-dimensional except for the groups $S^3$ and $\mathrm{SO}(3)$. Here we show that for semisimple Lie groups with identity component other than $S^3, \mathrm{SO}(3)$, already the tube coefficients span an infinite-dimensional space.

\begin{Proposition}
	Let $G$ be a compact semisimple Lie group equipped with a bi-invariant riemannian metric. Then $\mathcal{TC}(G)$ is infinite dimensional unless $\g=\mathfrak {so}(3,\R)$.
\end{Proposition}

\proof
Consider the $n$-forms of $\mathcal{TC}(G)$. The subspace of bi-degree $(n-1, 1)$ consists, after the Hodge star, of the forms $[\xi^k, \gamma]$, $k\geq 0$. We will verify those span an infinite-dimensional space unless $\g=\mathfrak {so}(3,\R)$.

If the span is finite dimensional, there is an identity of the form
\begin{displaymath}
	\sum_{j=0}^N c_j \ad_\xi^j=0,
\end{displaymath} 
which holds on $\xi^\perp$, for all $\xi$. 

Let $u, v$ span an invariant subspace of $\ad_\xi$ with $\ad_\xi u=\lambda v$, $\ad_\xi v=-\lambda u$, $\lambda \neq 0$. Putting $\beta=-\lambda^2$ and plugging $u$, it must hold that
\begin{displaymath}
	\sum_{k=0}^{\lfloor N/2 \rfloor} c_{2k} \beta^k=0, \quad \sum_{k=0}^{\lfloor (N-1)/2\rfloor}c_{2k+1}\beta^k=0. 
\end{displaymath}
If we can find $\lfloor N/2\rfloor+1$ different such values $\beta$, Vandermonde's determinant then shows $c_j=0$ for all $j$.

Assume that a root $\alpha$ of $\g$ and $\xi \in S(\h)$ satisfy $\alpha(\xi) \neq 0$. Let $0 \neq u+iv, u,v \in \g$ be in the root space $\g_\alpha$. Noting that $\alpha(\xi) \in \R i$, the equation $\ad_\xi(u+iv)=\alpha(\xi) (u+iv)$ implies that $(u,v)$ is a pair as above with $\lambda=i \alpha(\xi)$ and $\beta=-|\alpha(\xi)|^2$.

If the Cartan algebra $\mathfrak h\subset \mathfrak g$ is of dimension at least $2$, we may take any root $\alpha$ of $\mathfrak g$ and find $\xi,\xi'\in S(\mathfrak h)$ such that $\alpha(\xi')=0\neq \alpha(\xi)$. Connecting $\xi,\xi'$ by a curve in $S(\mathfrak h)$, the continuity of $\alpha$ along this curve provides us with infinitely many different values $\beta$ as desired.

It follows that $\dim\mathfrak h=1$, which implies $\g=\mathfrak {so}(3,\R)$.
\endproof

%\bibliographystyle{plain}
%\bibliography{biblio}

\begin{thebibliography}{10}
	
	\bibitem{alekseev_ivanov}
	Ilya Alekseev and Sergei~O. Ivanov.
	\newblock Higher {J}acobi identities, 2016.
	\newblock Preprint arXiv:1604.05281.
	
	\bibitem{alesker_val_man2}
	Semyon Alesker.
	\newblock Theory of valuations on manifolds. {II}.
	\newblock {\em Adv. Math.}, 207(1):420--454, 2006.
	
	\bibitem{alesker_survey07}
	Semyon Alesker.
	\newblock Theory of valuations on manifolds: a survey.
	\newblock {\em Geom. Funct. Anal.}, 17(4):1321--1341, 2007.
	
	\bibitem{alesker_val_man4}
	Semyon Alesker.
	\newblock Theory of valuations on manifolds. {IV}. {N}ew properties of the
	multiplicative structure.
	\newblock In {\em Geometric aspects of functional analysis}, volume 1910 of
	{\em Lecture Notes in Math.}, pages 1--44. Springer, Berlin, 2007.
	
	\bibitem{alesker_fourier}
	Semyon Alesker.
	\newblock A {F}ourier type transform on translation invariant valuations on
	convex sets.
	\newblock {\em Israel J. Math.}, 181:189--294, 2011.
	
	\bibitem{alesker_bernig_convolution}
	Semyon Alesker and Andreas Bernig.
	\newblock Convolution of valuations on manifolds.
	\newblock {\em J. Differential Geom.}, 107(2):203--240, 2017.
	
	\bibitem{alesker_val_man3}
	Semyon Alesker and Joseph H.~G. Fu.
	\newblock Theory of valuations on manifolds. {III}. {M}ultiplicative structure
	in the general case.
	\newblock {\em Trans. Amer. Math. Soc.}, 360(4):1951--1981, 2008.
	
	\bibitem{allendoerfer}
	Carl~B. Allendoerfer.
	\newblock Steiner's formulae on a general {$S^{n+1}$}.
	\newblock {\em Bull. Amer. Math. Soc.}, 54:128--135, 1948.
	
	\bibitem{bernig_faifman_opq}
	Andreas Bernig and Dmitry Faifman.
	\newblock Valuation theory of indefinite orthogonal groups.
	\newblock {\em J. Funct. Anal.}, 273(6):2167--2247, 2017.
	
	\bibitem{bernig_faifman_kotrbaty_part1}
	Andreas Bernig, Dmitry Faifman, and Jan Kotrbatý.
	\newblock Invariant valuations on {L}ie groups, 2025.
	\newblock Preprint arXiv:2512.01004.
	
	\bibitem{bernig_faifman_solanes_pseudo}
	Andreas Bernig, Dmitry Faifman, and Gil Solanes.
	\newblock Curvature measures of pseudo-{R}iemannian manifolds.
	\newblock {\em J. Reine Angew. Math.}, (788):77--127, 2022.
	
	\bibitem{bernig_fu_solanes_wannerer}
	Andreas Bernig, Joseph Fu, Gil Solanes, and Thomas Wannerer.
	\newblock {The Weyl tube formula for K\"ahler manifolds}.
	\newblock To appear in Geometry and Topology.
	
	\bibitem{bernig_fu_solanes}
	Andreas Bernig, Joseph H.~G. Fu, and Gil Solanes.
	\newblock Integral geometry of complex space forms.
	\newblock {\em Geom. Funct. Anal.}, 24(2):403--492, 2014.
	
	\bibitem{graded_schouten_nijenhuis}
	J.~A. de~Azc\'{a}rraga, J.~M. Izquierdo, A.~M. Perelomov, and J.~C.
	P\'{e}rez-Bueno.
	\newblock The {$Z_2$}-graded {S}chouten-{N}ijenhuis bracket and generalized
	super-{P}oisson structures.
	\newblock {\em J. Math. Phys.}, 38(7):3735--3749, 1997.
	
	\bibitem{duistermaat_kolk}
	J.~J. Duistermaat and J.~A.~C. Kolk.
	\newblock {\em Lie groups}.
	\newblock Universitext. Springer-Verlag, Berlin, 2000.
	
	\bibitem{faifman_heisenberg}
	Dmitry Faifman.
	\newblock Contact integral geometry and the {H}eisenberg algebra.
	\newblock {\em Geom. Topol.}, 23(6):3041--3110, 2019.
	
	\bibitem{faifman_hofstaetter}
	Dmitry Faifman and Georg~C. Hofst\"atter.
	\newblock Convex valuations from {W}hitney to {N}ash.
	\newblock {\em Duke Math. J.}, 174(14):3063--3133, 2025.
	
	\bibitem{faifman_wannerer_fourier}
	Dmitry Faifman and Thomas Wannerer.
	\newblock The {F}ourier transform on valuations is the {F}ourier transform.
	\newblock {\em J. Funct. Anal.}, 288(3):Paper No. 110741, 42, 2025.
	
	\bibitem{fresnel_vanderput}
	Jean Fresnel and Marius van~der Put.
	\newblock Compact subgroups of {{\(\text{GL}_n(\mathbb C)\)}}.
	\newblock {\em Rend. Semin. Mat. Univ. Padova}, 116:187--192, 2006.
	
	\bibitem{fu_wannerer}
	Joseph H.~G. Fu and Thomas Wannerer.
	\newblock Riemannian curvature measures.
	\newblock {\em Geom. Funct. Anal.}, 29(2):343--381, 2019.
	
	\bibitem{gray_vanhecke}
	A.~Gray and L.~Vanhecke.
	\newblock The volumes of tubes in a {R}iemannian manifold.
	\newblock {\em Rend. Sem. Mat. Univ. Politec. Torino}, 39(3):1--50 (1983),
	1981.
	
	\bibitem{gray_submanifolds}
	Alfred Gray.
	\newblock Volumes of tubes about complex submanifolds of complex projective
	space.
	\newblock {\em Trans. Amer. Math. Soc.}, 291(2):437--449, 1985.
	
	\bibitem{gray_tubes}
	Alfred Gray.
	\newblock {\em Tubes}, volume 221 of {\em Progress in Mathematics}.
	\newblock Birkh\"{a}user Verlag, Basel, second edition, 2004.
	\newblock With a preface by Vicente Miquel.
	
	\bibitem{horn_johnson}
	Roger~A. Horn and Charles~R. Johnson.
	\newblock {\em Matrix analysis}.
	\newblock Cambridge University Press, Cambridge, second edition, 2013.
	
	\bibitem{hotelling}
	Harold Hotelling.
	\newblock Tubes and {S}pheres in n-{S}paces, and a {C}lass of {S}tatistical
	{P}roblems.
	\newblock {\em Amer. J. Math.}, 61(2):440--460, 1939.
	
	\bibitem{klain00}
	Daniel~A. Klain.
	\newblock Even valuations on convex bodies.
	\newblock {\em Trans. Amer. Math. Soc.}, 352(1):71--93, 2000.
	
	\bibitem{manetti}
	Marco Manetti.
	\newblock {\em Lie methods in deformation theory}.
	\newblock Springer Monographs in Mathematics. Springer, Singapore, [2022]
	\copyright 2022.
	
	\bibitem{santalo}
	Luis~A. Santal\'o.
	\newblock {\em Integral geometry and geometric probability}, volume Vol. 1 of
	{\em Encyclopedia of Mathematics and its Applications}.
	\newblock Addison-Wesley Publishing Co., Reading, Mass.-London-Amsterdam, 1976.
	\newblock With a foreword by Mark Kac.
	
	\bibitem{schneider96}
	Rolf Schneider.
	\newblock Simple valuations on convex bodies.
	\newblock {\em Mathematika}, 43(1):32--39, 1996.
	
	\bibitem{solanes_trillo}
	Gil Solanes and Juan~Andr{\'e}s Trillo.
	\newblock Tube formulas for valuations in complex space forms.
	\newblock {\em Math. Ann.}, 391(1):881--913, 2025.
	
	\bibitem{solanes_wannerer}
	Gil Solanes and Thomas Wannerer.
	\newblock Integral geometry of exceptional spheres.
	\newblock {\em J. Differential Geom.}, 117(1):137--191, 2021.
	
	\bibitem{weyl_tubes}
	Hermann Weyl.
	\newblock On the {V}olume of {T}ubes.
	\newblock {\em Amer. J. Math.}, 61(2):461--472, 1939.
	
\end{thebibliography}
\end{document}